\documentclass[12pt,a4paper, leqno]{article}
\usepackage{fullpage}
\usepackage{amsmath}
\usepackage{amssymb}
\usepackage{amsthm}
\usepackage{mathrsfs}
\usepackage[dvips]{graphicx}
\usepackage{psfrag}
\usepackage[hang,small,bf]{caption}
\newtheorem{theorem}{Theorem}[section]

\newtheorem{proposition}[theorem]{Proposition}
\newtheorem{lemma}[theorem]{Lemma}

\newtheorem{remark}[theorem]{Remark}

\numberwithin{equation}{section}

\begin{document}

\title{Random walks on discrete cylinders with large bases and random interlacements}
\author{David Windisch}
\date{}
\maketitle
%\begin{center}
% {\bf Preliminary draft}
%\end{center}

\begin{abstract}
Following the recent work of Sznitman \cite{S08}, we investigate the microscopic picture induced by a random walk trajectory on a cylinder of the form $G_N \times {\mathbb Z}$, where $G_N$ is a large finite connected weighted graph, and relate it to the model of random interlacements on infinite transient weighted graphs. Under suitable assumptions, the set of points not visited by the random walk until a time of order $|G_N|^2$ in a neighborhood of a point with $\mathbb Z$-component of order $|G_N|$ converges in distribution to the law of the vacant set of a random interlacement on a certain limit model describing the structure of the graph in the neighborhood of the point. The level of the random interlacement depends on the local time of a Brownian motion. The result also describes the limit behavior of the joint distribution of the local pictures in the neighborhood of several distant points with possibly different limit models. As examples of $G_N$, we treat the $d$-dimensional box of side length $N$, the Sierpinski graph of depth $N$ and the $d$-ary tree of depth $N$, where $d \geq 2$.
\end{abstract}

\section{Introduction}

In recent works, Sznitman introduces the model of random interlacements on ${\mathbb Z}^{d+1}$, $d \geq 2$ (cf.~\cite{S07}, \cite{SS08}), and in \cite{S08} explores its relation with the microscopic structure left by simple random walk on an infinite discrete cylinder $({\mathbb Z}/N{\mathbb Z})^{d} \times {\mathbb Z}$ by times of order $N^{2d}$. The present work extends this relation to random walk on $G_N \times {\mathbb Z}$ running for a time of order $|G_N|^2$, where the bases $G_N$ are given by finite weighted graphs satisfying suitable assumptions, as proposed by Sznitman in \cite{S08}. The limit models that appear in this relation are random interlacements on transient weighted graphs describing the structure of $G_N$ in a microscopic neighborhood. Random interlacements on such graphs have been constructed in \cite{T08}. Among the examples of $G_N$ to which our result applies are boxes of side-length $N$, discrete Sierpinski graphs of depth $N$ and $d$-ary trees of depth $N$.

\vspace{2mm}

We proceed with a more precise description of the setup. A weighted graph $({\mathcal G}, {\mathcal E}, w_{.,.})$ consists of a countable set ${\mathcal G}$ of vertices, a set $\mathcal E$ of unordered pairs of distinct vertices, called edges, and a weight $w_{.,.}$, which is a symmetric function associating to every ordered pair $({\mathsf y}, {\mathsf y}')$ of vertices a non-negative number $w_{{\mathsf y},{\mathsf y}'}= w_{{\mathsf y}',{\mathsf y}}$, non-zero if and only if $\{{\mathsf y}, {\mathsf y}' \} \in {\mathcal E}$. Whenever $\{{\mathsf y}, {\mathsf y}' \} \in {\mathcal E}$, the vertices $\mathsf y$ and ${\mathsf y}'$ are called neighbors. A path of length $n$ in $\mathcal G$ is a sequence of vertices $({\mathsf y}_0, \ldots, {\mathsf y}_n)$ such that ${\mathsf y}_{i-1}$ and ${\mathsf y}_i$ are neighbors for $1 \leq i \leq n$. The distance $d({\mathsf y},{\mathsf y}')$ between vertices $\mathsf y$ and ${\mathsf y}'$ is defined as the length of the shortest path starting at $\mathsf y$ and ending at ${\mathsf y}'$ and $B({\mathsf y},r)$ denotes the closed ball centered at $\mathsf y$ of radius $r \geq 0$. We generally omit $\mathcal E$ and $w_{.,.}$ from the notation and simply refer to $\mathcal G$ as a weighted graph. A standing assumption is that $\mathcal G$ is connected. The random walk on $\mathcal G$ is defined as the irreducible reversible Markov chain on $\mathcal G$ with transition probabilities $p^{\mathcal G}({\mathsf y},{\mathsf y}') = w_{{\mathsf y},{\mathsf y}'}/w_{\mathsf y}$ for $\mathsf y$ and ${\mathsf y}'$ in $\mathcal G$, where $w_{\mathsf y} = \sum_{{\mathsf y}' \in {\mathcal G}} w_{{\mathsf y},{\mathsf y}'}$. Then $w_{{\mathsf y}} p^{\mathcal G}({\mathsf y},{\mathsf y}') = w_{{\mathsf y}'} p^{\mathcal G}({\mathsf y}',{\mathsf y})$, so a reversible measure for the random walk is given by $w({\mathcal A})= \sum_{{\mathsf y} \in {\mathcal A}} w_{\mathsf y}$ for ${\mathcal A} \subseteq {\mathcal G}$. A bijection $\phi$ between subsets $\mathcal B$ and ${\mathcal B}^*$ of weighted graphs $\mathcal G$ and ${\mathcal G}^*$ is called an isomorphism between ${\mathcal B}$ and ${\mathcal B}^*$ if $\phi$ preserves the weights, i.e.~if $w_{\phi({\mathsf y}), \phi({\mathsf y}')} =  w_{{\mathsf y},{\mathsf y}'}$ for all $\mathsf y$, ${\mathsf y}' \in \mathcal B$.

\vspace{2mm}

This setup allows the definition of a random walk $(X_n)_{n \geq 0}$ on the discrete cylinder
\begin{align}
G_N \times {\mathbb Z}, \label{def:cyl}
\end{align}
where $G_N$, $N \geq 1$, is a sequence of finite connected weighted graphs with weights $(w_{y,y'})_{y,y' \in G_N}$ and $G_N \times {\mathbb Z}$ is equipped with the weights
\begin{align}
w_{x,x'} = w_{y,y'} \mathbf{1}_{\{z=z'\}} + \frac{1}{2} \mathbf{1}_{\{y=y', |z-z'|=1\}}, \textrm{ for } x=(y,z) \textrm{, } x'=(y',z') \textrm{ in } G_N \times {\mathbb Z}. \label{def:cylweights}
\end{align}
We will mainly consider situations where all edges of the graphs have equal weight $1/2$.
The random walk $X$ starts from $x \in G_N \times {\mathbb Z}$ or from the uniform distribution on $G_N \times \{0\}$ under suitable probabilities $P_x$ and $P$ defined in (\ref{def:pr1}) and (\ref{def:pr2}) below. 
We consider $M \geq 1$ and sequences of points $x_{m,N}=(y_{m,N},z_{m,N})$, $1 \leq m \leq M$, in $G_N \times {\mathbb Z}$ with mutual distance tending to infinity. We assume that the neighborhoods around any vertex $y_{m,N}$ look like balls in a fixed infinite graph ${\mathbb G}_m$, in the sense that 
\begin{align}
&\textrm{we choose an } {\mathsf r}_N \to \infty, \textrm{ such that there are isomorphisms } \phi_{m,N} \textrm{ from }  \label{iso0}\\
&B(y_{m,N},{\mathsf r}_N) \textrm{ to } B(o_m,{\mathsf r}_N) \subset {\mathbb G}_m \textrm{ with } \phi_{m,N}(y_{m,N}) = o_m \textrm{ for all $N$.} \nonumber
\end{align}
%\begin{align}
%&\textrm{$y_{m,N} \in B_{m,N} \subseteq G_N$, and for a fixed infinite graph ${\mathbb G}_m$, there is an isomorphism} \\
%&\textrm{$\phi_{m,N}$ from $B_{m,N}$ to a subset ${\mathbb B}_{m,N} \subset {\mathbb G}_m$, such that ${\mathbb G}_m = \cup_N \cap_{n \geq N} {\mathbb %B}_{m,n}$.} \nonumber
%\end{align}
The points not visited by the random walk in the neighborhood of $x_{m,N}$ until time $t \geq 0$ induce a random configuration of points in the limit model ${\mathbb G}_m \times {\mathbb Z}$, called the vacant configuration in the neighborhood of $x_{m,N}$, which is defined as the $\{0,1\}^{{\mathbb G}_m \times {\mathbb Z}}$-valued random variable
\begin{align}
\omega^{m,N}_t ({\mathtt x}) = \left\{ \begin{array}{cl} \mathbf{1}\{X_n \neq \Phi_{m,N}^{-1}({\mathtt x}), \textrm{ for } 0 \leq n \leq t\}, & \textrm{if } {\mathtt x} \in B(o_m,{\mathsf r}_N) \times {\mathbb Z}, \\
0, & \textrm{otherwise, for $t \geq 0$,}
                         \end{array} \right. \label{def:pic}
\end{align}
where the isomorphism $\Phi_{m,N}$ is defined by $\Phi_{m,N}(y,z) = (\phi_{m,N}(y),z-z_{m,N})$ for $(y,z)$ in $B(y_{m,N},{\mathsf r}_N) \times {\mathbb Z}$. 

\vspace{2mm}

Random interlacements on ${\mathbb G}_m \times {\mathbb Z}$ enter the asymptotic behavior of the distribution of the local pictures $\omega^{m,N}$. For the construction of random interlacements on transient weighted graphs we refer to \cite{T08}. For our purpose it suffices to know that for a weighted graph ${\mathbb G}_m \times {\mathbb Z}$ with weights defined such that the random walk on it is transient, the law ${\mathbb Q}^{{\mathbb G}_m \times {\mathbb Z}}_u$ on $\{0,1\}^{{\mathbb G}_m \times {\mathbb Z}}$ of the indicator function of the vacant set of the random interlacement at level $u \geq 0$ on ${\mathbb G}_m \times {\mathbb Z}$ is characterized by, cf.~equation (1.1) of \cite{T08},
%Under suitable assumptions, the large $N$ limit of the distribution of the vacant configurations $\omega^m_.$ is given by the distribution of the vacant set of the random interlacement. In the terminology of \cite{T08}, the interlacement at level $u \geq 0$ on a connected transient weighted graph $\mathcal G$ is the trace left on $\mathcal G$ by a cloud of paths constituting a Poisson point process on the space of doubly infinite trajectories modulo time-shift, tending to infinity at positive and negative infinite times. The parameter $u$ is a multiplicative factor of the intensity measure of this point process. The complement of the random interlacement is the so-called vacant set at level $u$. Assuming that the weighted graph ${\mathbb G}_m \times {\mathbb Z}$ with weights defined as in (\ref{def:cylweights}) is connected and the random walk on it is transient, the law ${\mathbb Q}^{{\mathbb G}_m \times {\mathbb Z}}_u$ on $\{0,1\}^{{\mathbb G}_m \times {\mathbb Z}}$ of the indicator function of the vacant set at level $u \geq 0$ on ${\mathbb G}_m \times {\mathbb Z}$ is by equation (1.1) of \cite{T08} characterised by the property
\begin{align}
&{\mathbb Q}^{{\mathbb G}_m \times {\mathbb Z}}_u [\omega({\mathtt x}) = 1, \textrm{ for all } {\mathtt x} \in {\mathbb V}] = \exp \{-u \textup{ cap}^m({\mathbb V})\}, \label{def:int}\\
&\textrm{for all finite subsets } {\mathbb V} \textrm{ of } {\mathbb G}_m \times {\mathbb Z},  \nonumber
\end{align}
where  $\omega({\mathtt x})$, ${\mathtt x} \in {\mathbb G}_m \times {\mathbb Z}$, are the canonical coordinates on $\{0,1\}^{{\mathbb G}_m \times {\mathbb Z}}$, and $\textup{cap}^m({\mathbb V})$ the capacity of $\mathbb V$ as defined in (\ref{def:cap}) below. 

\vspace{2mm}

The main result of the present work requires the assumptions \ref{hyp:w}-\ref{hyp:k} on the graph $G_N$, which we discuss below. In order to state the result, we have yet to introduce the local time of the $\mathbb Z$-projection $\pi_{\mathbb Z}(X)$ of $X$, defined as
\begin{align}
L^z_n = \sum_{l=0}^{n-1} \mathbf{1}_{\{\pi_{\mathbb Z}(X_l) = z\}}, \textrm{ for } z \in {\mathbb Z}, n \geq 1, \label{def:locpr}
\end{align}
as well as the canonical Wiener measure $W$ and a jointly continuous version $L(v,t)$, $v \in {\mathbb R}$, $t \geq 0$, of the local time of the canonical Brownian motion. The main result asserts that under suitable hypotheses the joint distribution of the vacant configurations in the neighborhoods of $x_{1,N}, \ldots, x_{M,N}$ and the scaled local times of the $\mathbb Z$-projections of these points at a time of order $|G_N|^2$ converges as $N$ tends to infinity to the joint distribution of the vacant sets of random interlacements on ${\mathbb G}_m \times {\mathbb Z}$ and local times of a Brownian motion. The levels of the random interlacements depend on the local times, and conditionally on the local times, the random interlacements are independent. Here is the precise statement:

\begin{theorem}\label{thm:d}
Assume \textup{\ref{hyp:w}-\ref{hyp:k}} (see below (\ref{def:spec})), as well as
\begin{align}
\frac{w(G_N)}{|G_N|} \stackrel{N \to \infty}{\longrightarrow} \beta, \textrm{ for some } \beta >0,  \label{hyp:d}
\end{align}
and for all $1 \leq m \leq M$,
$$\frac{z_{m,N}}{|G_N|} \stackrel{N \to \infty}{\longrightarrow} v_m, \textrm{ for some } v_m \in {\mathbb R},$$
which is in fact assumption \textup{\ref{hyp:zconv}}, see below.
Then the graphs ${\mathbb G}_m \times {\mathbb Z}$ are transient and as $N$ tends to infinity, the $\prod_{m=1}^M \{0,1\}^{{\mathbb G}_m} \times {\mathbb R}_+^M$-valued random variables
\begin{align*}
\Bigl(\omega^{1,N}_{\alpha |G_N|^2}, \ldots, \omega^{M,N}_{\alpha |G_N|^2}, \frac{L^{z_{1,N}}_{\alpha |G_N|^2}}{|G_N|}, \ldots, \frac{L^{z_{M,N}}_{\alpha |G_N|^2}}{|G_N|} \Bigr),  \quad \alpha >0, \, N \geq 1,
\end{align*}
defined by (\ref{def:pic}) and (\ref{def:locpr}), with ${\mathsf r}_N$ and $\phi_{m,N}$ chosen in (\ref{eq:rs}) and (\ref{def:iso}), converge in joint distribution under $P$ to the law of the random vector
$(\omega_1, \ldots, \omega_M, U_1, \ldots, U_M)$
with the following distribution: the variables $(U_m)_{m=1}^M$ are distributed as $((1+\beta) L(v_m,\alpha/(1+\beta)))_{m=1}^M$ under $W$, and conditionally on $(U_m)_{m=1}^M$, the variables $(\omega_m)_{m=1}^M$ have joint distribution $ \prod_{1 \leq m \leq M} {\mathbb Q}_{U_m/(1+\beta)}^{{\mathbb G}_m \times {\mathbb Z}}$.
\end{theorem}

%\begin{remark}
%\textup{For future reference, we write the conclusion of Theorem~\ref{thm:d} in concise form as 
%\begin{align}
%\Bigl( \omega^m_{\alpha |G|^2}, \frac{L^{z_m}_{\alpha |G|^2}}{|G|} \Bigr) \rightharpoonup \Bigl( \otimes_{m} {\mathbb Q}^{{\mathbb G}_m}_{L(v_m %\alpha/(1+\beta))}, (1+\beta) L \Bigl(v_m, \frac{\alpha}{1+\beta} \Bigr)  \Bigr), \textrm{ for } \alpha>0. \label{thm:c}
%\end{align}
%}
%\end{remark}

\begin{remark}
\textup{
Sznitman proves a result analogous to Theorem~\ref{thm:d} in \cite{S08}, Theorem~0.1, for $G_N$ given by $({\mathbb Z}/N{\mathbb Z})^d$ and ${\mathbb G}_m = {\mathbb Z}^d$ for $1 \leq m \leq M$. This result is covered by Theorem~\ref{thm:d} by choosing, for any $y$ and $y'$ in $({\mathbb Z}/N{\mathbb Z})^d$, $w_{y,y'} = 1/2$ if $y$ and $y'$ are at Euclidean distance $1$ and $w_{y,y'}=0$ otherwise. Then the random walk $X$ on $({\mathbb Z}/N{\mathbb Z})^d \times {\mathbb Z}$ with weights as in (\ref{def:cylweights}) is precisely the simple random walk considered in \cite{S08}. We then have $\beta=d$ in (\ref{hyp:d}) and recover the result of \cite{S08}, noting that the factor $1/(1+d)$ appearing in the law of the vacant set cancels with the factor $w_{\mathtt x} = d+1$ in our definition of the capacity (cf.~(\ref{def:cap})), different from the one used in \cite{S08} (cf.~(1.7) in \cite{S08}).
}
\end{remark}

We now make some comments on the proof of Theorem~\ref{thm:d}. In order to extract the relevant information from the behavior of the $\mathbb Z$-component of the random walk, we follow the strategy in \cite{S08} and use a suitable version of the partially inhomogeneous grids on $\mathbb Z$ introduced there. Results from \cite{S08} show that the total time elapsed and the scaled local time of a simple random walk on $\mathbb Z$ can be approximated by the random walk restricted to certain stopping times related to these grids. The difficulty that arises in the application of these results in our setup is that unlike in \cite{S08}, the $\mathbb Z$-projection of our random walk $X$ is not a Markov process. Indeed, the $\mathbb Z$-projection is delayed at each step for an amount of time that depends on the current position of the $G_N$-component. In order to overcome this difficulty, we decouple the $\mathbb Z$-component of the random walk from the $G_N$-component by introducing a continuous-time process ${\mathsf X}=({\mathsf Y}, {\mathsf Z})$, such that the $G_N$- and $\mathbb Z$-components $\mathsf Y$ and $\mathsf Z$ are independent and such that the discrete skeleton of $\mathsf X$ is the random walk $X$ on $G_N \times {\mathbb Z}$. It is not trivial to regain information about the random walk $X$ after having switched to continuous time, because the waiting times of the process $\mathsf X$ depend on the steps of the discrete skeleton $X$ and are in particular not iid. We therefore prove in Theorem~\ref{thm:c} the continuous-time version of Theorem~\ref{thm:d} first, essentially by using an abstraction of the arguments in \cite{S08} and making frequent use of the independence of the $G_N$- and $\mathbb Z$-components of $\mathsf X$, and defer the task of transferring the result to discrete time to later.

\vspace{2mm}

Let us make a few more comments on the partially inhomogeneous grids just mentioned. Every point of these grids is a center of two concentric intervals $I \subset {\tilde I}$ with diameters of order $d_N$ and $h_N \gg d_N$, where $h_N$ is also the order of the mesh size of the grids throughout $\mathbb Z$. The definition of the grids ensures that all points $z_{m,N}$ are covered by the smaller intervals, hence the partial inhomogeneity. We then consider the successive returns to the intervals $I$ and departures from $\tilde I$ of the discrete skeleton $Z$ of $\mathsf Z$. According to a result from \cite{S08} (see Proposition~\ref{pr:sp} below) and Lemma~\ref{lem:loc}, these excursions contain all the relevant information needed to approximate the total time elapsed and to relate the scaled local time ${\mathsf L}^{z_{m,N}}_{\alpha |G_N|^2} / |G_N|$ of $\mathsf Z$ (see~(\ref{def:loct})) to the number of returns of $Z$ to the box containing $z_{m,N}$. For these estimates to apply, the mesh size $h_N$ of the grids has to be smaller than the square root of the total number of steps of the walk, i.e. less than $|G_N|$. At the same time, we shall need $h_N$ to be larger than the square root of the relaxation time $\lambda_N^{-1}$ of $G_N$, so that the $G_N$-component $\mathsf Y$ approaches its stationary, i.e.~uniform, distribution between different excursions. This motivates the condition \ref{hyp:mix}, see below (\ref{def:spec}), on the spectral gap $\lambda_N$ of $G_N$.

\vspace{2mm}

Once the partially inhomogeneous grids are introduced, the law ${\mathbb Q}^{{\mathbb G}_m \times {\mathbb Z}}_.$ of the vacant set appears as follows: For concentric intervals $I \subset {\tilde I}$, $z \in \partial (I^c)$ and $z' \in \partial {\tilde I}$ we define the probability $P_{z,z'}$ as the law of the finite-time random walk trajectory started at a uniformly distributed point in $G_N \times \{z\}$ and conditioned to exit $G_N \times {\tilde I}$ through $G_N \times \{z'\}$ at its final step. We have mentioned that the distribution of the $G_N$-component of $\mathsf X$ approaches the uniform distribution between different excursions from $G_N \times I$ to $(G_N \times {\tilde I})^c$. It follows that the law of these successive excursions of $\mathsf X$ under $P$, conditioned on the points $z$ and $z'$ of entrance and departure of the $\mathbb Z$-component, can be approximated by a product of the laws $P_{z,z'}$. This is shown in Lemma~\ref{lem:tv}. A crucial element in the proof of the continuous-time Theorem~\ref{thm:c} is the investigation of the $P_{z,z'}$-probability that a set $V$ in the neighborhood of a point $x_{m,N}$ in $G_N \times I$ is not left vacant by one excursion. We find that up to a factor tending to $1$ as $N$ tends to infinity, this probability is equal to $\textup{cap}^m(\Phi_{m,N}(V)) h_N/|G_N|$. With the relation between the number of such excursions taking place up to time $\alpha |G_N|^2$ and the scaled local time ${\mathsf L}^{z_{m,N}}_{\alpha |G_N|^2} / |G_N|$ from Proposition~\ref{pr:sp} and Lemma~\ref{lem:loc}, the law ${\mathbb Q}^{{\mathbb G}_m \times {\mathbb Z}}_.$, see (\ref{def:int}), appears as the limiting distribution of the vacant configuration in the neighborhood of $x_{m,N}$.

\vspace{2mm}

Let us describe the derivation of the asymptotic behavior of the $P_{z,z'}$-probability just mentioned in a little more detail. As in \cite{S08}, a key step in the proof is to show that the probability that the random walk escapes from a vertex in a set $V \subset G_N \times I$ in the vicinity of $x_{m,N}$ to the complement of $G_N \times {\tilde I}$ before hitting the set $V$ converges to the corresponding escape probability to infinity for the set $\Phi_{m,N}(V)$ in the limit model ${\mathbb G}_m \times {\mathbb Z}$. This is where the required capacity appears. The assumption \ref{hyp:iso1} that (potentially small) neighborhoods $B(y_{m,N},r_N)$ of the points $y_{m,N}$ are isomorphic to neighborhoods in ${\mathbb G}_m$ is necessary but not sufficient for this purpose. We still need to ensure that the probability that the random walk returns from the boundary of $B(x_{m,N},r_N)$ to the vicinity of $x_{m,N}$ before exiting $G_N \times {\tilde I}$ decays. This is the reason why we assume the existence of larger neighborhoods $C_{m,N}$ containing $B(y_{m,N},r_N)$ in \ref{hyp:inc}. 
These neighborhoods $C_{m,N}$ are assumed to be either identical or disjoint for points with similarly-behaved $\mathbb Z$-components in \ref{hyp:disj}. Crucially, we assume in \ref{hyp:iso2} that the sets $C_{m,N}$ are themselves isomorphic to neighborhoods in infinite graphs ${\hat {\mathbb G}}_m$ that are sufficiently close to being transient, as is formalized by \ref{hyp:trans}. We additionally assume in \ref{hyp:k} that $X$ started from any point in the boundary of $C_{m,N} \times {\mathbb Z}$ typically does not reach the vicinity of $x_{m,N}$ until time $\lambda_N^{-1} |G_N|^\epsilon$, i.e.~until well after the relaxation time of $Y$. 
These assumptions ensure that the random walk, when started from the boundary of $B(x_{m,N},r_N)$, is unlikely to return to a point close to $x_{m,N}$ before exiting $G_N \times {\tilde I}$. 
For this last argument, we need the mesh size $h_N$ of the grids to be smaller than $(\lambda_N^{-1} |G_N|^\epsilon)^{1/2}$, so that $h_N$ can be only slightly larger than the $\lambda_N^{-1/2}$ required for the homogenization of the $G_N$-component.

\vspace{2mm}

In order to deduce Theorem~\ref{thm:d} from the continuous-time result, we need an estimate on the long term-behavior of the process of jump times of $\mathsf X$ and a comparison of the local time of $\mathsf X$ and the local time of the discrete skeleton $X$. This requires a kind of ergodic theorem, with the feature that both time and the process itself depend on $N$. To show the required estimates, we use estimates on the covariance between sufficiently distant increments of the jump process that follow from bounds on the spectral gap of $G_N$. With the assumption (\ref{hyp:d}), we find that the total number of jumps made by $\mathsf X$ up to a time of order $|G_N|^2$ is essentially proportional to the limit of the average weight $(1+\beta)$ per vertex in $G_N \times {\mathbb Z}$, see Lemma~\ref{lem:erg}. In this context, the hypothesis \ref{hyp:w} of uniform boundedness of the vertex-weights of $G_N$ plays an important role for stochastic domination of jump processes by homogeneous Poisson processes.

\vspace{2mm}

The article is organized as follows:
In Section \ref{sec:not}, we introduce notation and state the hypotheses \ref{hyp:w}-\ref{hyp:k} for Theorem~\ref{thm:d}.
In Section \ref{sec:grid}, we introduce the partially inhomogeneous grids with the relevant results described above.
Section \ref{sec:ind} shows that the dependence between the $G_N$-components of different excursions related to these grids is negligible. With these ingredients at hand, we can prove the continuous-time version of Theorem~\ref{thm:d} in Section \ref{sec:c}.
The crucial estimates on the jump process needed to transfer the result to discrete time are derived in Section \ref{sec:jump}. With the help of these estimates, we finally deduce Theorem~\ref{thm:d} in Section \ref{sec:proof}. Section \ref{sec:ex} is devoted to applications of Theorem~\ref{thm:d} to three concrete examples of $G_N$.

\vspace{2mm}

Throughout this article, $c$ and $c'$ denote positive constants changing from place to place. Numbered constants $c_0, c_1, \ldots$ are fixed and refer to their first appearance in the text. Dependence of constants on parameters appears in the notation.

\paragraph{Acknowledgments.} The author is grateful to Alain-Sol Sznitman for proposing the problem and for helpful advice.

\section{Notation and hypotheses} \label{sec:not}

The purpose of this section is to introduce some useful notation and state the hypotheses \ref{hyp:w}-\ref{hyp:k} made in Theorem~\ref{thm:d}.

\vspace{2mm}

Given any sequence $a_N$ of real numbers, $o(a_N)$ denotes a sequence $b_N$ with the property $b_N/a_N \to 0$ as $N \to \infty$. The notation $a \wedge b$ and $a \vee b$ is used to denote the respective minimum and maximum of the numbers $a$ and $b$. For any set $A$, we denote by $|A|$ the number of its elements. For a set ${\mathcal B}$ of vertices in a graph $\mathcal G$, we denote by $\partial {\mathcal B}$ the boundary of $\mathcal B$, defined as the set of vertices in the complement of $\mathcal B$ with at least one neighbor in $\mathcal B$ and define the closure of $\mathcal B$ as ${\bar {\mathcal B}} = {\mathcal B} \cup \partial {\mathcal B}$.

\vspace{2mm}

We now construct the relevant probabilities for our study. For any weighted graph $\mathcal G$, the path space ${\mathcal P}({\mathcal G})$ is defined as the set of right-continuous functions from $[0,\infty)$ to $\mathcal G$ with infinitely many discontinuities and finitely many discontinuities on compact intervals, endowed with the canonical $\sigma$-algebra generated by the coordinate projections. We let $({\mathsf Y}_t)_{t \geq 0}$ stand for the canonical coordinate process on ${\mathcal P}({\mathcal G})$. We consider the probability measures $P^{\mathcal G}_{\mathsf y}$ on ${\mathcal P}({\mathcal G})$ such that $\mathsf Y$ is distributed as a continuous-time Markov chain on $\mathcal G$ starting from ${\mathsf y} \in {\mathcal G}$ with transition rates given by the weights $w_{{\mathsf y}, {\mathsf y}'}$. Then the discrete skeleton $(Y_n)_{n \geq 0}$, defined by $Y_n = {\mathsf Y}_{\sigma^{\mathsf Y}_n}$, with $(\sigma^{\mathsf Y}_n)_{n \geq 0}$ the successive times of discontinuity of $\mathsf Y$ (where $\sigma^{\mathsf Y}_0 = 0$), is a random walk on $\mathcal G$ starting from $\mathsf y$ with transition probabilities $p^{\mathcal G}({\mathsf y},{\mathsf y}') = w_{{\mathsf y},{\mathsf y}'}/w_{\mathsf y}$. The discrete- and continuous-time transition probabilities for general times $n$ and $t$ are denoted by $p^{\mathcal G}_n({\mathsf y},{\mathsf y}') = P^{\mathcal G}_{\mathsf y} [Y_n = {\mathsf y}']$ and  $q^{\mathcal G}_t({\mathsf y},{\mathsf y}') = P^{\mathcal G}_{\mathsf y} [{\mathsf Y}_t = {\mathsf y}'].$ The jump process $(\eta^{\mathsf Y}_t)_{t \geq 0}$ of $\mathsf Y$ is denoted by $\eta^{\mathsf Y}_t = \sup \{ n \geq 0: \sigma^{\mathsf Y}_n \leq t\}$, so that ${\mathsf Y}_t = Y_{\eta^{\mathsf Y}_t}$, $t \geq 0$. 

\vspace{2mm}

Next, we adapt the notation of the last paragraph to the graphs we consider. Let ${\mathcal G}$ be any of the graphs $\mathbb Z = \{z,z', \ldots\}$ with weight $1/2$ attached to any edge, $G_N = \{y,y',\ldots\}$, ${\mathbb G}_m = \{{\mathtt y}, {\mathtt y}', \ldots \}$ or ${\hat {\mathbb G}}_m = \{{\mathtt y}, {\mathtt y}', \ldots \}$, where $G_N$ are the finite bases of the cylinder in (\ref{def:cyl}), and for $1 \leq m \leq M$, ${\mathbb G}_m$ are the infinite graphs in (\ref{iso0}) and ${\hat {\mathbb G}}_m$ are infinite connected weighted graphs. Unlike ${\mathbb G}_m$, the graphs ${\hat {\mathbb G}}_m$ do not feature in the statement of Theorem~\ref{thm:d}. They do, however, play a crucial role in its proof. Indeed, we will assume that neighborhoods of the points $y_{m,N}$ that are in general much larger than $B(y_{m,N}, {\mathsf r}_N)$ are isomorphic to subsets of ${\hat {\mathbb G}}_m$. For some examples such as the Euclidean box treated in Section~\ref{sec:ex}, this assumption requires that ${\hat {\mathbb G}}_m$ be different from ${\mathbb G}_m$. Assumptions on ${\hat {\mathbb G}}_m$ will then allow us to control certain escape probabilities from the boundary of $B(x_{m,N},{\mathsf r}_N)$ to the complement of $G_N \times {\tilde I}$, for an interval ${\tilde I}$ containing $z_{m,N}$. See also assumptions \ref{hyp:inc}-\ref{hyp:k} and Remark~\ref{rem:ghat} below for more on the graphs ${\hat {\mathbb G}}_m$.

\vspace{2mm}

Under the product measures $P^{\mathcal G}_{\mathsf y} \times P^{\mathbb Z}_z$ on ${\mathcal P}({\mathcal G}) \times {\mathcal P}({\mathbb Z})$, we consider the process ${\mathsf X} = ({\mathsf Y}, {\mathsf Z})$ on ${\mathcal G} \times {\mathbb Z}$. The crucial observation is that $\mathsf X$ has the same distribution as the random walk in continuous time on ${\mathcal G} \times {\mathbb Z}$ attached to the weights 
\begin{align}
&w_{({\mathsf y}, z),({\mathsf y}', z')} = w_{{\mathsf y}, {\mathsf y}'} \mathbf{1}_{\{z = z'\}} + \frac{1}{2} \mathbf{1}_{\{{\mathsf y} = {\mathsf y}', |z - z'|=1\}}, \label{def:cylw}
\end{align}
for any pair of vertices $\{({\mathsf y},z), ({\mathsf y}', z')\}$ in ${\mathcal G} \times {\mathbb Z}$. 
We define the discrete skeleton $(X_n)_{n \geq 0}$ of $\mathsf X$ by $X_n = {\mathsf X}_{\sigma^{\mathsf X}_n}$, with $(\sigma^{\mathsf X}_n)_{n \geq 0}$ the times of discontinuity of $\mathsf X$ (where $\sigma^{\mathsf X}_0 = 0$) and similarly $Z_n = {\mathsf Z}_{\sigma^{\mathsf Z}_n}$ for the times $(\sigma^{\mathsf Z}_n)_{n \geq 0}$ of discontinuity of $\mathsf Z$. We will often rely on the fact that
\begin{align}
&\textrm{$X$ is distributed as the random walk on ${\mathcal G} \times {\mathbb Z}$ with weights as in (\ref{def:cylw}).} \label{eq:Xdist}
\end{align}
The jump process of $\mathsf X$ is defined as $\eta^{\mathsf X}_t = \sup \{n \geq 0: \sigma^{\mathsf X}_n \leq t\}$. We write
\begin{align}
P_x = P^{G_N}_y \times P^{\mathbb Z}_z, \,\, {\mathbb P}^m_{\mathtt x} = P^{{\mathbb G}_m}_{\mathtt y} \times P^{{\mathbb Z}}_z \textrm{ and } {\hat {\mathbb P}}^m_{\mathtt x} = P^{{\hat {\mathbb G}}_m}_{\mathtt y} \times P^{\mathbb Z}_z, \label{def:pr1}
\end{align}
for vertices $x=(y,z)$ in $G_N \times {\mathbb Z}$ and ${\mathtt x}=({\mathtt y},z)$ in ${\mathbb G}_m \times {\mathbb Z}$ or ${\hat {\mathbb G}}_m \times {\mathbb Z}$.
Two measures on $G_N$ are of particular interest: the reversible probability $\pi_{G_N}(y) = w_y/w(G_N)$ for $p^{G_N}(.,.)$ and the uniform measure $\mu(y) = 1/|G_N|$, $y \in G_N$, which is reversible for the continuous-time transition probabilities $q^{G_N}_t(.,.)$, $t \geq 0$. We define 
\begin{align}
P^{G_N} = \sum_{y \in G_N} \mu(y) P^{G_N}_y, \, P_z = \sum_{y \in G_N} \mu(y) P_{(y,z)}, \textrm{ and } P = \sum_{y \in G_N} \mu(y) P_{(y,0)}. \label{def:pr2}
\end{align}
On any path space ${\mathcal P}({\mathcal G})$, the canonical shift operators are denoted by $(\theta_t)_{t \geq 0}$. The shift operators for the discrete-time process $X$ are denoted by $\theta^X_n = \theta_{\sigma^{\mathsf X}_n}$, $n \geq 0$.

\vspace{2mm}

For the process $X$, the entrance-, exit- and hitting times of a set $A$ are defined as
\begin{align}
&H_A = \inf \{n \geq 0: X_n \in A\}, \textrm{ } T_A = \inf \{n \geq 0: X_n \notin A\} \\ 
&\textrm{ and } {\tilde H}_A = \inf \{n \geq 1: X_n \in A\}. \nonumber
\end{align}
In the case $A= \{x\}$, we simply write $H_x$ and ${\tilde H}_x$. We also use the same notation for the corresponding times of the processes $Y$ and $Z$. The analogous times for the continuous-time processes $\mathsf X$, $\mathsf Y$ and $\mathsf Z$ are denoted ${\mathsf H}_A$ and ${\mathsf T}_A$.
Recall the definition of the local time of the $\mathbb Z$-projection of the random walk on $G \times {\mathbb Z}$ from (\ref{def:locpr}). The local times of $\mathsf Z$ and its discrete skeleton $Z$ are defined as
\begin{align}
&{\mathsf L}^z_t = \int_{0}^t \mathbf{1}_{\{{\mathsf Z}_s = z\}} ds \textrm{ and } {\hat L}^z_n = \sum_{l=0}^{n-1} \mathbf{1}_{\{Z_l=z\}}. \label{def:loct}
\end{align}
Note that ${\hat L}^z_n$ should not be confused with the local time $L^z_n$ of the $\mathbb Z$-projection of $X$, defined in (\ref{def:locpr}). The capacity of a finite subset ${\mathbb V}$ of ${\mathbb G}_m \times {\mathbb Z}$ is defined as
\begin{align}
\textup{cap}^m({\mathbb V}) = \sum_{{\mathtt x} \in {\mathbb V}} {\mathbb P}^m_{\mathtt x} [{\tilde H}_{\mathbb V} = \infty] w_{\mathtt x}. \label{def:cap}
\end{align}
For an arbitrary real-valued function $f$ on $G_N$, the Dirichlet form ${\mathcal D}_N(f,f)$ is given by
\begin{align}
{\mathcal D}_N(f,f) = \frac{1}{2} \sum_{y,y' \in G_N} (f(y)-f(y'))^2 \frac{w_{y,y'}}{|G_N|}, \label{def:dir}
\end{align}
and related to the spectral gap $\lambda_N$ of the continuous-time random walk $\mathsf Y$ on $G_N$ via
\begin{align}
\lambda_N = \min \biggl\{ \frac{{\mathcal D}_N(f,f)}{\textup{var}_\mu (f)}: f \textrm{ is not constant} \biggr\}, \textrm{ where } \textup{var}_\mu (f) = \mu \bigl( (f- \mu(f))^2 \bigr). \label{def:spec} 
\end{align}
The inverse $\lambda_N^{-1}$ of the spectral gap is known as the relaxation time of the continuous-time random walk, due to the estimate (\ref{st}).

\vspace{2mm}

We now come to the specification of the hypotheses for Theorem~\ref{thm:d}. Recall that $(G_N)_{N \geq 1}$ is a sequence of finite connected weighted graphs. We consider $M \geq 1$, sequences $x_{m,N} = (y_{m,N},z_{m,N})$, $1 \leq m \leq M$, in $G_N \times {\mathbb Z}$ and an $0 < \epsilon < 1$ such that the assumptions \ref{hyp:w}-\ref{hyp:k} below hold. The first assumption is that the weights attached to vertices of $G_N$ are uniformly bounded from above and below, i.e.
\begin{align}
&\textrm{there are constants $0 < c_0 \leq c_1$ such that } c_0 \leq w_{y} \leq c_1, \textrm{ for all } y \in G_N. \tag{A1} \label{hyp:w}
\end{align}
A frequently used consequence of this assumption is that the jump process of $\mathsf Y$ under $P^G$ can be bounded from above and from below by a Poisson process of constant parameter, see Lemma~\ref{lem:sd} below. Moreover, by taking a function $f$ vanishing everywhere except at a single vertex in (\ref{def:spec}), \ref{hyp:w} implies that $\lambda_N \leq c$. If in addition also the edge-weights $w_{y,y'}$ of $G_N$ are uniformly elliptic, it follows from Cheeger's inequality (see \cite{SC}, Lemma~3.3.7, p.~383) that the relaxation time $\lambda_N^{-1}$ is bounded from above by $c|G_N|^2$. We assume a little bit more, namely that for $\epsilon$ as above,
\begin{align}
&\lambda_N^{-1} \leq |G_N|^{2-\epsilon}, \tag{A2} \label{hyp:mix}
\end{align}
which in particular rules out nearly one-dimensional graphs $G_N$.
We further assume that the mutual distances between different sequences $x_{m,N}$ diverge,
\begin{align}
&\lim_N \min_{1 \leq m < m' \leq M} d(x_{m,N},x_{m',N})  = \infty, \tag{A3} \label{hyp:dist}
\end{align}
and that in scale $|G_N|$, the $\mathbb Z$-components of the sequences $z_{m,N}$ converge:
\begin{align}
&\lim_N \frac{z_{m,N}}{|G_N|} = v_m \in {\mathbb R}, \textrm{ for } 1 \leq m \leq M.  \tag{A4} \label{hyp:zconv}
\end{align}
The key assumption is the existence of balls of diverging size centered at the points $y_{m,N}$ that are isomorphic to balls with fixed centers $o_m$ in the infinite graphs ${\mathbb G}_m$:
\begin{align}
&\textrm{For some } r_N \to \infty, \textrm{ there are isomorphisms } \phi_{m,N} \textrm{ from } B(y_{m,N},r_N) \tag{A5} \label{hyp:iso1}\\
& \textrm{to } B(o_m,r_N) \subset {\mathbb G}_m,  \textrm{ such that } \phi_{m,N}(y_{m,N}) = o_m \textrm{ for all } N, m. \nonumber
\end{align}
In the proof of Theorem~\ref{thm:d}, we want to show the decay of the probability that the random walk $X$ under $P$ returns to the close vicinity of the center $x_{m,N}$ from the boundary of each of the balls $B(x_{m,N},r_N) \subset G_N \times {\mathbb Z}$ before exiting a large box. With this aim in mind, we make the remaining assumptions. For any $m$, $N$, we assume that there exists an associated subset $C_{m,N}$ of $G_N$ such that
\begin{align}
&B(y_{m,N},r_N) \subseteq C_{m,N}, \tag{A6} \label{hyp:inc}
\end{align}
and ${\bar C}_{m,N}$ are isomorphic to a subset of the auxiliary limit model ${\hat {\mathbb G}}_m$, i.e.
\begin{align}
&\textrm{there is an isomorphism } \psi_{m,N} \textrm{ from } {\bar C}_{m,N} \textrm{ with a set } {\bar {\mathbb C}}_m \subset {\hat {\mathbb G}}_m, \tag{A7} \label{hyp:iso2} \\ &\textrm{such that } \psi_{m,N}(\partial C_{m,N}) = \partial {\mathbb C}_{m,N}, \nonumber
\end{align}
where the last condition is to ensure that the distributional identity (\ref{eq:isod2}) below holds. Note that we are allowing the infinite graphs ${\hat {\mathbb G}}_m$ to be different from ${\mathbb G}_m$. For an explanation, we refer to Remark~\ref{rem:ghat} below (see also Remark~\ref{rem:aux}). 
We further assume that the sets $C_{m,N}$ as $m$ varies are essentially either disjoint or equal (unless the corresponding $\mathbb Z$-components $z_{m,N}$ are far apart), i.e.
\begin{align}
&\textrm{whenever } v_m = v_{m'}, \textrm{ then for all $N$ either } C_{m,N} = C_{m',N} \textrm{ or } C_{m,N} \cap C_{m',N} = \emptyset.  \tag{A8} \label{hyp:disj} 
\end{align}
Concerning the limit model $\hat {\mathbb G}_m$, we require that the measure of a constant-size ball centered at ${\hat o}_{m,N} \stackrel{\textrm{(def.)}}{=} \psi_{m,N}(y_{m,N})$ under the law $Y_n \circ P^{\hat {\mathbb G}_m}_.$ decays faster than $n^{-\frac{1}{2} - \epsilon}$,
\begin{align}
&\lim_{n \to \infty} n^{\frac{1}{2} + \epsilon} \sup_{{\mathtt y}_0 \in {\hat {\mathbb G}}_m} \sup_{{\mathtt y} \in B({\hat o}_{m,N},\rho_0), N \geq 1} p^{{\hat {\mathbb G}}_m}_n ({\mathtt y}_0,{\mathtt y}) = 0, \textrm{ for any } \rho_0>0. \tag{A9}\label{hyp:trans}
\end{align}
This assumption is only used to prove Lemma~\ref{lem:mntrans} below. Let us mention that \ref{hyp:trans} typically holds whenever the on-diagonal transition densities decay at the same rate, see Remark~\ref{rem:dec} below. Finally, we assume that the random walk on $G_N \times {\mathbb Z}$, started at the interior boundary of $C_{m,N} \times {\mathbb Z}$, is unlikely to reach the vicinity of $x_{m,N}$ until well after the relaxation time of $Y$:
\begin{align}
\lim_N \sup_{y_0 \in \partial (C_m^c), z_0 \in {\mathbb Z}} P_{(y_0,z_0)} [ H_{(\phi_{m,N}^{-1}({\mathtt y}),z_{m,N} + z)}  < \lambda_N^{-1} |G_N|^\epsilon] =0,  \tag{A10} \label{hyp:k}
\end{align}
for any  $({\mathtt y},z) \in {\mathbb G}_m \times {\mathbb Z}$ (note that $\phi_{m,N}^{-1}({\mathtt y})$ is well-defined for large $N$ by \ref{hyp:iso1}).
%\begin{align}
%&\lim_N \sup_{0 \leq n \leq \lambda_N^{-1} |G_N|^\epsilon} \sup_{\substack{y_0 \in \partial (C_{m,N}^c) \\ y \in B(y_{m,N},\rho_0)}} \frac{p^{G_N}_n %(y_0,y)|G_N|^\epsilon}{\sqrt{\lambda_{N}}} =0, \textrm{ for any $\rho_0 > 0$.}  \tag{A10} \label{hyp:k}
%\end{align}

\begin{remark}\label{rem:ghat}
 \textup{
The infinite graphs ${\hat {\mathbb G}}_m$ in \ref{hyp:iso2} can be different from the graphs ${\mathbb G}_m$ describing the neighborhoods of the points $y_{m,N}$. The reason is that for \ref{hyp:k} to hold, the sets $C_{m,N}$ will generally have to be of much larger diameter than their subsets $B(y_m,r_N)$. Hence, ${\bar C}_m$ is not necessarily isomorphic to a subset of the same infinite graph as $B(y_m,r_N)$. This situation occurs, for example, if $G_N$ is given by a Euclidean box, see Remark~\ref{rem:aux}.
}
\end{remark}

\begin{remark}\label{rem:dec}
 \textup{
Typically, the weights attached to the vertices of ${\hat {\mathbb G}}_m$ are uniformly bounded from above and from below, as are the weights in $G_N$ (see~(\ref{hyp:w})). In this case, assumption~\ref{hyp:trans} holds in particular whenever one has the on-diagonal decay $$\lim_n n^{\frac{1}{2} + \epsilon} \sup_{{\mathsf y} \in {\hat {\mathbb G}}_m} p^{{\hat {\mathbb G}}_m}_n ({\mathtt y},{\mathtt y}) \to 0,$$ see \cite{T06}, Lemma~8.8, p.~108, 109.
}
\end{remark}

\vspace{2mm}

From now on, we often drop the $N$ from the notation in $G_N$, $C_{m,N}$, $x_{m,N}$, $\phi_{m,N}$ and $\psi_{m,N}$. We extend the isomorphisms $\phi_m$ and $\psi_m$ in \ref{hyp:iso1} and \ref{hyp:iso2} to isomorphisms $\Phi_m$ and $\Psi^{z_0}_m$ defined on $B(y_m,r_N) \times {\mathbb Z}$ and on ${\bar C}_m \times {\mathbb Z}$ by
\begin{align}
&\Phi_m: (y,z) \mapsto (\phi_m(y),z-z_m), \textrm{ and } \label{def:Phi}\\
&\Psi^{z_0}_m:  (y,z)  \mapsto  (\psi_m(y),z-z_0), \textrm{ for } z_0 \in {\mathbb Z}. \label{def:Psi}
\end{align}
A crucial consequence of (\ref{hyp:iso1}) and (\ref{hyp:iso2}) is that for $r_N \geq 1$,
\begin{align}
&\textrm{$({\mathsf X}_t:0 \leq t \leq {\mathsf T}_{B(y_m,r_N-1) \times {\mathbb Z}})$ under $P_x$ has the same distribution as} \label{eq:isod1}\\
&\qquad \textrm{$(\Phi_m^{-1}({\mathsf X}_t):0 \leq t \leq {\mathsf T}_{B(o_m,r_N-1) \times {\mathbb Z}})$ under ${\mathbb P}^m_{\Phi_m(x)}$, and} \nonumber\\
&\textrm{$({\mathsf X}_t:0 \leq t \leq {\mathsf T}_{C_m \times {\mathbb Z}})$ under $P_x$ has the same distribution as} \label{eq:isod2}\\
&\qquad \textrm{$((\Psi^{z_0}_m)^{-1}({\mathsf X}_t):0 \leq t \leq {\mathsf T}_{{\mathbb C}_m \times {\mathbb Z}})$ under ${\hat {\mathbb P}}^m_{\Psi^{z_0}_m(x)}$.} \nonumber
\end{align}

The assumption \ref{hyp:trans} only enters the proof of the following lemma showing the decay of the probability that the random walk on the cylinders ${\mathbb G}_m \times {\mathbb Z}$ or ${\hat {\mathbb G}}_m \times {\mathbb Z}$ returns from distance $\rho$ to a constant-size neighborhood of $(o_m,0)$ or $(\psi_m (y_m),0)$ as $\rho$ tends to infinity. Note that this in particular implies that these cylinders are transient and the random interlacements appearing in Theorem~\ref{thm:d} make sense.
\begin{lemma} \label{lem:mntrans}
\textup{($1 \leq m \leq M$)} Assuming \textup{\ref{hyp:w}-\ref{hyp:k}}, for any $\rho_0 >0$,
\begin{align}
&\lim_{\rho \to \infty}  \sup_{\substack{d({\mathtt x},({\hat o}_m,0)) \leq \rho_0 \\ d({\mathtt x}_0,{\mathtt x}) \geq \rho}} {\hat {\mathbb P}}^m_{{\mathtt x}_0} [H_{\mathtt x} < \infty] =0, \textrm{ and } \lim_{\rho \to \infty}  \sup_{\substack{d({\mathtt x},(o_m,0)) \leq \rho_0 \\ d({\mathtt x}_0,{\mathtt x}) \geq \rho}} {\mathbb P}^{m}_{{\mathtt x}_0} [H_{\mathtt x} < \infty] =0.  \label{mntrans}
\end{align}
\end{lemma}

The proof of Lemma \ref{lem:mntrans} requires the following two lemmas of frequent use.

\begin{lemma} \label{lem:sd}
Let $\mathcal G$ be a weighted graph such that $0 < \inf_{\mathsf y} w_{\mathsf y} \leq \sup_{\mathsf y} w_{\mathsf y} < \infty$.
\begin{align}
&\textrm{Under $P^{\mathcal G}_{\mathsf y}$, } e_n = (\sigma^{\mathsf Y}_n - \sigma^{\mathsf Y}_{n-1}) w_{Y_{n-1}}, \, n \geq 1, \textrm{ is a sequence of}\label{sd1}\\ 
&\textrm{iid $\exp(1)$ random variables, independent of } Y, \textrm{ and} \nonumber\\
&\eta^{\inf_{\mathsf y} w_{\mathsf y}}_t \leq \eta^{\mathsf Y}_t \leq \eta^{\sup_{\mathsf y} w_{\mathsf y}}_t, \textrm{ for } t \geq 0, \label{sd2}
\end{align}
where $\eta^\nu_t = \sup \{n \geq 0: e_1 + \ldots + e_n \leq \nu t\}$, $t \geq 0$, with $(e_n)_{n \geq 1}$ as defined above, is a Poisson process with rate $\nu \geq 0$.
\end{lemma}

\begin{proof}
The assertion \eqref{sd1} follows from a standard construction of the continuous-time Markov chain $\mathsf Y$, see for example \cite{N97}, pp.~88, 89. For \eqref{sd2}, note that for any $k \geq 0$,
\begin{align}
\label{e:sd1} \frac{w_{Y_k}}{\sup_{\mathsf y} w_{\mathsf y}} \leq 1 \leq   \frac{w_{Y_k}}{\inf_{\mathsf y} w_{\mathsf y}},
\end{align}
hence for $t \geq 0$,
\begin{align*}
 \eta^{\mathsf Y}_t &= \sup \Bigl\{ n \geq 0: \sum_{k=1}^n (\sigma^{\mathsf Y}_k - \sigma^{\mathsf Y}_{k-1}) \leq t \Bigr\} \\
&\stackrel{\eqref{e:sd1}}{\leq} \sup \Bigl\{ n \geq 0: \sum_{k=1}^n (\sigma^{\mathsf Y}_k - \sigma^{\mathsf Y}_{k-1})  \frac{w_{Y_{k-1}}}{\sup_{\mathsf y} w_{\mathsf y}} \leq t \Bigr\} = \eta^{\sup_{\mathsf y} w_{\mathsf y}}_t,
\end{align*}
as well as
\begin{align*}
 \eta^{\mathsf Y}_t &\stackrel{\eqref{e:sd1}}{\geq} \sup \Bigl\{ n \geq 0: \sum_{k=1}^n (\sigma^{\mathsf Y}_k - \sigma^{\mathsf Y}_{k-1})  \frac{w_{Y_{k-1}}}{\inf_{\mathsf y} w_{\mathsf y}} \leq t \Bigr\} = \eta^{\inf_{\mathsf y} w_{\mathsf y}}_t.
\end{align*}

\end{proof}

%\begin{proof}
%Construct a measure on ${\mathcal G}^{\mathbb N}$ under which a discrete-time process $Y'$ is distributed as the random walk on $G$  starting from $\mathsf y$ with transition probabilities $p^G(.,.)$. Then add independent iid $\exp(1)$ random variables $e_1, e_2, \ldots$, set $\sigma^{{\mathsf Y}'}_0=0$, $\sigma^{{\mathsf Y}'}_n = e_1/w_{Y'_0} + \ldots + e_n/w_{Y'_n}$ for $n \geq 1$, define $\eta^{{\mathsf Y}'}_t = \sup \{n \geq 0: \sigma^{{\mathsf Y}'}_n \leq t \}$ and ${\mathsf Y}'$ by ${\mathsf Y}'_t = Y'_{\eta^{{\mathsf Y}'}_t}$. Then $P^{\mathcal G}_{\mathsf y}$ is the law of ${\mathsf Y}'$ on ${\mathcal P}({\mathcal G})$, from which (\ref{sd1}) follows immediately and implies (\ref{sd2}).
%\end{proof}

\begin{lemma} \label{lem:zk}
\begin{align}
&P^{\mathbb Z}_z [z' \in {\mathsf Z}_{[s,t]}] \leq c \frac{1+t-s}{\sqrt{s}}, \textrm{ for } 0 < s \leq t < \infty, \, z, z' \in {\mathbb Z}. \label{zk}
\end{align}
\end{lemma}

\begin{proof}
By the strong Markov property applied at time $s+ {\mathsf H}_{z'} \circ \theta_s$,
\begin{align}
E^{\mathbb Z}_z \bigl[ \int_s^{t+1} \mathbf{1}_{\{{\mathsf Z}_r=z'\}} dr \bigr] &\geq E^{\mathbb Z}_z \bigl[ s+ {\mathsf H}_{z'} \circ \theta_s \leq t, \int_{s+{\mathsf H}_{z'} \circ \theta_s}^{t+1} \mathbf{1}_{\{{\mathsf Z}_r=z'\}} dr \bigr] \label{eq:zk1}\\
&\geq P^{\mathbb Z}_z [{\mathsf H}_{z'} \circ \theta_s \leq t-s] E^{\mathbb Z}_{z'} \bigl[ \int_0^1 \mathbf{1}_{\{{\mathsf Z}_r = z'\}} dr \bigr] \nonumber\\
& \geq P^{\mathbb Z}_z [z' \in {\mathsf Z}_{[s,t]}] E^{\mathbb Z}_{z'}[ \sigma^{\mathsf Z}_1 \wedge 1] \geq c P^{\mathbb Z}_z [z' \in {\mathsf Z}_{[s,t]}]. \nonumber
\end{align}
It follows from the local central limit theorem, see \cite{L91}, (1.10), p.~14, (or from a general upper bound on heat kernels of random walks, see Corollary 14.6 in \cite{W00}) that 
\begin{align}
P^{\mathbb Z}_z [ Z_n = z'] \leq c/\sqrt{n}, \textrm{ for all $z$ and $z'$ in $\mathbb Z$ and $n \geq 1$.} \label{eq:zk2}
\end{align}
 Using an exponential bound on the probability that a Poisson variable of intensity $2t$ is not in the interval $[t,4t]$, it readily follows that $P^{\mathbb Z}_z [ {\mathsf Z}_t = z'] \leq c/\sqrt{t}$ for all $t>0$, hence
\begin{align*}
E^{\mathbb Z}_z \bigl[ \int_s^{t+1} \mathbf{1}_{\{{\mathsf Z}_r=z'\}} dr \bigr] \leq c\int_s^{t+1} \frac{1}{\sqrt{r}} dr \leq c \frac{1+t-s}{\sqrt{s}}.
\end{align*}
With (\ref{eq:zk1}), this implies (\ref{zk}). 
\end{proof}

\begin{proof}[Proof of Lemma~\ref{lem:mntrans}.]
Denote by ${\mathbb G}$ either one of the graphs ${\hat {\mathbb G}}_m$ or ${\mathbb G}_m$ and by $\mathbb P$ the corresponding probabilities ${\hat {\mathbb P}}^m$ and ${\mathbb P}^m$. Assume for the moment that for all $n \geq c(\epsilon,\rho_0)$,
\begin{align}
\sup_{{\mathtt y}_0 \in {\mathbb G}} \sup_{{\mathtt y} \in B(o,\rho_0)} p^{\mathbb G}_n({\mathtt y}_0, {\mathtt y}) \leq c(\rho_0) n^{-\frac{1}{2}-\epsilon}, \label{eq:mntrans1}
\end{align}
where $o$ denotes the corresponding vertex ${\hat o}_{m,N}$ or $o_m$. For any points ${\mathtt x}=({\mathtt y},z)$ in $B((o,0),\rho_0)$ and ${\mathtt x}_0=({\mathtt y}_0,z_0)$ in ${\mathbb G} \times {\mathbb Z}$ such that $d({\mathtt x}_0, {\mathtt x}) \geq \rho$, we have
\begin{align}
{\mathbb P}_{{\mathtt x}_0}[H_{\mathtt x} < \infty] \leq  \sum_{n= [\rho]}^\infty P_{({\mathtt y}_0,z_0)} [ Y_n = {\mathtt y} , z \in {\mathsf Z}_{[\sigma^{\mathsf Y}_n, \sigma^{\mathsf Y}_{n+1}]}],  \label{eq:mntrans2}
\end{align}
By independence of  $(Y,\sigma^{\mathsf Y})$ and $\mathsf Z$, the probability in this sum can be rewritten as
\begin{align*}
E^G_{{\mathtt y}_0} \Bigl[ Y_n = {\mathtt y}, P^{\mathbb Z}_{z_0}[ z \in {\mathsf Z}_{[s,t]}] \bigl|_{\substack{s = \sigma^{\mathsf Y}_n \\ t= \sigma^{\mathsf Y}_{n+1}}} \Bigr],
\end{align*}
which by the estimate (\ref{zk}) and the strong Markov property at time $\sigma^{\mathsf Y}_n$ is smaller than
\begin{align*}
cE^G_{{\mathtt y}_0} \Bigl[ Y_n = {\mathtt y}, \frac{1 + \sigma_1 \circ \theta_{\sigma^{\mathsf Y}_n}}{\sqrt{\sigma^{\mathsf Y}_n}} \Bigr] \stackrel{(\ref{hyp:w})}{\leq} cE^G_{{\mathtt y}_0} \Bigl[ Y_n = {\mathtt y}, \frac{1}{\sqrt{\sigma^{\mathsf Y}_n}} \Bigr]. 
\end{align*}
By (\ref{sd1}) and \ref{hyp:w}, the sum in (\ref{eq:mntrans2}) can be bounded by 
\begin{align}
c\sum_{n=[\rho]}^\infty p^G_n({\mathtt y}_0,{\mathtt y}) E \Bigl[ \frac{1}{\sqrt{e_1 + \ldots + e_n}} \Bigr] \leq c\sum_{n=[\rho]}^\infty p^G_n({\mathtt y}_0,{\mathtt y}) \frac{1}{\sqrt{n}}, \label{eq:mntrans3}
\end{align}
where we have used that $E[1/(e_1 + \ldots + e_n)] = 1/(n-1)$ for $n \geq 2$ (note that $e_1 + \ldots + e_n$ is $\Gamma(n,1)$-distributed), together with Jensen's inequality. By the bound assumed in (\ref{eq:mntrans1}), this implies with (\ref{eq:mntrans2}) that
\begin{align*}
\sup_{\substack{d({\mathtt x},(o,0)) \leq \rho_0 \\ d({\mathtt x}_0,{\mathtt x}) \geq \rho}} {\mathbb P}_{{\mathtt x}_0}[H_{\mathtt x} < \infty] \leq  c(\rho_0) \sum_{n=[\rho]}^\infty n^{-1-\epsilon}.
\end{align*}
Since the right-hand side tends to $0$ as $\rho$ tends to infinity, this proves both claims in (\ref{mntrans}), provided (\ref{eq:mntrans1}) holds for ${\hat {\mathbb G}}_m$ and ${\mathbb G}_m$ in place of $\mathbb G$. 
In fact, (\ref{eq:mntrans1}) does hold for ${\mathbb G} = {\hat {\mathbb G}}_m$ by assumption \ref{hyp:trans}, and also holds for ${\mathbb G} = {\mathbb G}_m$ by the following argument: Consider any ${\mathtt y}_0 \in {\mathbb G}_m$, ${\mathtt y} \in B(o_m,\rho_0)$ and $n \geq 0$.
Choose $N$ sufficiently large such that $r_N - d({\mathtt y}_0,o_m) > n$ and both ${\mathtt y}_0$ and ${\mathtt y}$ are contained in $B(o_m,r_N)$ (cf.~\ref{hyp:iso1}). Using the isomorphism ${\hat \psi} = \psi_m \circ \phi^{-1}_m$ from $B(o_m,r_N)$ to $B({\hat o}_m,r_N) \subset {\hat {\mathbb G}}_m$, we deduce that
\begin{align}
p^{{\mathbb G}_m}_n({\mathtt y}_0, {\mathtt y}) &= P^{{\mathbb G}_m}_{{\mathtt y}_0} [Y_n = {\mathtt y}, T_{B(o_m,r_N-1)} \geq r_N - d({\mathtt y}_0,o_m)]. \label{eq:mntrans5} \\
&= P^{{\hat {\mathbb G}}_m}_{{\hat \psi}({\mathtt y}_0)} [Y_n = {\hat \psi}({\mathtt y}), T_{B({\hat o}_m,r_N-1)} \geq r_N - d({\mathtt y}_0,o_m)] \nonumber\\
&\leq p^{{\hat {\mathbb G}}_m}_n({\hat \psi}({\mathtt y}_0), {\hat \psi}({\mathtt y})) \leq c(\rho_0) n^{-\frac{1}{2}-\epsilon}, \nonumber
\end{align}
using assumption \ref{hyp:trans} in the last step. This concludes the proof of Lemma~\ref{lem:mntrans}.
\end{proof}

\section{Auxiliary results on excursions and local times} \label{sec:grid}

In this section we reproduce a suitable version of the partially inhomogeneous grids on $\mathbb Z$ introduced in Section~2 of \cite{S08}. These grids allow to relate excursions of the walk $Z$ associated to the grid points to the total time elapsed and to the local time $\hat L$ of $Z$. This is essentially the content of Proposition~\ref{pr:sp} below, quoted from \cite{S08}. We then complement this result with an estimate relating the local time $\hat L$ of $Z$ to the local time $\mathsf L$ of the continuous-time process $\mathsf Z$ in Lemma~\ref{lem:loc}.

\vspace{2mm}

For integers $1 \leq d_N \leq h_N$ and points $z^*_{l,N}$, $1 \leq l \leq M$, in $\mathbb Z$ (to be specified below), we define the intervals
\begin{align}
I_l = [z^*_l-d_N,z^*_l+d_N] \subseteq {\tilde I}_l = (z^*_l-h_N,z^*_l+h_N), \label{def:i}
\end{align}
dropping the $N$ from $z^*_{l,N}$ for ease of notation. The collections of these intervals are denoted by
\begin{align}
{\mathcal I} = \{I_l, 1 \leq l \leq M\}, \textrm{ and } {\tilde {\mathcal I}} = \{{\tilde I}_l, 1 \leq l \leq M\}. \label{def:I}
\end{align}
The anisotropic grid ${\mathcal G}_N \subset {\mathbb Z}$, is defined as in \cite{S08}, (2.4):
\begin{align}
&{\mathcal G}_N = {\mathcal G}_N^* \cup {\mathcal G}^0, \textrm{ where } {\mathcal G}_N^* = \{z^*_l, 1 \leq l \leq M\} \textrm{ and} \label{def:grid}\\
&{\mathcal G}^0_N = \{ z \in 2h_N{\mathbb Z}: |z-z^*_l| \geq 2h_N, \textrm{ for } 1 \leq l \leq M\}. \nonumber
\end{align}
It remains to choose $d_N$, $h_N$ and $z^*_l$. In \cite{S08}, no upper bound other than $o(|G_N|)$ is needed on the distance between neighboring grid points, but we want an upper bound not much larger than $\lambda_N^{-1/2}$. A consequence of this requirement is that unlike in \cite{S08}, we may attach several points $z^*_l$ to the same limit $v_m$ in \ref{hyp:zconv}. We satisfy this requirement by a judicious choice such that
\begin{align}
&\lambda_N^{-1/2} |G_N|^{\epsilon/8} \leq d_N, \textrm{ } d_N =o(h_N), \textrm{ } h_N \leq \lambda_N^{-1/2} |G_N|^{\epsilon/4}, \label{s1}\\
& \min_{1 \leq l < l' \leq M} |z^*_l - z^*_{l'}| \geq 100h_N, \textrm{ and}\label{s2}\\
& \{z_1, \ldots, z_M\} \subseteq \cup_{l=1}^M [z^*_l-[d_N/2],z^*_l+[d_N/2]], \textrm{ for all } N \geq c(\epsilon,M). \label{s3}
\end{align}
\begin{proposition}\label{lem:s}
Points $z^*_1, \ldots, z^*_M$ in $\mathbb Z$ and sequences $d_N$, $h_N$ in $\mathbb N$ satisfying \textup{(\ref{s1})-(\ref{s3})} exist.
\end{proposition}
The proof of Proposition \ref{lem:s} is a consequence of the following simple lemma, asserting that for prescribed numbers $a, b$ and $q \geq 2$, any $M$ points in a metric space can be covered by balls of radius between $a$ and $b^{2M}a$, such that the balls with radius multiplied by $b$ are disjoint and no more than $M$ balls are required.
\begin{lemma} \label{lem:met}
Let ${\mathcal X}$ be a metric space and $x_1, \ldots, x_M$, $M \geq 1$, points in $\mathcal X$. Consider real numbers $a \geq 1$ and $b \geq 2$. Then for some $M_* \leq M$ and $a \leq p \leq b^{2M} a$, there are points $\{x^*_1, \ldots, x^*_{M_*}\}$ in $\mathcal X$ such that 
\begin{align*}
\bigcup_{1 \leq i \leq M_*} B(x^*_i, p) \supseteq \{x_1, \ldots, x_M\} \textrm{, and } \textrm{the balls } (B(x^*_i, bp))_{i=1}^{M_*} \textrm{ are disjoint,}
\end{align*}
where $B(x,r)$ denotes the closed ball of radius $r \geq 0$ centered at $x \in {\mathcal X}$.
\end{lemma}

\begin{proof}[Proof of Proposition~\ref{lem:s}.]
Lemma~\ref{lem:met}, applied with ${\mathcal X} = {\mathbb Z}$ and the points $z_1, \ldots, z_M$ with $a= [\lambda_N^{-1/2} |G|^{\epsilon/8}]$ and $b=[(|G|^{\epsilon/8})^{1/(2M+1)}]$, yields points $z^*_1, \ldots, z^*_{M_*}$ in $\mathbb Z$ and a $p$ between $a$ and $b^{2M} a$ such that (\ref{s1})-(\ref{s3}) hold for $d_N=[2p]$, $h_N=[bp/100]$ and $M_*$ in place of $M$. The additional points $z^*_{M_*+1}, \ldots, z^*_M$ can be chosen arbitrarily subject only to (\ref{s2}).
\end{proof}

\begin{proof}[Proof of Lemma~\ref{lem:met}.]
For $m \geq 0$, set
\begin{align*}
k_m = \min \{ k \geq 0: \textrm{ for some } x'_1, \ldots, x'_k \textrm{ in } {\mathcal X}, \cup_{i=1}^k B(x'_i, b^{2m}a) \supseteq \{x_1, \ldots, x_M\} \},
\end{align*}
and denote points for which the minimum is attained by $x^m_1, \ldots, x^m_{k_m}$. The first observation on $k_m$ is that clearly $1 \leq k_m \leq M$. The second observation is that
\begin{align*}
\textrm{either the balls $B(x^m_i,b^{2m+1}a)$, $1 \leq i \leq k_m$, are disjoint, or $k_{m+1} < k_m$, for $m \geq 0$.}
\end{align*}
Indeed, assume that ${\bar x} \in B(x^m_i,b^{2m+1} a) \cap B(x^m_j,b^{2m+1} a)$ for $1 \leq i < j \leq k_m$. Then since $b \geq 2$, the $k_m -1$ balls of radius $b^{2(m+1)} a$ centered at $(\{x^m_1, \ldots, x^m_{k_m}\} \cup \{{\bar x}\}) \setminus \{x^m_i, x^m_j\}$ still cover $\{x_1, \ldots, x_M\}$. Thanks to these two observations, we may define 
\begin{align*}
m_* = \min \{ m \geq 0: \textrm{ the balls } B(x^m_i,b^{2m+1}a), \, 1 \leq i \leq k_m \textrm{ are disjoint} \} \leq M,\
\end{align*}
and set $M_* = k_{m_*}$, $x^*_i = x^{m_*}_i$ for $1 \leq i \leq M_*$ and $p=b^{2m_*}a$.
\end{proof}

The grids ${\mathcal G}_N$ we consider from now on are specified by (\ref{def:i})-(\ref{s3}). In order to define the associated excursions, we define the sets $C$ and $O$, whose components are intervals of radius $d_N$ and $h_N$, centered at the points in the grid ${\mathcal G}_N$, i.e. 
\begin{align}
\label{def:CO} C = {\mathcal G}_N + [-d_N,d_N] \subset O= {\mathcal G}_N + (-h_N,h_N).
\end{align}
The times $R_n$ and $D_n$ of return to $C$ and departure from $O$ of the process $Z$ are defined as 
\begin{align}
&R_1 = H_C, D_1 = T_O \circ \theta_{R_1} + R_1, \textrm{ and for } n \geq 1, \label{eq:ex}\\
&R_{n+1} = R_1 \circ \theta_{D_n} + D_n, \textrm{ } D_{n+1} = D_1 \circ \theta_{D_n} + D_n,  \nonumber
\end{align}
so that $0 \leq R_1 < D_1 < \ldots < R_n < D_n$, $P^{\mathbb Z}_z$-a.s. For later use, we denote for any $\alpha > 0$,
\begin{align}
&{\mathsf t}_N = E^{\mathbb Z}_0[T_{(-h_N+d_N,h_N-d_N)}] + E^{\mathbb Z}_{d_N}[T_{(-h_N,h_N)}] = (h_N-d_N)^2 + h_N^2 -d_N^2, \label{def:t}\\
&\sigma_N=[\alpha |G|^2/{\mathsf t}_N], \textrm{ } k_*(N) = \sigma_N - [\sigma_N^{3/4}], \textrm{ } k^*(N) = \sigma_N + [\sigma_N^{3/4}], \label{def:sigma}
\end{align}
where we will often drop the $N$ from now on. We come to the crucial result on these returns and departures from \cite{S08}, relating the times $D_k$ to the total time elapsed (\ref{sp1}) and to the local time $\hat L$ of $Z$ ((\ref{sp2})-(\ref{sp4})).
\begin{proposition} \label{pr:sp}
Assuming \ref{hyp:mix},
\begin{align}
&\lim_N P^{\mathbb Z}_0 [D_{k_*} \leq \alpha |G_N|^2 \leq D_{k^*} ] =1. \label{sp1}\\
& \lim_N \sup_{z \in C} E^{\mathbb Z}_0[ (|{\hat L}^z_{[\alpha|G_N|^2]} - {\hat L}^z_{D_{k_*}}|/|G_N|) \wedge 1] = 0. \label{sp2}\\
& \sup_N \max_{I \in {\mathcal I}} \frac{h_N}{|G_N|} E^{\mathbb Z}_0 \Bigl[ \sum_{1 \leq k \leq k_*} 1_{\{Z_{R_k} \in I\}} \Bigr] < \infty. \label{sp3}\\
& \lim_N \max_{I \in {\mathcal I}} \sup_{z \in I} E^{\mathbb Z}_0 \Bigl[ \Bigl| {\hat L}^z_{D_{k_*}} - h_N \sum_{1 \leq k \leq k_*} 1_{\{Z_{R_k} \in I\}} \Bigr| \Bigr] /|G_N| =0. \label{sp4} 
\end{align}
\end{proposition}

\begin{proof}
The above statement is proved by Sznitman in \cite{S08}. Indeed, in \cite{S08}, the author considers three sequences of non-negative integers $(a_N)_{N \geq 1}$, $(h_N)_{N \geq 1}$, $(d_N)_{N \geq 1}$, such that
\begin{align}
 \label{e:sp1} 
\begin{array}{l}
 \lim_N a_N = \lim_N h_N = \infty, \textrm{ and}\\
 d_N = o(h_N), \, h_N=o(a_N) \textrm{ (cf.~(2.1) in \cite{S08})},
\end{array}
\end{align}
as well as sequences $z_{l,N}^*$ of points in $\mathbb Z$ satisfying \eqref{s2} (cf.~(2.2) in \cite{S08}). The grids ${\mathcal G}_N$ are then defined as in \eqref{def:grid} (cf.~(2.4) in \cite{S08}) and the corresponding sets $C$ and $O$ as in \eqref{def:CO} (cf.~(2.5) in \cite{S08}). For any $\gamma \in (0,1]$, $z \in {\mathbb Z}$, Sznitman in \cite{S08} then introduces the canonical law $Q^\gamma_z$ on ${\mathbb Z}^{\mathbb N}$ of the random walk on $\mathbb Z$ which jumps to one of its two neighbors with probability $\gamma/2$ and stays at its present location with probability $1-\gamma$. The times $(R_n)_{n \geq 1}$ and $(D_n)_{n \geq 0}$ of return to $C$ and departure from $O$ are introduced in (2.9) of \cite{S08}, exactly as in \eqref{eq:ex} above. The sequences ${\mathsf t}_N$, $\sigma_N$, $k_*(N)$, $k^*(N)$ are defined in (2.10)-(2.12) of \cite{S08} as in \eqref{def:t} and \eqref{def:sigma} above, with $|G_N|$ replaced by $a_N$ and $E^{\mathbb Z}_.$ replaced by the $Q^\gamma_.$-expectation $E^\gamma_.$. Under these conditions, the statements \eqref{sp1}-\eqref{sp4} are proved in \cite{S08}, Proposition~2.1, with $|G_N|$ replaced by $a_N$ and $P^{\mathbb Z}_0$ and $E^{\mathbb Z}_0$ replaced by $P^\gamma_0$ and $E^\gamma_0$. All we have to do to deduce the above statements is to choose $\gamma=1$ and $a_N = |G_N|$ in Proposition 2.1 of \cite{S08}, noting that \eqref{e:sp1} is then satisfied, by \eqref{s1} and \ref{hyp:mix}.
\end{proof}

We now relate the local time of $Z$ to the local time of the continuous-time process $\mathsf Z$.

\begin{lemma} \label{lem:loc}
\begin{align}
&\sup_{z \in {\mathbb Z}} E^{\mathbb Z}_0 \bigl[ {\hat L}^z_{[\alpha |G_N|^2]}\bigr] \leq c(\alpha) |G_N|, \textrm{ for } \alpha > 0. \label{loct2}\\
&\lim_N \sup_{z \in {\mathbb Z}} E^{\mathbb Z}_0[ (|{\mathsf L}^z_{\alpha|G_N|^2} - {\hat L}^z_{[\alpha|G_N|^2]}|/|G_N|) \wedge 1] =0. \label{loc1}
\end{align}
\end{lemma}

\begin{proof}
For (\ref{loct2}), apply the bound $P_0[Z_n=z] \leq c/ \sqrt{n}$ (cf.~(\ref{eq:zk2})), see (2.34) in \cite{S08}. 

We write $T=\alpha |G|^2$. By the strong Markov property applied at time $\sigma^{\mathsf Z}_{[T]} \wedge T$,
\begin{align}
E^{\mathbb Z}_0 [ |{\mathsf L}^z_{\sigma^{\mathsf Z}_{[T]}} - {\mathsf L}^z_T| ] &= E^{\mathbb Z}_0 \Bigl[ \int_{\sigma^{\mathsf Z}_{[T]} \wedge T}^{\sigma^{\mathsf Z}_{[T]} \vee T} \mathbf{1}_{\{{\mathsf Z}_s = z\}} ds  \Bigr] \leq \sup_{z_0 \in {\mathbb Z}} E^{\mathbb Z}_{z_0} \Bigl[ \int_0^{|\sigma^{\mathsf Z}_{[T]} - T|} \mathbf{1}_{\{{\mathsf Z}_s=z\}} ds \Bigr] \label{eq:loc0}\\ 
&\leq \int_0^{T^{2/3}} \sup_{z_0 \in {\mathbb Z}} P^{\mathbb Z}_{z_0}[{\mathsf Z}_s = z] ds + E_0^{\mathbb Z} [ (\sigma^{\mathsf Z}_{[T]} - T)^2]/T^{2/3}, \nonumber
\end{align}
using the Chebyshev inequality in the last step. By the bound (\ref{zk}) on $P^{\mathbb Z}_{z_0}[{\mathsf Z}_s = z]$ and a bound of $cT$ on the variance of the $\Gamma([T],1)$-distributed variable $\sigma^{\mathsf Z}_{[T]}$, the right-hand side of (\ref{eq:loc0}) is bounded by $cT^{1/3}$. Hence, the expectation in (\ref{loc1}) is bounded by
\begin{align}
c(\alpha) |G|^{-1/3} + E^{\mathbb Z}_0[ (|{\mathsf L}^z_{\sigma^{\mathsf Z}_{[T]}} - {\hat L}^z_{[T]}|/|G|) \wedge 1]. \label{eq:loc1}
\end{align}
The strategy is to now split up the last expectation into expectations on the events 
\begin{align*}
A_1 = \{ \delta |G| \leq {\hat L}^z_{[T]} \leq \theta |G| \}, \, A_2 = \{ {\hat L}^z_{[T]} < \delta |G|\}, \, A_3 = \{ {\hat L}^z_{[T]} > \theta |G| \}, \textrm{ } 0 < \delta < \theta.
\end{align*}
In this way, one obtains the following bound on (\ref{eq:loc1}):
\begin{align}
c(\alpha)|G|^{-1/3} + E^{\mathbb Z}_0 \Bigl[ A_1, \Bigl( \Bigl|\sum_{n=0}^{[T]-1} (\sigma^{\mathsf Z}_{n+1} - \sigma^{\mathsf Z}_n - 1) \mathbf{1}_{\{Z_n=z\}} \Bigr|/|G| \Bigr) \wedge 1 \Bigr] + 2 \delta + P^{\mathbb Z}_0[A_3], \label{eq:loc2}
\end{align}
where we have used the fact that $(\sigma^{\mathsf Z}_{n+1} - \sigma^{\mathsf Z}_n)_{n \geq 0}$ are iid $\exp(1)$ variables independent of $Z$ to bound the expectation on $A_2$ by $2 \delta$. By Chebyshev's inequality and (\ref{loct2}),
\begin{align*}
P^{\mathbb Z}_0[A_3] \leq E^{\mathbb Z}_0 \bigl[ {\hat L}^z_{[\alpha |G|^2]}\bigr]/ (\theta |G|)\leq c(\alpha)/\theta.
\end{align*}
In order to bound the expectation in (\ref{eq:loc2}), we apply Fubini's theorem to obtain
\begin{align*}
&E^{\mathbb Z}_0 \Bigl[ A_1, \Bigl( \Bigl|\sum_{n=0}^{[T]-1} (\sigma^{\mathsf Z}_{n+1} - \sigma^{\mathsf Z}_n - 1) \mathbf{1}_{\{Z_n=z\}} \Bigr|/|G| \Bigr) \wedge 1 \Bigr] \leq E^{\mathbb Z}_0 \biggl[ A_1, f({\hat L}^z_{[T]}) \frac{{\hat L}^z_{[T]}}{|G|}  \biggr],\\
&\textrm{where for any $l \geq 1$, }  f(l) = E^{\mathbb Z}_0 \Bigl[ \Bigl( \Bigl| \sum_{n=0}^{l-1} (\sigma^{\mathsf Z}_{n+1} - \sigma^{\mathsf Z}_n -1) \Bigr| / l \Bigr) \wedge (|G| / l) \Bigr].
\end{align*}
Collecting the above estimates and using the definition of $A_1$, we have found the following bound on the expectation in (\ref{loc1}) for any $z \in {\mathbb Z}$:
\begin{align*}
c(\alpha)|G|^{-1/3} + \theta \sup_{l \geq \delta |G|} f(l) +  2 \delta + \frac{c(\alpha)}{\theta}.
\end{align*}
Note that this expression does not depend on $z$, so it remains unchanged after taking the supremum over all $z \in {\mathbb Z}$.  Since moreover $\sup_{l \geq \delta |G|} f(l)$ tends to $0$ as $|G|$ tends to infinity by the law of large numbers and dominated convergence, this shows that the left-hand side of (\ref{loc1}) (with $\lim$ replaced by $\limsup$) is bounded from above by $2\delta + c(\alpha)/\theta$. The result follows by letting $\delta$ tend to $0$ and $\theta$ to infinity.
\end{proof}

Consider now the times ${\mathsf R}_n$ and ${\mathsf D}_n$, defined as the continuous-time analogs of the times $R_n$ and $D_n$ in (\ref{eq:ex}):  
\begin{align*}
{\mathsf R}_n = \sigma^{\mathsf Z}_{R_n} \textrm{ and } {\mathsf D}_n = \sigma^{\mathsf Z}_{D_n}, \textrm{ for } n \geq 1,
\end{align*}
%\begin{align*}
%&{\mathsf R}_1 = {\mathsf H}_C, {\mathsf D}_1 = {\mathsf T}_O \circ \theta_{{\mathsf R}_1} + {\mathsf R}_1, \textrm{ and for } n \geq 1, \\
%&{\mathsf R}_{n+1} = {\mathsf R}_1 \circ \theta_{{\mathsf D}_n} + {\mathsf D}_n, \textrm{ } {\mathsf D}_{n+1} = {\mathsf D}_1 \circ \theta_{{\mathsf %D}_n} + {\mathsf D}_n. 
%\end{align*}
so that the times ${\mathsf R}_n$ and ${\mathsf D}_n$ coincide with the successive times of return to $C$ and departure from $O$ for the process $\mathsf Z$. We record the following observation:

\begin{lemma} \label{lem:DLLN}
For any sequence $a_N \geq 0$ diverging to infinity,
\begin{align}
&\lim_N \sup_{z \in {\mathbb Z}} E^{\mathbb Z}_z [ |{\mathsf D}_{a_N}/D_{a_N} - 1| \wedge 1 ] = 0. \label{sl0}
\end{align}
\end{lemma}

\begin{proof}
We define the function $g: {\mathbb N} \to {\mathbb R}$ by $g(n) = \sum_{i=1}^{n} (\sigma^{\mathsf Z}_i - \sigma^{\mathsf Z}_{i-1}) /n$, so that ${\mathsf D}_{a_N}/D_{a_N} = g(D_{a_N})$. By independence of the two sequences $(\sigma^{\mathsf Z}_n)_{n \geq 1}$ and $(D_n)_{n \geq 1}$, Fubini's theorem yields
\begin{align}
\sup_{z \in {\mathbb Z}} E^{\mathbb Z}_z [ |{\mathsf D}_{a_N}/D_{a_N} - 1| \wedge 1 ] &= \sup_{z \in {\mathbb Z}} E^{\mathbb Z}_z \bigl[ E^{\mathbb Z}_0 [|g(n) -1| \wedge 1] \bigl|_{n=D_{a_N}} \bigr], \label{eq:DLLN}
\end{align}
where we have used that the distribution of $(\sigma^{\mathsf Z}_n)_{n \geq 1}$ is the same under all measures $P^{\mathbb Z}_z$, $z \in {\mathbb Z}$. Fix any $\epsilon > 0$. By the law of large numbers, the $E^{\mathbb Z}_0$-expectation in (\ref{eq:DLLN}) is less than $\epsilon$ for all $n \geq c(\epsilon)$. Hence, for any $N$ such that $c(\epsilon) \leq a_N$, we have $c(\epsilon) \leq a_N \leq D_{a_N}$ and the expression in (\ref{eq:DLLN}) is less than $\epsilon$.
\end{proof}

\section{Excursions are almost independent} \label{sec:ind}

The purpose of this section is to derive an estimate on the continuous-time excursions $({\mathsf X}_{[{\mathsf R}_k,{\mathsf D}_k]})_{1 \leq k \leq k_*}$ between $C$ and the complement of $O$. The main result is Lemma~\ref{lem:tv}, showing that these excursions can essentially be replaced by independent excursions after conditioning on the $\mathbb Z$-projections of the successive return and departure points. The reason is that the $G_N$-component of $\mathsf X$ has enough time to mix and become close to uniformly distributed between every departure and subsequent return, thanks to the choice of $h_N$ in the definition of the grids ${\mathcal G}_N$, see (\ref{s1}). The following estimate is the crucial ingredient:

\begin{proposition} \label{pr:st}
\begin{align}
\sup_{y,y' \in G_N} \biggl| q^{G_N}_t(y,y') - \frac{1}{|G_N|} \biggr| \leq e^{-\lambda_N t}, \textrm{ for } t \geq 0.\label{st}
\end{align}
\end{proposition}

\begin{proof}
If $w_y=1$ for all $y \in G$, then the statement is immediate from \cite{SC}, Corollary 2.1.5, page 328. As we now show, the argument given in \cite{SC} extends to the present context. For any $|G| \times |G|$ matrix $A$ and real-valued function $f$ on $G$, we define the function $Af$ by
\begin{align*}
 Af (y) = \sum_{y' \in G} A_{y,y'} f(y').
\end{align*}
We define the matrices $K$ and $W$ by $K_{y,y'}= p^G(y,y')$ and $W_{y,y'} = w_y \delta_{y=y'}$, for $y$, $y' \in G$. Then we claim that for any real-valued function $f$ on $G$,
\begin{align}
\label{st0} E_y[f({\mathsf Y}_t)] = H_t f (y), \textrm{ where $H_t = e^{-tW(I-K)}$, $t \geq 0$}.
\end{align}
In words, this claim asserts that the infinitesimal generator matrix $Q$ of the Markov chain $({\mathsf Y}_t)_{t \geq 0}$ is given by $Q= -W(I-K)$, an elementary fact that is proved in \cite{N97}, Theorem~2.8.2, p.~94. Recall the definition of the Dirichlet form $\mathcal D$ from \eqref{def:dir}. Let us also define the inner product of real-valued functions $f$ and $g$ on $G$ by
\begin{align*}
 \left\langle f, g \right\rangle = \sum_{y \in G} f(y)g(y) |G|^{-1}.
\end{align*}
Then elementary computations show that
\begin{align*}
 \frac{d}{dt} \mu ((H_t f)^2) = - 2 \left\langle W(I-K) H_tf, H_t f \right\rangle = - 2 {\mathcal D} (H_t f, H_t f).
\end{align*}
This equation implies that the function $u$, defined by $u(t) = \textup{var}_\mu (H_t f)$, $t \geq 0$, satisfies
\begin{align*}
 u'(t) = - 2 {\mathcal D} (H_t (f-\mu(f)), H_t (f-\mu(f))) \stackrel{\eqref{def:spec}}{\leq} - 2 \lambda_N u(t), \, t \geq 0,
\end{align*}
hence by integration of of $u'/u$,
\begin{align}
 \label{st1} \textup{var}_\mu (H_t f) = u(t) \leq e^{-2 \lambda_N t} u(0) = e^{-2 \lambda_N t} \textup{var}_\mu (f).
\end{align}
Using symmetry of $q^G_t(.,.)$, \eqref{st0} and  the Cauchy-Schwarz inequality for the first estimate, we obtain for any $t \geq 0$ and $y,y' \in G$,
\begin{align*}
 \bigl| |G| q^G_t (y,y') - 1 \bigr| &= \biggl| \sum_{y'' \in G} \Bigl( |G| q_{t/2}^G(y,y'') - 1\Bigr) \Bigl( |G| q_{t/2}^G(y'',y') -1 \Bigr) \frac{1}{|G|} \biggr|\\
&\leq %\mu \Bigl( \bigl( |G| q_{t/2}^G(.,y) - 1 \bigr)^2 \Bigr)^{1/2} \mu \Bigl( \bigl( |G| q_{t/2}^G(.,y') - 1 \bigr)^2 \Bigr)^{1/2} \\
%&= 
\textup{var}_\mu \bigl( H_{t/2} |G| \delta_y (.) \bigr)^{1/2} \textup{var}_\mu \bigl( H_{t/2} |G| \delta_{y'} (.) \bigr)^{1/2}\\
&\stackrel{\eqref{st1}}{\leq} e^{- \lambda_N t} \textup{var}_\mu \bigl( |G| \delta_y (.) \bigr)^{1/2} \textup{var}_\mu \bigl( |G| \delta_{y'} (.) \bigr)^{1/2} \\
&=e^{-\lambda_N t} (|G|-1).
\end{align*}
Dividing both sides by $|G|$, we obtain \eqref{st}.
\end{proof}

Next, we show that the time between any departure and successive return indeed is typically much longer than the relaxation time $\lambda_N^{-1}$ of $\mathsf Y$:

\begin{lemma} \label{lem:mix}
\begin{align}
\limsup_N |G_N|^{-\epsilon/16} \log \sup_{k \geq 2} P^{\mathbb Z}_0 [{\mathsf R}_k-{\mathsf D}_{k-1} \leq \lambda_N^{-1} |G_N|^\epsilon ] <0. \label{mix}
\end{align}
\end{lemma}

\begin{proof}
By (\ref{s1}), we may assume that $N$ is large enough so that $d_N < h_N/2$. We put $$\gamma = 2 \lambda_N^{-1} |G_N|^{\epsilon/8},$$ so that $\gamma$ diverges as $N$ tends to infinity (see below \ref{hyp:w}), and define the stopping times $({\mathsf U}_n)_{n \geq 1}$ as the times of successive displacements of $\mathsf Z$ at distance $[\sqrt{\gamma}]$, i.e.
\begin{align*}
{\mathsf U}_1 &= \inf \{ t \geq 0: |{\mathsf Z}_t - {\mathsf Z}_0| \geq [\sqrt{\gamma}]\}, \textrm{ and for } n \geq 2, \\
{\mathsf U}_n &= {\mathsf U}_1 \circ \theta_{{\mathsf U}_{n-1}} + {\mathsf U}_{n-1}.
\end{align*}
To get from a point in $O^c$ to $C$, $\mathsf Z$ has to travel a distance of at least $h_N/2 \geq [h_N/(2\sqrt{\gamma})][\sqrt{\gamma}]$. As a consequence, ${\mathsf R}_k- {\mathsf D}_{k-1} \geq {\mathsf U}_{[h_N/(2\sqrt{\gamma})]} \circ \theta_{{\mathsf D}_{k-1}}$ and it follows from the strong Markov property applied at time ${\mathsf D}_{k-1}$, then inductively at the times ${\mathsf U}_{[h_N/(2\sqrt{\gamma})]-1}, \ldots, {\mathsf U}_1$ that
\begin{align}
P^{\mathbb Z}_0 [{\mathsf R}_k- {\mathsf D}_{k-1} \leq \gamma] \leq e E^{\mathbb Z}_0 [ \exp \{- {\mathsf U}_{[h_N/(2\sqrt{\gamma})]}/\gamma\}] \leq e \bigl( E^{\mathbb Z}_0 [ \exp \{-{\mathsf U}_1 / \gamma\}] \bigr)^{[h_N/(2\sqrt{\gamma})]}. \label{eq:mix1}
\end{align}
Since ${\mathsf U}_1 = {\mathsf T}_{(-[\sqrt{\gamma}], [\sqrt{\gamma}])} = \sigma^{\mathsf Z}_{T_{(-[\sqrt{\gamma}], [\sqrt{\gamma}])}}$, we find with independence of $(\sigma^{\mathsf Z}_n)_{n \geq 0}$ and $T_{(-[\sqrt{\gamma}], [\sqrt{\gamma}])}$,
\begin{align*}
E^{\mathbb Z}_0 [ \exp \{-{\mathsf U}_1 / \gamma\}] =  E^{\mathbb Z}_0 \bigl[ (1-1/\gamma)^{T_{(-[\sqrt{\gamma}], [\sqrt{\gamma}])}} \bigr],
\end{align*}
by computing the moment generating function of the $\Gamma(n,1)$-distributed variable $\sigma^{\mathsf Z}_n$. By the invariance principle, the last expectation is bounded from above by $1-c$ for some constant $c>0$. Inserting this bound into (\ref{eq:mix1}) and using the bound $h_N \geq c \sqrt{\gamma} |G_N|^{\epsilon/16}$ from (\ref{s1}), we find (\ref{mix}). 
\end{proof}

We finally come to the announced result, which is similar to Proposition 3.3 in \cite{S08}. We introduce, for $\mathcal G$ any one of the graphs $G_N$, $\mathbb Z$ or $G_N \times {\mathbb Z}$, the spaces ${\mathcal P}({\mathcal G})^f$ of right-continuous functions from $[0, \infty)$ to $\mathcal G$ with finitely many discontinuities, endowed with the canonical $\sigma$-algebras generated by the finite-dimensional projections. The measurable functions $(.)_{s_0}^{s_1}$ from ${\mathcal P}({\mathcal G})$ to ${\mathcal P}({\mathcal G})^f$ are defined for $0 \leq s_0 < s_1$ by 
\begin{align}
(({\mathsf w})_{s_0}^{s_1})_t = {\mathsf w}_{(s_0 + t) \wedge s_1}, \, t \geq 0. \label{def:path}
\end{align}
Given $z \in C$ and $z'$ with $P_z[{\mathsf Z}_{{\mathsf D}_1} = z'] >0$, for $P_z$ defined in (\ref{def:pr2}) (in other words $z' \in \partial {\tilde I}$ if $\partial {\tilde I}$ is the connected component of $O$ containing $z$), we set
\begin{align}
P_{z,z'} = P_z[.|{\mathsf Z}_{{\mathsf D}_1} = z']. \label{def:el}
\end{align}

\begin{lemma} \label{lem:tv}
For any measurable functions $f_k: {\mathcal P}(G_N)^f \times {\mathcal P}({\mathbb Z})^f \to [0,1]$, $1 \leq k \leq k_*$,
\begin{align}
\lim_N \Bigl| E \Bigl[ \prod_{1 \leq k \leq k_*} f_k (({\mathsf X})_{{\mathsf R}_k}^{{\mathsf D}_k} ) \Bigr] - E^{\mathbb Z}_0 \Bigl[ \prod_{1 \leq k \leq k_*} E_{{\mathsf Z}_{{\mathsf R}_k}, {\mathsf Z}_{{\mathsf D}_k}} [f_k (({\mathsf X})_{0}^{{\mathsf D}_1})] \Bigr] \Bigr| = 0. \label{tv}
\end{align}
\end{lemma}

\begin{proof}[Proof of Lemma~\ref{lem:tv}.]
Consider first arbitrary measurable functions $g_k: {\mathcal P}(G)^f \to [0,1]$, $1 \leq k \leq k_*$, real numbers $0 \leq s_1 < s'_1 < \ldots < s_{k_*} < s'_{k_*} < \infty$ and set $$H_k = g_k(({\mathsf Y})_{s_k}^{s'_k}) %= g_k(({\mathsf Y} \circ \theta_{s_k} )_0^{s'_k - s_k})
.$$ With the simple Markov property applied at time $s_{k_*}$, then at time $s_{k_*-1}$, one obtains
\begin{align*}
&E^G \Bigl[ \prod_{1 \leq k \leq k_*} H_k \Bigr] = E^G \Bigl[ \Bigl( \prod_{1 \leq k \leq k_*-1} H_k \Bigr) E^G_{{\mathsf Y}_{s_{k_*}}} [g_{k_*}(({\mathsf Y})_0^{s'_{k_*} - s_{k_*}}) ]  \Bigr] \\ 
&\qquad = E^G \biggl[ \Bigl( \prod_{1 \leq k \leq k_*-1} H_k \Bigr) \sum_{y \in G} q^G_{s_{k_*}-s'_{k_*-1}}({\mathsf Y}_{s_{k_*-1}},y) \biggr] E^G_y [g_{k_*}(({\mathsf Y})_0^{s'_{k_*} - s_{k_*}}) ]  .
\end{align*}
With the estimate (\ref{st}) on the difference between the transition probability of $\mathsf Y$ inside the expectation and the uniform distribution and the fact that $g_k \in [0,1]$, it follows that 
\begin{align*}
\biggl| E^G \Bigl[ \prod_{1 \leq k \leq k_*} H_k \Bigr] - E^G \Bigl[  \prod_{1 \leq k \leq k_*-1} H_k \Bigr] E^G [g_{k_*}(({\mathsf Y})_0^{s'_{k_*} - s_{k_*}}) ] \biggr| \leq  c |G|\exp \{-(s_{k_*} - s'_{k_*-1}) \lambda_N \}.
\end{align*}
By induction, we infer that
\begin{align}
\biggl| E^G \Bigl[ \prod_{1 \leq k \leq k_*} g_k(({\mathsf Y})_{s_k}^{s'_k}) \Bigr] - \prod_{1 \leq k \leq k_*} E^G [g_{k}(({\mathsf Y})_0^{s'_k - s_k}) ] \biggr| \leq c |G| \sum_{2 \leq k \leq k_*} e^{-(s_k-s'_{k-1}) \lambda_N}. \label{eq:tv1}
\end{align}
Let us now consider the first expectation in (\ref{tv}). By Fubini's theorem, we find that
\begin{align*}
E \Bigl[ \prod_{1 \leq k \leq k_*} f_k (({\mathsf X})_{{\mathsf R}_k}^{{\mathsf D}_k}) \Bigr] = E^{\mathbb Z}_0 \Bigl[ E^G \Bigl[ \prod_{1 \leq k \leq k_*} f_k (({\mathsf Y})_{s_k}^{s'_k}, ({\bar z})_{s_k}^{s'_k}) \Bigr] \bigr|_{({\bar z})_{s_k}^{s'_k} =  ({\mathsf Z})_{{\mathsf R}_k}^{{\mathsf D}_k}}\Bigr].
\end{align*}
Observe that (\ref{eq:tv1}) applies to the $E^G$-expectation with $g_k(.) = f_k (., ({\bar z})_{s_k}^{s'_k})$, and yields
\begin{align}
&\biggl| E \Bigl[ \prod_{1 \leq k \leq k_*} f_k (({\mathsf X})_{{\mathsf R}_k}^{{\mathsf D}_k}) \Bigr] - E^{\mathbb Z}_0 \Bigl[ \prod_{1 \leq k \leq k_*} E^G \bigl[  f_k (({\mathsf Y})_0^{s'_k-s_k}, ({\mathsf Z})_{{\mathsf R}_k}^{{\mathsf D}_k}) \bigr] \Bigr] \biggr| \label{eq:tv2}\\
&\qquad \leq c |G|\sum_{2 \leq k \leq k_*} E^{\mathbb Z}_0 [e^{-({\mathsf R}_k-{\mathsf D}_{k-1}) \lambda_N}]. \nonumber
\end{align}
Note that for large $N$, the last term can be bounded with the estimate (\ref{mix}) on ${\mathsf R}_k-{\mathsf D}_{k-1}$:
\begin{align}
\sum_{2 \leq k \leq k_*} E^{\mathbb Z}_0 [e^{-({\mathsf R}_k-{\mathsf D}_{k-1})\lambda_N}]  \leq ck_* \exp\{-c'|G|^{c\epsilon}\} \stackrel{(\ref{def:sigma})}{\leq} c(\alpha) |G|^c \exp\{-c'|G|^{c\epsilon}\}. \label{eq:tv3}
\end{align}
It thus only remains to show that the second expectation on the left-hand side of (\ref{eq:tv2}) is equal to the second expectation in (\ref{tv}). Note that for any measurable functions $h_k: {\mathcal P}({\mathbb Z})^f \to [0,1]$, $1 \leq k \leq k_*$ and points $z_1, \ldots, z_{k_*}$, $z'_1, \ldots, z'_{k_*}$ in $\mathbb Z$ such that $P^{\mathbb Z}_{z_k}[{\mathsf Z}_{{\mathsf D}_1} = z'_k] > 0$ for $1 \leq k \leq k_*$, one has by two successive inductive applications of the strong Markov property at the times ${\mathsf R}_{k_*}, {\mathsf D}_{k_*-1}, {\mathsf R}_{k_*-1}, \ldots, {\mathsf D}_1$, with the convention $P_{z'_0} = P$,
\begin{align*}
&E^{\mathbb Z}_0 \Bigl[ \bigcap_{1 \leq k \leq k_*} \{{\mathsf Z}_{{\mathsf R}_k} = z_k, {\mathsf Z}_{{\mathsf D}_k} = z'_k\}, \prod_{1 \leq k \leq k_*} h_k(({\mathsf Z})_{{\mathsf R}_k}^{{\mathsf D}_k}) \Bigr] \\
&\qquad = \prod_{1 \leq k \leq k_*} \Bigl( P^{\mathbb Z}_{z'_{k-1}} [{\mathsf Z}_{{\mathsf R}_1} = z_k] E_{z_k,z'_k} [h_k(({\mathsf Z})_0^{{\mathsf D}_1})] P^{\mathbb Z}_{z_k}[{\mathsf Z}_{{\mathsf D}_1} = z'_k] \Bigr) \\
&\qquad = P^{\mathbb Z}_0 \Bigl[\bigcap_{1 \leq k \leq k_*} \{{\mathsf Z}_{{\mathsf R}_k} = z_k, {\mathsf Z}_{{\mathsf D}_k} = z'_k\}\Bigr]  \prod_{1 \leq k \leq k_*} E_{z_k,z'_k} [h_k(({\mathsf Z})_0^{{\mathsf D}_1})]. 
\end{align*}
Summing this last equation over all $z_k$, $z'_k$ as above, one obtains
\begin{align*}
E^{\mathbb Z}_0 \Bigl[ \prod_{1 \leq k \leq k_*} h_k(({\mathsf Z})_{{\mathsf R}_k}^{{\mathsf D}_k}) \Bigr] = E^{\mathbb Z}_0 \Bigl[ \prod_{1 \leq k \leq k_*} E_{{\mathsf Z}_{{\mathsf R}_k}, {\mathsf Z}_{{\mathsf D}_k}} [ h_k(({\mathsf Z})_0^{{\mathsf D}_1})] \Bigr].
\end{align*}
Applying this equation with $$h_k(({\mathsf Z})_{{\mathsf R}_k}^{{\mathsf D}_k}) =  E^G \bigl[  f_k (({\mathsf Y})_0^{s'_k-s_k}, ({\bar z})_{s_k}^{s'_k}) \bigr] \Bigr|_{({\bar z})_{s_.}^{s'_.} = ({\mathsf Z})_{{\mathsf R}_k}^{{\mathsf D}_k}},$$ substituting the result into (\ref{eq:tv2}) and remembering (\ref{eq:tv3}), we have shown (\ref{tv}).
\end{proof}

\section{Proof of the result in continuous time} \label{sec:c}

The purpose of this section is to prove in Theorem~\ref{thm:c} the continuous-time version of Theorem~\ref{thm:d}. Let us explain the role of the crucial estimates appearing in Lemmas~\ref{lem:ch} and \ref{lem:cap2}. Under the assumptions \ref{hyp:w}-\ref{hyp:k}, these lemmas exhibit the asymptotic behavior of the $P_{z,z'}$-probability (see~(\ref{def:el})) that an excursion of the path $\mathsf X$ visits vertices in the neighborhoods of the sites $x_m$ contained in a box $G_N \times I$. It is in particular shown that the probability that a set $V_m$ in the neighborhood of $x_m$ is visited equals $\textup{cap}^m(\Phi_m(V_m)) h_N/|G_N|$, up to a multiplicative factor tending to $1$ as $N$ tends to infinity. This estimate is similar to a more precise result proved by Sznitman for $G_N=({\mathbb Z}/N{\mathbb Z})^d$ in Lemma~1.1 of \cite{S08b}, where an identity is obtained for the same probability, if the distribution of the starting point of the excursion is the uniform distribution on the boundary of $G_N \times {\tilde I}$ (rather than the uniform distribution on $G_N \times \{z\}$).

\vspace{2mm}

According to the characterization (\ref{def:int}), these crucial estimates show that the law of the vertices in the neighborhood of $x_m$ not visited by such an excursion is comparable to ${\mathbb Q}^{{\mathbb G}_m \times {\mathbb Z}}_{h_N/|G_N|}$. 
In Lemma~\ref{lem:tv} of the previous section, we have seen that different excursions of the form $({\mathsf X})_{{\mathsf R}_k}^{{\mathsf D}_k}$, conditioned on the entrance and departure points of the $\mathbb Z$-projection, are close to independent for large $N$. According to the observation outlined in the last paragraph, the level of the random interlacement appearing in the neighborhood of $x_m$ at time $\alpha |G_N|^2$ is hence approximately equal to $h_N/|G_N|$ times the number of excursions to the interval $I$ performed until time $\alpha |G_N|^2$. As we have seen in Proposition~\ref{pr:sp} and Lemma~\ref{lem:loc}, this quantity is close to the local time ${\hat L}^{z_m}_{\alpha |G_N|^2}/|G_N|$ for large $N$. An invariance principle for local times due to R\'ev\'esz \cite{R81} (with assumption \ref{hyp:zconv}) serves to identify the limit of this quantity, hence the level of the random interlacement appearing in the large $N$ limit, as $L(v_m,\alpha)$. This strategy will yield the following result:
 
\begin{theorem}\label{thm:c}
Assume that \textup{\ref{hyp:w}-\ref{hyp:k}} are satisfied. Then the graphs ${\mathbb G}_m \times {\mathbb Z}$ are transient and as $N$ tends to infinity, the $\prod_{m=1}^M \{0,1\}^{{\mathbb G}_m} \times {\mathbb R}_+^M$-valued random variables
\begin{align*}
\Bigl( {\omega}^{1,N}_{\eta^{\mathsf X}_{\alpha |G_N|^2}}, \ldots, {\omega}^{M,N}_{\eta^{\mathsf X}_{\alpha |G_N|^2}}, \frac{{\mathsf L}^{z_1}_{\alpha |G_N|^2}}{|G_N|}, \ldots, \frac{{\mathsf L}^{z_M}_{\alpha |G_N|^2}}{|G_N|} \Bigr),  \, \alpha >0,
\end{align*}
defined by (\ref{def:pic}), (\ref{def:loct}), with ${\mathsf r}_N$ and $\phi_{m,N}$ chosen in (\ref{eq:rs}) and (\ref{def:iso}), converge in joint distribution under $P$ to the law of the random vector
$(\omega_1, \ldots, \omega_M, U_1, \ldots, U_m)$ with the following distribution: $(U_m)_{m=1}^M$ is distributed as $(L(v_m,\alpha))_{m=1}^M$ under $W$, and conditionally on $(U_m)_{m=1}^M$, the random variables $(\omega_m)_{m=1}^M$ have joint distribution $\prod_{1 \leq m \leq M} {\mathbb Q}_{U_m}^{{\mathbb G}_m \times {\mathbb Z}}$.
\end{theorem}

\begin{proof}
The transience of the graphs ${\mathbb G}_m \times {\mathbb Z}$ is an immediate consequence of Lemma~\ref{lem:mntrans}. To define the local pictures in (\ref{def:pic}), we choose the ${\mathsf r}_N$ in (\ref{iso0}) as
\begin{align}
&{\mathsf r}_N = \Bigl(\min_{1 \leq m < m' \leq M} d(x_{m,N},x_{m',N}) \wedge r_N \wedge d_N \Bigr)/3 , \textrm{ cf.~\ref{hyp:dist}, \ref{hyp:iso1}, (\ref{s1})}\label{eq:rs} \\
&\textrm{and $\phi_{m,N}$ as the restriction of the isomorphism in \ref{hyp:iso1} to $B(y_{m,N},{\mathsf r}_N)$.} \label{def:iso}
\end{align}
Then the local pictures in (\ref{def:pic}) are defined. We set 
\begin{align}
B_{m,N} = B(x_{m,N},{\mathsf r}_N-1) \textrm{ and } {\mathbb B}_{m,N} = \Phi_{m,N}(B_{m,N}), \textrm{ for } {\mathsf r}_N \geq 1. \label{def:balls}
\end{align}
From now on, we drop $N$ from the notation in $\phi_{m,N}$, $B_{m,N}$ and ${\mathbb B}_{m,N}$ for simplicity.
Our present task is to show that for arbitrarily chosen finite subsets ${\mathbb V}_m$ of ${\mathbb G}_m \times {\mathbb Z}$,
\begin{align}
A_N(\alpha|G_N|^2,\alpha|G_N|^2) \to A(\alpha), \textrm{ for any $\theta_m \in {\mathbb R}_+$, $1 \leq m \leq M$.} \label{eq:c0}
\end{align}
where for times $s$, $s' \geq 0$ and $V_m = \Phi_m^{-1} {\mathbb V}_m$ (well-defined for large $N$, see (\ref{def:Phi})),
\begin{align}
A_N(s,s') &= E \Bigl[ \prod_{1 \leq m \leq M} \mathbf{1}_{\{{\mathsf H}_{V_m} > s\}} \exp \Bigl\{ - \frac{\theta_m}{|G_N|} {\mathsf L}^{z_m}_{s'} \Bigr\} \Bigr], \textrm{ and} \label{eq:AN} \\
A(\alpha) &= E^W \Bigl[ \exp \Bigl\{ - \sum_{1 \leq m \leq M}  L(v_m,\alpha) (\textup{cap}^m({\mathbb      V}_m) + \theta_m)  \Bigr\} \Bigr]. \label{eq:A}
\end{align}
Theorem~\ref{thm:c} then follows, as a result of the equivalence of weak convergence and convergence of Laplace transforms (see for example \cite{C74}, p.~189-191), the compactness of the set of probabilities on $\prod_m \{0,1\}^{{\mathbb G}_m \times {\mathbb Z}}$, and the fact that the canonical product $\sigma$-algebra on $\prod_m \{0,1\}^{{\mathbb G}_m \times {\mathbb Z}}$ is generated by the $\pi$-system of events $\cap_{m=1}^M \{\omega(x) = 1, \textrm{ for all } x \in {\mathbb V}_m\}$, with ${\mathbb V}_m$ varying over finite subsets of ${\mathbb G}_m \times {\mathbb Z}$.

\vspace{2mm}

We first introduce some additional notation and state some inclusions we shall use. For any interval $I \in {\mathcal I}$ (cf.~(\ref{def:I})), we denote by ${\mathcal J}_I$ the set of indices $m$ such that $z_m \in I$:
\begin{align}
&{\mathcal J}_I = \{1 \leq m \leq M: z_{m,N} \in I\} \, /= \emptyset \textrm{ if no $z_{m,N}$ belongs to } I. \label{def:J}
\end{align}
Note that the set ${\mathcal J}_I$ depends on $N$. Indeed, so does the labelling of the intervals $I_l$ in ${\mathcal I}$. It follows from the definition of ${\mathsf r}_N$ that
\begin{align}
&\textrm{the balls } ({\bar B}_m)_{1 \leq m \leq M} \textrm{ are disjoint, cf.~(\ref{def:balls}).} \label{eq:Bdisj}
\end{align}
Since the sets ${\mathbb V}_m$ are finite, we can choose a parameter $\kappa>0$ such that ${\mathbb V}_m \subset B((o_m,0),\kappa)$ for all $m$ and $N$. 
Since ${\mathsf r}_N$ tends to infinity with $N$, there is an $N_0 \in {\mathbb N}$ such that for all $N \geq N_0$, we have ${\mathsf r}_N \geq 1$ as well as for all $I \in {\mathcal I}$ and $m \in {\mathcal J}_I$,
\begin{align}
\begin{array}{lllllllll}
{\mathbb V}_m & \subset & B((o_m,0),\kappa) & \subset & {\mathbb B}_m & \subset & B(o_m,{\mathsf r}_N-1) \times {\mathbb Z}& &  \\
\downarrow \Phi_m^{-1} & & \downarrow \Phi_m^{-1} & & \downarrow \Phi_m^{-1} & & & & \\
V_m  & \subset & B(x_m,\kappa) & \subset & B_m & \stackrel{(\ref{eq:rs})}{\subset} & B(y_m,{\mathsf r}_N-1) \times I & \stackrel{\ref{hyp:inc}}{\subseteq} & C_m \times I.
\end{array} \label{eq:inc}
\end{align}
Since $d_N=o(|G_N|)$ (cf.~(\ref{s1}), \ref{hyp:mix}), any two sequences $z_m$ that are contained in the same interval $I \in {\mathcal I}$ infinitely often, when divided by $|G_N|$, must converge to the same number $v_m$, cf.~(\ref{hyp:zconv}). By \ref{hyp:disj}, we can hence increase $N_0$ if necessary, such that for all $N \geq N_0$,
\begin{align}
\textrm{for $m$ and $m'$ in ${\mathcal J}_I$, either $C_m = C_{m'}$ or $C_m \cap C_{m'} = \emptyset$.} \label{eq:disj}
\end{align}
We use $V_{I,m}$ to denote the union of all sets $V_{m'}$ included in $C_m \times I$ and $V_I$ for the union of all $V_m$ included in $G_N \times I$, i.e.
\begin{align}
V_{I,m} = \bigcup_{m' \in {\mathcal J}_I: C_{m'}=C_m} V_{m'} \subset C_m \times I, \textrm{ and } V_I = \bigcup_{m \in {\mathcal J}_I} V_m \stackrel{(\ref{eq:inc})}{\subset} G_N \times I, \label{def:VI}
\end{align}
with the convention that the union of no sets is the empty set.
\vspace{2mm}

The proof of (\ref{eq:c0}) uses three additional Lemmas that we now state. 
The first two lemmas show that the probability that the continuous-time random walk $\mathsf X$ started from the boundary of $G_N \times I$ hits a point in the set $V_I \subset G_N \times I$ (cf.~(\ref{def:VI})) before exiting $G \times {\tilde I}$ behaves like $h_N/|G_N|$ times the sum of the capacities of those sets ${\mathbb V}_m$ whose preimages under $\Phi_m$ are subsets of $G_N \times I$. 

\begin{lemma} \label{lem:ch}
Under \ref{hyp:w}-\ref{hyp:k}, for $N \geq N_0$ (cf.~(\ref{eq:inc}), (\ref{eq:disj})), any $I \in {\mathcal I}$, $I \subset {\tilde I} \in {\tilde {\mathcal I}}$, $z_1 \in \partial( I^c)$ and $z_2 \in \partial {\tilde I}$, 
\begin{align}
&1-c \frac{d_N}{h_N} \leq P_{z_1,z_2} \bigl[ {\mathsf H}_{V_I} < {\mathsf T}_{{\tilde B}} \bigr] \Bigl( \frac{h_N}{|G_N|} \textup{ cap}_{{\tilde B}}( V_I) \Bigr)^{-1} \leq 1+c \frac{d_N}{h_N}, \label{cap1}\\ 
& \textrm{ where } {\tilde B} = G_N \times {\tilde I} \textrm{ and } \textup{ cap}_{{\tilde B}}( V_I) = \sum_{x \in V_I} P_x[T_{\tilde B} < {\tilde H}_{V_I}]  w_x.  \nonumber
\end{align}
\end{lemma}

\begin{lemma} \label{lem:cap2}
With the assumptions and notation of Lemma~\ref{lem:ch},
\begin{align}
\lim_N \max_{I \in {\mathcal I}} \Bigl| \textup{ cap}_{\tilde B}(V_I) - \sum_{m \in {\mathcal J}_I} \textup{cap}^m({\mathbb      V}_m) \Bigr| = 0. \label{cap2}
\end{align}
\end{lemma}

\noindent The next lemma allows to disregard the the effect of the random walk trajectory until time ${\mathsf D}_1$, cf.~(\ref{sl2}), as well as the difference between ${\mathsf D}_{k^*}$ and ${\mathsf D}_{k_*}$, cf.~(\ref{sl3}).

\begin{lemma}\label{lem:sl}
Assuming \ref{hyp:w},
\begin{align}
&\lim_N \sup_{z \in {\mathbb Z}, x \in G_N \times {\mathbb Z}} P_z[{\mathsf H}_x \leq  {\mathsf D}_{k^*-k_*}] =0. \label{sl1}\\
&\lim_N \sup_{z \in {\mathbb Z}} P_z[{\mathsf H}_{\cup_{I} V_I} \leq  {\mathsf D}_1]=0. \label{sl2}\\
&\lim_N E \Bigl[ \Bigl| \prod_{1 \leq m \leq M} 1_{\{{\mathsf H}_{V_m} >  {\mathsf D}_{k^*}\}} - \prod_{1 \leq m \leq M} 1_{\{{\mathsf H}_{V_m} > {\mathsf D}_{k_*}\}} \Bigr| \Bigr]=0. \label{sl3}
\end{align}
\end{lemma}

\noindent Before we prove Lemmas~\ref{lem:ch}-\ref{lem:sl}, we show that they allow us to deduce Theorem~\ref{thm:c}. Throughout the proof, we set $T= \alpha |G_N|^2$ and say that two sequences of real numbers are limit equivalent if their difference tends to $0$ as $N$ tends to infinity. We first claim that in order to show (\ref{eq:c0}), it is sufficient to prove that
\begin{align}
A'_N = A_N({\mathsf D}_{k_*},T) \to A(\alpha), \textrm{ for } \alpha > 0.  \label{eq:suff0}
\end{align}
Indeed, by (\ref{sl3}), the statement (\ref{eq:suff0}) implies that also
\begin{align}
\lim_N A_N({\mathsf D}_{k^*},T) = A(\alpha), \textrm{ for } \alpha > 0.  \label{eq:suff1}
\end{align}
Now recall that $D_{k_*} \leq T \leq D_{k^*}$ with probability tending to $1$ by (\ref{sp1}). Together with (\ref{sl0}), it follows that
\begin{align*}
\lim_N P^{\mathbb Z}_0 [(1-\delta){\mathsf D}_{k_*} \leq T \leq (1+\delta) {\mathsf D}_{k^*} ] =1, \textrm{ for any } \delta>0.
\end{align*}
Monotonicity in both arguments of $A_N(.,.)$, (\ref{eq:suff0}) and (\ref{eq:suff1}) hence yield
\begin{align*}
&\limsup_N A_N \bigl( T/(1-\delta) , T/(1-\delta) \bigr) \leq \limsup_N A_N({\mathsf D}_{k_*},T) = A(\alpha) \textrm{ and}\\
&\liminf_N A_N \bigl( T/(1+\delta), T/(1+\delta) \bigr) \geq \liminf_N A_N({\mathsf D}_{k^*},T) = A(\alpha), \textrm{ for } 0 < \delta < 1.
\end{align*}
Replacing $\alpha$ by $\alpha(1-\delta)$ and $\alpha(1+\delta)$ respectively, we deduce that
\begin{align*}
A(\alpha(1+\delta)) \leq \liminf_N A_N(T,T) \leq \limsup_N A_N(T,T) \leq A(\alpha(1-\delta)), 
\end{align*}
for $\alpha > 0$ and $0 < \delta < 1$, from which (\ref{eq:c0}) follows by letting $\delta$ tend to $0$ and using the continuity of $A(.)$. Hence, it suffices to show (\ref{eq:suff0}). By (\ref{loc1}) $A'_N$ is limit equivalent to 
\begin{align}
E \Bigl[  \mathbf{1}_{\cap_{m} \{{\mathsf H}_{V_m} > {\mathsf D}_{k_*}\}} \exp \Bigl\{ - \sum_{1 \leq m \leq M}  \frac{\theta_m}{|G_N|} {\hat L}^{z_m}_{[T]} \Bigr\} \Bigr], \label{eq:tsuff}
\end{align}
which by (\ref{sl2}) remains limit equivalent if the event $\cap_{m} \{{\mathsf H}_{V_m} > {\mathsf D}_{k_*}\}$ is replaced by
\begin{align}
{\mathcal A}  =& \Bigl\{ \textrm{for all } 2 \leq k \leq k_*, \textrm{ whenever } {\mathsf Z}_{{\mathsf R}_k} \in I \textrm{ for some } I \in {\mathcal I}, {\mathsf X}_{[{\mathsf R}_k, {\mathsf D}_k]} \cap V_I = \emptyset \Bigr\}, \textrm{cf.~(\ref{def:I})}. \nonumber
\end{align}
Making use of (\ref{sp2}) and (\ref{sp4}) (together with $Z_{R_k} = {\mathsf Z}_{{\mathsf R}_k}$) we find that $A'_N$ is limit equivalent to
\begin{align}
E \Bigl[ \mathbf{1}_{\mathcal A} \exp \Bigl\{- \sum_{1 \leq l \leq M}  \frac{h_N (\sum_{m \in {\mathcal J}_{I_l}} \theta_m)}{|G_N|} \sum_{1 \leq k \leq k_*} \mathbf{1}_{\{{\mathsf Z}_{{\mathsf R}_k} \in I_l\}} \Bigr\} \Bigr], \textrm{ cf.~(\ref{def:J}).} \label{eq:t1}
\end{align}
Since $h_N = o(|G_N|)$ (cf.~(\ref{s1}), \ref{hyp:mix}), this expectation remains limit equivalent if we drop the $k=1$ term in the second sum. In other words, the expression in (\ref{eq:t1}) is limit equivalent to (recall the notation from (\ref{def:path}))
\begin{align*}
&E \Bigl[ \prod_{k=2}^{k_*} f( ({\mathsf X})_{{\mathsf R}_k}^{{\mathsf D}_k}) \Bigr], \textrm{ with } f: {\mathcal P}(G_N)^f \times {\mathcal P}({\mathbb Z})^f \to [0,1] \textrm{ defined by }\\
&f({\mathsf w}) = \prod_{1 \leq l \leq M} \Bigl( 1 - \mathbf{1}_{\{ {\mathsf w}_0 \in G_N \times I_l \}} \mathbf{1}_{\{ {\mathsf w}_{[0,\infty)} \cap V_{I_l} \neq \emptyset\}} \Bigr)  \exp \Bigl\{  - \frac{h_N (\sum_{m \in {\mathcal J}_{I_l}} \theta_m)}{|G_N|} \mathbf{1}_{\{{\mathsf w}_0 \in G_N \times I_l\}} \Bigr\}.
\end{align*}
By Lemma~\ref{lem:tv} with $f_1 = 1$, $f_k = f$ for $2 \leq k \leq k_*$, $A'_N$ is hence limit equivalent to
\begin{align*}
E^{\mathbb Z}_0 \Bigl[ \prod_{2 \leq k \leq k_*} E_{{\mathsf Z}_{{\mathsf R}_k}, {\mathsf Z}_{{\mathsf D}_k}} [f (({\mathsf X})_{0}^{{\mathsf D}_1})] \Bigr].
\end{align*}
The above expression equals
\begin{align}
&E^{\mathbb Z}_0 \Bigl[ \prod_{\substack{2 \leq k \leq k_* \\ 1 \leq l \leq M}} \Bigl( 1- \mathbf{1}_{\{{\mathsf Z}_{{\mathsf R}_k} \in I_l\}} g_l({\mathsf Z}_{{\mathsf R}_k},{\mathsf Z}_{{\mathsf D}_k})  \Bigr) \exp \Bigl\{ -  \frac{h_N (\sum_{m \in {\mathcal J}_{I_l}} \theta_m)}{|G_N|} \mathbf{1}_{\{{\mathsf Z}_{{\mathsf R}_k} \in I_l\}} \Bigr\} \Bigr],  \label{eq:tmess}\\
&\textrm{where } g_l(z,z') =  P_{z,z'} \bigl[ {\mathsf X}_{[0,{\mathsf D}_1]} \cap V_{I_l} \neq \emptyset \bigr]. \nonumber
\end{align}
From (\ref{cap1}), we know that
\begin{align}
1-c \frac{d_N}{h_N} \leq g_l({\mathsf Z}_{{\mathsf R}_k},{\mathsf Z}_{{\mathsf D}_k})  \Bigl( \frac{h_N}{|G_N|} \textup{ cap}_{G \times {\tilde I}_l}( V_I) \Bigr)^{-1} \leq 1+ c \frac{d_N}{h_N}. \label{eq:cappl}
\end{align}
With the inequality $0 \leq e^{-u} -1+u \leq u^2$ for $u \geq 0$, one obtains that
\begin{align*}
\Bigl| \prod_{\substack{2 \leq k \leq k_* \\ 1 \leq l \leq M}} \Bigl( 1- \mathbf{1}_{\{{\mathsf Z}_{{\mathsf R}_k} \in I_l\}} g  \Bigr) - \prod_{\substack{2 \leq k \leq k_* \\ 1 \leq l \leq M}} \exp \Bigl\{ - \mathbf{1}_{\{{\mathsf Z}_{{\mathsf R}_k} \in I_l\}} g \Bigr\} \Bigr| \leq  \sum_{\substack{2 \leq k \leq k_* \\ 1 \leq l \leq M}} \mathbf{1}_{\{{\mathsf Z}_{{\mathsf R}_k} \in I_l\}} g^2,
\end{align*}
where we have witten $g$ in place of $g_l({\mathsf Z}_{{\mathsf R}_k},{\mathsf Z}_{{\mathsf D}_k})$. The expectation of the right-hand side in the last estimate tends to $0$ as $N$ tends to infinity, thanks to (\ref{eq:cappl}) and (\ref{sp3}). The expression in (\ref{eq:tmess}) thus remains limit equivalent to $A'_N$ if we replace $1- \mathbf{1}_{\{{\mathsf Z}_{{\mathsf R}_k} \in I_l\}} g_l({\mathsf Z}_{{\mathsf R}_k},{\mathsf Z}_{{\mathsf D}_k})$ by $\exp \bigl\{ - \mathbf{1}_{\{{\mathsf Z}_{{\mathsf R}_k} \in I_l\}} g_l({\mathsf Z}_{{\mathsf R}_k},{\mathsf Z}_{{\mathsf D}_k}) \bigr\}$. Using again (\ref{sp3}), together with (\ref{cap2}) and (\ref{eq:cappl}), we may then replace $g_l({\mathsf Z}_{{\mathsf R}_k},{\mathsf Z}_{{\mathsf D}_k})$ by $\frac{h_N}{|G_N|} \sum_{m \in {\mathcal J}_{I_l}} \textup{ cap}^m({\mathbb V}_m)$. We deduce that the following expression is limit equivalent to $A'_N$:
\begin{align*}
E^{\mathbb Z}_0 \Bigl[ \exp \Bigl\{- \sum_{\substack{1 \leq k \leq k_*\\ 1 \leq l \leq M}} \sum_{m \in {\mathcal J}_{I_l}} \frac{h_N}{|G_N|} \mathbf{1}_{\{{\mathsf Z}_{{\mathsf R}_k} \in I_l\}}   \bigl(\textup{cap}^m({\mathbb V}_m) + \theta_m \bigr)  \Bigr\} \Bigr].
\end{align*}
By (\ref{sp4}) and (\ref{sp2}), this expression is also limit equivalent to
\begin{align}
E^{\mathbb Z}_0 \Bigl[ \exp \Bigl\{- \sum_{1 \leq m \leq M} \frac{1}{|G_N|} {\hat L}^{z_m}_{[T]}  \bigl(\textup{cap}^m({\mathbb V}_m) + \theta_m \bigr)  \Bigr\} \Bigr].\label{eq:tdis}
\end{align}
With Proposition~1 in \cite{R81}, one can construct a coupling of the simple random walk $Z$ on $\mathbb Z$ with a Brownian motion on $\mathbb R$ such that for any $\rho > 0$,
\begin{align*}
n^{-1/4 - \rho} \sup_{z \in {\mathbb Z}} \bigl| {\hat L}^z_n - L(z,n) \bigr| \stackrel{n \to \infty}{\longrightarrow} 0, \quad a.s.,
\end{align*}
where $L(.,.)$ is a jointly continuous version of the local time of the canonical Brownian motion.
It follows that (\ref{eq:tdis}), hence $A'_N$ is limit equivalent to
\begin{align}
E^{W} \Bigl[ \exp \Bigl\{- \sum_{1 \leq m \leq M}  \frac{1}{|G_N|} L(z_m,[\alpha |G_N|^2])   \bigl(\textup{cap}^m({\mathbb V}_m) + \theta_m \bigr)  \Bigr\} \Bigr]. \label{eq:tfin}
\end{align}
By Brownian scaling, $L(z_m,[\alpha |G|^2])/|G|$ has the same distribution as $$L(z_m/|G|,[\alpha |G|^2]/|G|^2).$$ Hence, the expression in (\ref{eq:tfin}) converges to $A(\alpha)$ in (\ref{eq:A}) by continuity of $L$ and convergence of $z_m/|G|$ to $v_m$, see~\ref{hyp:zconv}. We have thus shown that $A'_N \to A(\alpha)$ and by (\ref{eq:suff0}) completed the proof of Theorem~\ref{thm:c}.
\end{proof}

We still have to prove Lemmas~\ref{lem:ch}-\ref{lem:sl}. To this end, we first show that the random walk $X$ started at $\partial C_m \times I$ typically escapes from $G_N \times {\tilde I}$ before reaching a point in the vicinity of $x_m$. Here, the upper bound on $h_N$ in (\ref{s1}) plays a crucial role.
\begin{lemma} \label{lem:k}
Assuming \ref{hyp:w}-\ref{hyp:k}, for any fixed vertex ${\mathtt x} = ({\mathtt y},z) \in {\mathbb      G}_m \times {\mathbb Z}$, intervals $I \in {\mathcal I}$, $I \subset {\tilde I} \in {\tilde {\mathcal I}}$ (cf.~(\ref{def:I})) and $z_m \in I$, 
\begin{align}
\lim_N \sup_{y_0 \in \partial (C_m^c), z_0 \in {\mathbb Z}} P_{(y_0,z_0)} [ H_{\Phi^{-1}_m({\mathtt x})} < T_{G_N \times {\tilde I}}] =0. \label{k}
\end{align}
(Note that $\Phi^{-1}_m({\mathtt x})$ is well-defined for large $N$ by \ref{hyp:iso1}.)
\end{lemma}

\begin{proof}[Proof of Lemma~\ref{lem:k}.]
Consider any $x_0 = (y_0,z_0)$ with $y_0 \in \partial (C_m^c)$ and $z_0 \in {\mathbb Z}$.
In order to bound the expectation of $T_{G \times {\tilde I}}$, recall that $T_{\tilde I}$ denotes the exit time of the interval $\tilde I$ by the discrete-time process $Z$, so that $T_{G \times {\tilde I}}$ can be expressed as $T_{\tilde I}$ plus the number of jumps $\mathsf Y$ makes until $T_{\tilde I}$. Since $\mathsf Y$ and $\mathsf Z$, hence $\eta^Y$ and $\sigma^Z_.$, are independent under $P_{x_0}$, this implies with Fubini's theorem and stochastic domination of $\eta^Y$ by the Poisson process $\eta^{c_1}$ (cf.~(\ref{sd2})) that 
\begin{align*}
E_{x_0}[T_{G \times {\tilde I}}] =  E^{\mathbb Z}_{z_0} \bigl[ T_{\tilde I} + E^G_{y_0} [\eta^{\mathsf Y}_{\sigma^{\mathsf Z}_{T_{\tilde I}}} ] \bigr] \leq  E^{\mathbb Z}_{z_0} [ T_{\tilde I} ] + c_1 E^{\mathbb Z}_{z_0}[\sigma^{\mathsf Z}_{T_{\tilde I}}] = (1+c_1) E^{\mathbb Z}_{z_0} [ T_{\tilde I} ] \leq ch_N^2,
\end{align*}
using a standard estimate on one-dimensional simple random walk in the last step. Hence by the Chebyshev inequality and the bound (\ref{s1}) on $h_N$, 
\begin{align*}
P_{x_0}[T_{G \times \tilde I} \geq \lambda_N^{-1} |G|^\epsilon] \leq E_{x_0}[T_{G \times \tilde I}] \lambda_N |G|^{-\epsilon} \leq ch_N^2 \lambda_N |G|^{-\epsilon} \leq  c |G|^{-\epsilon/2}.
\end{align*}
The claim (\ref{k}) thus follows from \ref{hyp:k}.
\end{proof}

\begin{proof}[Proof of Lemma~\ref{lem:ch}.]
With $z_1$, $z_2$ as in the statement, we have by the strong Markov property applied at the hitting time of $V_I \subset G \times I$ (cf.~(\ref{eq:inc})),
\begin{align*}
&P_{z_1, z_2} [{\mathsf H}_{V_I} < {\mathsf T}_{{\tilde B}}] = P_{z_1} [ {\mathsf H}_{V_I} < {\mathsf T}_{\tilde B}, {\mathsf Z}_{{\mathsf T}_{\tilde B}} = z_2] / P^{\mathbb Z}_{z_1}[{\mathsf Z}_{T_{\tilde I}} = z_2] \\
&\qquad = E_{z_1} \bigl[{\mathsf H}_{V_I} < {\mathsf T}_{\tilde B}, P^{\mathbb Z}_{{\mathsf Z}_{{\mathsf H}_{V_I}}}[ {\mathsf Z}_{{\mathsf T}_{\tilde I}} = z_2] \bigr] / P^{\mathbb Z}_{z_1}[{\mathsf Z}_{{\mathsf T}_{\tilde I}} = z_2].
\end{align*}
From (\ref{s1}) and the definition of the intervals $I \subset {\tilde I}$, it follows that
\begin{align*}
\sup_{z \in I} \bigl| P^{\mathbb Z}_z [{\mathsf Z}_{{\mathsf T}_{\tilde I}} = z_2] - 1/2 \bigr| \leq cd_N/h_N,
\end{align*}
hence from the previous equality that
\begin{align}
(1- cd_N/h_N) P_{z_1} [{\mathsf H}_{V_I} < {\mathsf T}_{{\tilde B}}] \leq P_{z_1, z_2} [{\mathsf H}_{V_I} < {\mathsf T}_{{\tilde B}}] \leq P_{z_1} [{\mathsf H}_{V_I} < {\mathsf T}_{{\tilde B}}] (1+cd_N/h_N). \label{eq:ch1}
\end{align}
Note that $\{{\mathsf H}_{V_I} < {\mathsf T}_{\tilde B}\} = \{H_{V_I} < T_{\tilde B}\}$, $P_{z_1}$-a.s. Summing over all possible locations and times of the last visit of $X$ to the set $V_I$, one thus finds
\begin{align*}
P_{z_1} [{\mathsf H}_{V_I} < {\mathsf T}_{{\tilde B}}]  = \sum_{x \in V_I} \sum_{n=1}^\infty P_{z_1} \bigl[ \{ X_n=x, n< T_{\tilde B} \} \cap (\theta^X_n)^{-1} \{ {\tilde H}_x > T_{\tilde B} \} \bigr].
\end{align*}
After an application of the simple Markov property to the probability on the right-hand side, this last expression becomes
\begin{align*}
&\sum_{x \in V_I} E_{z_1} \Bigl[ \sum_{n=1}^{T_{\tilde B}} \mathbf{1}_{\{X_n=x\}}  \Bigr] P_x [{\tilde H}_x > T_{\tilde B}] \\
&\qquad = \sum_{x = (y,z) \in V_I} w_x E_{z_1} \Bigl[ \int_0^\infty \mathbf{1}_{\{{\mathsf Y}_t=y\}} \mathbf{1}_{\{{\mathsf Z}_t = z, t< {\mathsf T}_{\tilde I}\}} dt \Bigr] P_x [{\tilde H}_x > T_{\tilde B}],
\end{align*}
because the expected duration of each visit to $x$ by $\mathsf X$ is $1/w_x$.
Exploiting independence of $\mathsf Y$ and $(\mathsf Z, {\mathsf T}_{\tilde I})$ and the fact that ${\mathsf Y}_t$ is distributed according to the uniform distribution on $G$ under $P_{z_1}$, one deduces that
\begin{align}
P_{z_1} [{\mathsf H}_{V_I} < {\mathsf T}_{{\tilde B}}]  &= \sum_{x = (y,z) \in V_I} \frac{w_x}{|G|} E^{\mathbb Z}_{z_1} \Bigl[ \int_0^\infty \mathbf{1}_{\{{\mathsf Z}_t = z, t< {\mathsf T}_{\tilde I}\}} dt \Bigr] P_x [{\tilde H}_x > T_{\tilde B}]. \label{eq:ch2}
\end{align}
Since the expected duration of each visit of $\mathsf Z$ to any point is equal to $1$, we also have
\begin{align}
E^{\mathbb Z}_{z_1} \Bigl[ \int_0^\infty \mathbf{1}_{\{{\mathsf Z}_t = z, t< {\mathsf T}_{\tilde I}\}} dt \Bigr] = E^{\mathbb Z}_{z_1} \Bigl[ \sum_{n=0}^{T_{\tilde I}} \mathbf{1}_{\{Z_n=z\}} \Bigr] = P^{\mathbb Z}_{z_1} [H_z < T_{\tilde I}] / P^{\mathbb Z}_z [{\tilde H}_z > T_{\tilde I}], \label{eq:ch3}
\end{align}
where we have applied the strong Markov property at $H_z$ and computed the expectation of the geometrically distributed random variable with success parameter $P^{\mathbb Z}_z[{\tilde H}_z > T_{\tilde I}]$ in the last step.
Standard arguments on one-dimensional simple random walk (see for example \cite{durrett}, Section 3.1, (1.7), p.~179) show with (\ref{s1}) that the right-hand side of (\ref{eq:ch3}) is bounded from below by $h_N (1-cd_N/h_N)$ and from above by $h_N (1+cd_N/h_N)$. Substituting what we have found into (\ref{eq:ch2}) and remembering (\ref{eq:ch1}), we have proved (\ref{cap1}).
\end{proof}

\begin{proof}[Proof of Lemma~\ref{lem:cap2}.]
In order to prove (\ref{cap2}) it suffices to show that
\begin{align}
\lim_N \max_{m \in {\mathcal J}_I, {\mathtt x} \in {\mathbb V}_m} \bigl|P_{\Phi_m^{-1}({\mathtt x})}[ T_{\tilde B} < {\tilde H}_{V_I}] - {\mathbb P}^m_{\mathtt x}[ {\tilde H}_{{\mathbb      V}_m} = \infty] \bigr| = 0. \label{eq:ch3.1}
\end{align}
Indeed, since the sets $V_m$ are disjoint by (\ref{eq:Bdisj}) and (\ref{eq:inc}), assertion (\ref{eq:ch3.1}) implies that
\begin{align*}
&\max_{I \in {\mathcal I}} \Bigl| \textup{cap}_{\tilde B}(V_I) - \sum_{m \in {\mathcal J}_I} \textup{cap}^m({\mathbb      V}_m) \Bigr| \\
&\qquad = \max_{I \in {\mathcal I}}\Bigl| \sum_{m \in {\mathcal J}_I} \sum_{{\mathtt x} \in {\mathbb V}_m} \bigl( P_{\Phi_m^{-1}({\mathtt x})}[ T_{\tilde B} < {\tilde H}_{V_I}] -  {\mathbb P}^m_{\mathtt x} [ {\tilde H}_{{\mathbb      V}_m} = \infty] \bigr) w_{\mathtt x}  \Bigr| \longrightarrow 0, \textrm{ as } N \to \infty.
\end{align*}
The statement (\ref{eq:ch3.1}) follows from the two claims
\begin{align}
&\lim_N \max_{m \in {\mathcal J}_I, {\mathtt x} \in {\mathbb V}_m} \bigl| P_{\Phi_m^{-1}({\mathtt x})} [ T_{\tilde B} < {\tilde H}_{V_I}] - P_{\Phi_m^{-1}({\mathtt x})} [ T_{B_m} < {\tilde H}_{V_m}] \bigr| = 0 \textrm{ and} \label{eq:ch4}\\
&\lim_N \max_{m \in {\mathcal J}_I, {\mathtt x} \in {\mathbb V}_m} \bigl| {\mathbb P}^m_{\mathtt x}[ {\tilde H}_{{\mathbb      V}_m} = \infty] - P_{\Phi_m^{-1}({\mathtt x})} [ T_{B_m} < {\tilde H}_{V_m} ] \bigr| = 0.  \label{eq:ch4.1}
\end{align}
We first prove (\ref{eq:ch4}). It follows from the inclusions (\ref{eq:inc}) that $P_{\Phi_m^{-1}({\mathtt x})}\textrm{-a.s.}$,
$$T_{\tilde B} = T_{B_m} + T_{C_m \times {\tilde I}} \circ \theta^X_{T_{B_m}} + T_{\tilde B} \circ \theta^X_{T_{C_m \times {\tilde I}}} \circ \theta^X_{T_{B_m}}.$$
Since the sets $B_m$ are disjoint (cf.~(\ref{eq:Bdisj})), the strong Markov property applied at the exit times of $B_m$ and $C_m \times {\tilde I}$ shows that for $x = \Phi_m^{-1}({\mathtt x}) \in V_m$,
\begin{align}
&P_x[T_{\tilde B} < {\tilde H}_{V_I} ] = E_x \Bigl[ T_{B_m} < {\tilde H}_{V_m}, E_{X_{T_{B_m}}} \bigl[ T_{C_m \times {\tilde I}} < H_{V_{I,m}}, P_{X_{T_{C_m \times {\tilde I}}}} [T_{\tilde B} < H_{V_I}] \bigr] \Bigr]  \label{eq:ch5}\\
&\qquad \geq P_x [ T_{B_m} < {\tilde H}_{V_m}] \inf_{x_0 \in \partial B_m} P_{x_0}[T_{C_m \times {\tilde I}} < H_{V_{I,m}}] \inf_{x_0 \in \partial(C_m \times {\tilde I})} P_{x_0} [T_{\tilde B} < H_{V_I}]. \nonumber
\end{align}
We now show that $a_1$ and $a_2$ tend to $1$ as $N$ tends to infinity, where we have set
\begin{align}
a_1 =\inf_{x_0 \in \partial B_m} P_{x_0}[T_{C_m \times {\tilde I}} < H_{V_{I,m}}], \,\, a_2 =\inf_{x_0 \in \partial(C_m \times {\tilde I})} P_{x_0} [T_{\tilde B} < H_{V_I}]. \label{eq:ch5.1}
\end{align}
Concerning $a_1$, note first that
\begin{align}
a_1 \geq 1 - M \max_{m': C_{m'} = C_m} \sup_{x_0 \in \partial B_m} P_{x_0}[H_{V_{m'}} < T_{C_{m'} \times {\tilde I}}]. \label{eq:ch5.2}
\end{align}
With the strong Markov property applied at the entrance time of ${\bar B}_{m'}$, recall that ${\bar B}_m$ is either identical to or disjoint from ${\bar B}_m$ by (\ref{eq:Bdisj}), we can replace $\partial B_m$ by $\partial B_{m'}$ on the right-hand side of (\ref{eq:ch5.2}).
With this remark and the application of the isomorphism $\Psi_{m'}^{z_{m'}}$, one finds with (\ref{eq:isod2}) and ${\hat o}_m = \psi_m(y_m)$ that
\begin{align*}
\sup_{x_0 \in \partial B_m} P_{x_0}[H_{V_{m'}} < T_{C_{m'} \times {\tilde I}}] &\leq \sup_{{\mathtt x}_0 \in \partial B(({\hat o}_{m'},0),{\mathsf r}_N-1)} {\hat {\mathbb P}}^{m'}_{{\mathtt x}_0}[H_{\Psi_{m'}^{z_{m'}}(V_{m'})} < T_{\Psi_{m'}^{z_{m'}}(C_{m'} \times {\tilde I})}] \\
&\leq \sup_{{\mathtt x}_0 \in \partial B(({\hat o}_{m'},0),{\mathsf r}_N-1)} {\hat {\mathbb P}}^{m'}_{{\mathtt x}_0}[H_{\Psi_{m'}^{z_{m'}}(V_{m'})} < \infty].
\end{align*}
From $\Psi^{z_m}_m(V_m) \subset \Psi^{z_m}_m (B(x_m,\kappa)) = B(({\hat o}_m,0),\kappa)$, see~(\ref{eq:inc}), and the left-hand estimate in (\ref{mntrans}), we see that the right-hand side tends to $0$, and hence $a_1$ tends to $1$ as $N$ tends to infinity.
We now show that $a_2$ tends to $1$ as well. The infimum defining $a_2$ can only be attained for points $x_0=(y_0,z_0)$ with $y_0 \in \partial C_m$ (if $z_0 \in \partial {\tilde I}$, the probability is equal to $1$). Hence, we see that
\begin{align}
a_2 \geq 1- |V_I| \max_{m' \in {\mathcal J}_I} \max_{{\mathtt x}' \in {\mathbb V}_{m'}} \sup_{y_0 \in \partial C_m, z_0 \in {\tilde I}} P_{(y_0,z_0)}[ H_{\Phi^{-1}_{m'}({\mathtt x}')} < T_{\tilde B}]. \label{eq:ch7}
\end{align}
By applying the strong Markov property at the entrance time of the set $C_{m'} \times {\tilde I}$ (which is either identical to or disjoint from $C_m \times {\tilde I}$ by (\ref{eq:disj})), it follows that the supremum on the right-hand side of (\ref{eq:ch7}) is bounded from above by
\begin{align*}
\sup_{y_0 \in \partial (C_{m'}^c), z_0 \in {\tilde I}} P_{(y_0,z_0)}[ H_{\Phi^{-1}_{m'}({\mathtt x}')} < T_{\tilde B}],
\end{align*}
which tends to $0$ by the estimate (\ref{k}) of Lemma~\ref{lem:k}. Thus, both $a_1$ and $a_2$ in (\ref{eq:ch5.1}) tend to $1$ as $N$ tends to infinity. With (\ref{eq:ch5}) and the $P_x$-a.s.~inclusion $\{T_{\tilde B} < {\tilde H}_{V_I}\} \subseteq \{T_{B_m} < {\tilde H}_{V_m}\}$, we have shown the announced claim (\ref{eq:ch4}). 

To show (\ref{eq:ch4.1}), we apply the strong Markov property at the exit time of ${\mathbb B}_m$ and obtain for any ${\mathtt x} \in {\mathbb V}_m \subset {\mathbb B}_m$,
\begin{align*}
{\mathbb P}^m_{\mathtt x}[ {\tilde H}_{{\mathbb      V}_m} = \infty] &= {\mathbb E}^m_{\mathtt x} \bigl[ T_{{\mathbb B}_m} < {\tilde H}_{{\mathbb      V}_m}, {\mathbb P}^m_{X_{T_{{\mathbb B}_m}}} [ H_{{\mathbb      V}_m} = \infty] \bigr].
\end{align*}
The right-hand side can be bounded from above by
\begin{align*}
{\mathbb P}^m_{\mathtt x} [ T_{{\mathbb B}_m} < {\tilde H}_{{\mathbb      V}_m} ] = P_{\Phi_m^{-1}({\mathtt x})} [ T_{B_m} < {\tilde H}_{V_m} ], \textrm{ cf. } (\ref{eq:isod1}),
\end{align*}
and using ${\mathbb V}_m \subset B((o_m,0),\kappa)$ (cf.~(\ref{eq:inc})) from below by
\begin{align*}
P_{\Phi_m^{-1}({\mathtt x})} [ T_{B_m} < {\tilde H}_{V_m} ] (1- |V_m| \sup_{{\mathtt x}_0 \in \partial {\mathbb B}_m} \sup_{{\mathtt x}' \in B((o_m,0),\kappa)}   {\mathbb P}^m_{{\mathtt x}_0} [H_{{\mathtt x}'} < \infty]).
\end{align*}
The right-hand estimate in (\ref{mntrans}) shows that this last supremum tends to $0$, hence (\ref{eq:ch4.1}). This completes the proof of Lemma~\ref{lem:cap2}.
\end{proof}

\begin{proof}[Proof of Lemma~\ref{lem:sl}.]
Following the argument of Lemma~4.1 in \cite{S08}, we begin with the proof of (\ref{sl1}). To this end, it suffices to show that for 
\begin{align}
\gamma = {\mathsf t}_N \sigma_N^{3/4}, \textrm{ cf.~(\ref{def:t}), (\ref{def:sigma}),} \label{eq:gamma}
\end{align}
and some constant $c_2>0$,
\begin{align}
&\sup_{z \in {\mathbb Z}} P^{\mathbb Z}_z [{\mathsf D}_{k^*-k_*} > c_2 \gamma] \stackrel{N \to \infty}{\longrightarrow} 0 \textrm{ and } \label{eq:sl0} \\
&\sup_{z \in {\mathbb Z}, x \in G \times {\mathbb Z}} P_z [{\mathsf H}_x \leq c_2 \gamma] \stackrel{N \to \infty}{\longrightarrow} 0. \label{eq:sl1}
\end{align}
Observe first that by the definition of the grid in (\ref{def:grid}), the random variables $T_O$ and $R_1$ are both bounded from above by an exit-time $T_{[z-ch_N,z+ch_N]}$, $P^{\mathbb Z}_z$-a.s. With $E^{\mathbb Z}_z[T_{[z-ch_N,z+ch_N]}] \leq ch_N^2 \leq c {\mathsf t}_N$, it follows from Kha\'sminskii's Lemma (see \cite{S99}, Lemma 1.1, p.~292, and also \cite{K59}) that for some constant $c_3 >0$,
\begin{align}
\sup_{z \in {\mathbb Z}} E^{\mathbb Z}_z \bigl[ \exp \{ c_3 (T_O \vee R_1)/{\mathsf t}_N \} \bigr] \leq 2. \label{eq:sl2}
\end{align}
With the exponential Chebyshev inequality and the strong Markov property applied at the times $R_{k^*-k_*}, D_{k^*-k_*-1}, \ldots, D_1, R_1$, one deduces that
\begin{align*}
\sup_{z \in {\mathbb Z}} P^{\mathbb Z}_z [D_{k^*-k_*} > c \gamma] &\leq \exp \{- c c_3 \sigma_N^{3/4}\} \sup_{z \in {\mathbb Z}} E^{\mathbb Z}_z \bigl[ \exp \{ c_3 D_{k^*-k_*}/{\mathsf t}_N \} \bigr] \\
& \leq \exp \{ -c c_3 \sigma_N^{3/4} \} \Bigl( \sup_{z \in {\mathbb Z}} E^{\mathbb Z}_z \bigl[ \exp \{ c_3 (T_O \vee R_1)/{\mathsf t}_N \} \bigr] \Bigr)^{2 (k^*-k_*)} \\
&\stackrel{(\ref{eq:sl2})}{\leq} \exp\{ -c c_3 \sigma_N^{3/4} + 2 (\log 2) 2 [\sigma_N^{3/4}]\}.
\end{align*}
Hence, the claim (\ref{eq:sl0}) with $\mathsf D$ replaced by $D$ follows for a suitably chosen constant $c$. The claim with $\mathsf D$ for a slightly larger constant $c_2$ is then a simple consequence of Lemma~\ref{lem:DLLN}, applied with $a_N = k^*-k_*$. 

To prove (\ref{eq:sl1}), note that the expected amount of time spent by the random walk $\mathsf X$ at a site $x$ during the time interval $[{\mathsf H}_x, {\mathsf H}_x + 1]$ is bounded from below by $(1 \wedge \sigma^{\mathsf X}_1) \circ \theta_{{\mathsf H}_x}$. Hence, for $z \in {\mathbb Z}$ and $x=(y',z') \in G \times {\mathbb Z}$, the Markov property at time ${\mathsf H}_x$ yields
\begin{align*}
E_z \Bigl[ \int_0^{c_2 \gamma+1} \mathbf{1}_{\{{\mathsf X}_t = x\}} dt \Bigr] \geq  P_z[{\mathsf H}_x \leq c_2 \gamma]  \inf_{x' \in G \times {\mathbb Z}} E_{x'}[1 \wedge \sigma^{\mathsf X}_1] \stackrel{\ref{hyp:w}}{\geq} c P_z[{\mathsf H}_x \leq c_2 \gamma].
\end{align*}
Using the fact that ${\mathsf Y}_t$ is distributed according to the uniform distribution on $G$ under $P_z$, and the bound (\ref{zk}) on the heat kernel of $\mathsf Z$, the left-hand side is bounded by
\begin{align*}
\frac{c}{|G|} \int_0^{c_2 \gamma +1} P^{\mathbb Z}_z [ {\mathsf Z}_t = z' ] dt \leq c \frac{\sqrt{\gamma}}{|G|}.
\end{align*}
We have therefore found that
\begin{align*}
&\sup_{z \in {\mathbb Z}, x \in E} P_z [{\mathsf H}_x \leq c_2 \gamma] \leq c \sqrt{\gamma} |G|^{-1} \stackrel{(\ref{eq:gamma})}{\leq} c \sqrt{{\mathsf t}_N} \sigma^{3/8} |G|^{-1} \stackrel{(\ref{def:sigma}),(\ref{def:t})}{\leq} c(\alpha) (h_N/|G|)^{1/4} 
\end{align*}
and by (\ref{s1}) and \ref{hyp:mix}, we know that $h_N/|G|$ is bounded by $|G|^{-\epsilon/4}$. This completes the proof of (\ref{eq:sl1}) and hence (\ref{sl1}).

\vspace{2mm}

Note that (\ref{sl2}) is a direct consequence of (\ref{sl1}), since the probability in (\ref{sl2}) is smaller than $(\sum_{m} |V_m|) \sup_{z \in {\mathbb Z}, x \in E} P_z[{\mathsf H}_x \leq {\mathsf D}_1]$.

\vspace{2mm}

Finally, the expectation in (\ref{sl3}) is smaller than 
\begin{align*}
P \bigl[ {\mathsf \theta}^{-1}_{{\mathsf D}_{k_*}} \{{\mathsf H}_{\cup_{I} V_I} \leq {\mathsf D}_{k^*-k_*}\} \bigr]  = E \bigl[ P_{{\mathsf Z}_{{\mathsf D}_{k_*}}}  [{\mathsf H}_{\cup_{I} V_I} \leq {\mathsf D}_{k^*-k_*} ] \bigr],
\end{align*}
and hence (\ref{sl3}) follows from (\ref{sl2}).
\end{proof}

\section{Estimates on the jump process} \label{sec:jump}

In this section, we provide estimates on the jump process $\eta^{\mathsf X} = \eta^{\mathsf Y} + \eta^{\mathsf Z}$ of $\mathsf X$ that will be of use in the reduction of Theorem~\ref{thm:d} to the continuous-time result Theorem~\ref{thm:c} in the next section. There, the number $[\alpha |G|^2]$ of steps of $X$ will be replaced by a random number $\eta^{\mathsf X}_{\alpha' |G|^2}$ of jumps and this will make the local time $L^z (\eta^{\mathsf X}_{\alpha|G|^2})$ appear. We hence prove results on the large $N$ behavior of $\eta^X_{\alpha |G|^2}$ (Lemma~\ref{lem:erg}) and $L^z (\eta^{\mathsf X}_{\alpha|G|^2})$ (Lemma~\ref{lem:loct}), for $\alpha>0$. Of course, there is no difficulty in analyzing the Poisson process $\eta^Z$ of constant parameter $1$. The crux of the matter is the $N$-dependent and inhomogeneous component $\eta^{\mathsf Y}$. Let us start by investigating the expectation of $\eta^{\mathsf Y}_t$.

\begin{lemma} \label{lem:exp}
\begin{align}
&\sup_{y \in G} E^G_y [\eta^{\mathsf Y}_t] \leq \max_{y \in G} w_y t, \textrm{ and } \label{incl1}\\
&E^G[ \eta^{\mathsf Y}_t] = tw(G)/|G|, \textrm{ for } t \geq 0 \textrm{ and all } N. \label{exp}
\end{align}
\end{lemma}

\begin{proof}
Under $P^G_y$, $y \in G$, the process
\begin{align}
M_t=\eta_t - \int_0^t w({\mathsf Y}_s) \, ds, \, t \geq 0, \label{eq:exp1}
\end{align}
is a martingale, see Chou and Meyer \cite{CM75}, Proposition 3. A proof of a slightly more general fact is also given by Darling and Norris \cite{DN08}, Theorem~8.4. In order to prove (\ref{incl1}), we take the $E^G_y$-expectation in (\ref{eq:exp1}). If we take the $E^G$-expectation in (\ref{eq:exp1}) and use that $E^G[w({\mathsf Y}_s)] = E^G[w({\mathsf Y}_0)] = w(G)/|G|$ by stationarity, we find (\ref{exp}).
\end{proof}

We next bound the covariance and variance of increments of $\eta^{\mathsf Y}$. Let us denote the compensated increments of $\eta^{\mathsf Y}$ as
\begin{align}
I^{\mathsf Y}_{s,t} = \eta^{\mathsf Y}_t - \eta^{\mathsf Y}_s - (t-s)w(G)/|G|, \textrm{ for } 0 \leq s \leq t. \label{def:Yinc}
\end{align}

\begin{lemma} \label{lem:inc}
Assuming \ref{hyp:w}, one has for $0 \leq s \leq t \leq s' \leq t'$,
\begin{align}
& | \textup{cov}_{P^G} (I^{\mathsf Y}_{s,t}, I^{\mathsf Y}_{s',t'}) | \leq  c_1^2 (t-s) (t'-s') |G| \exp \{-(s'-t) \lambda_N\}, \label{inccov}\\
& \textup{var}_{P^G} (I^{\mathsf Y}_{s,t}) \leq c_1 (t-s) + c_1^2 (t-s)^2 . \label{incvar}
\end{align}
\end{lemma}

\begin{proof}
In Lemma~\ref{lem:exp}, we have proved that $E^G[I_{r,r'}]=0$ for $0 \leq r \leq r'$, so that by the Markov property applied at time $s'$, the left-hand side of (\ref{inccov}) can be expressed as
\begin{align*}
|E^G[I_{s,t} I_{s',t'}]| = |E^G[I_{s,t} (E^G_{{\mathsf Y}_{s'}}[I_{0,t'-s'}] - E^G[I_{0,t'-s'}])]|.
\end{align*}
With an application of the Markov property at time $t$, this last expression becomes
\begin{align*}
&\Bigl| \sum_{y \in G} E^G \bigl[ I_{s,t} (q^G_{s'-t}({\mathsf Y}_t, y) - |G|^{-1}) \bigr] E^G_y [I_{0,t'-s'}] \Bigr| \\
& \leq  \sum_{y \in G} E^G \bigl[ |I_{s,t}| |q^G_{s'-t}({\mathsf Y}_t, y) - |G|^{-1}| \bigr] |E^G_y [I_{0,t'-s'}]|. 
\end{align*}
The claim (\ref{inccov}) thus follows by applying the estimate (\ref{st}) inside the expectation, then (\ref{incl1}) and $w(G)/|G| \leq c_1$ in order to bound the remaining terms.

To show (\ref{incvar}), we apply the Markov property at time $s$ and domination of $\eta^{\mathsf Y}_{t-s}$ by a Poisson random variable of parameter $c_1 (t-s)$ (cf.~(\ref{sd2})):
\begin{flushright}
$\textup{var}_{P^G} (I^{\mathsf Y}_{s,t}) \leq E^G [(\eta^{\mathsf Y}_t - \eta^{\mathsf Y}_s)^2] = E^G [(\eta^{\mathsf Y}_{t-s})^2] \leq c_1 (t-s) + c_1^2 (t-s)^2. \qquad \qquad  \qedhere$
\end{flushright}
\end{proof}

In the next Lemma, we transfer some of the previous estimates to the process $\eta^{\mathsf Y}_{\sigma^{\mathsf Z}_.}$.

\begin{lemma} \label{cor:inc}
Assuming \ref{hyp:w},
\begin{align}
&E[ \eta^{\mathsf Y}_{\sigma^{\mathsf Z}_1}] = w(G)/|G|. \label{exp'}\\
&\sup_{x \in G \times {\mathbb Z}} E_x [ \eta^{\mathsf Y}_{\sigma^{\mathsf Z}_1} ] \leq c_1. \label{incl1'} \\
%&\sup_{x \in E} \textup{var}_{P_x} (\eta^{\mathsf Y}_{\sigma^{\mathsf Z}_1}) \leq c_1. \label{incvar'}
&\sup_{x \in G \times {\mathbb Z}} E_x [(\eta^{\mathsf Y}_{\sigma^{\mathsf Z}_1})^2] \leq  c_1+ 2c_1^{2}. \label{incvar'}
\end{align}
\end{lemma}

\begin{proof}
All three claims are shown by using independence of $\eta^{\mathsf Y}$ and $\sigma^{\mathsf Z}$ and applying Fubini's theorem. To show (\ref{exp'}), note that
\begin{align*}
E[ \eta^{\mathsf Y}_{\sigma^{\mathsf Z}_1}] = E \bigl[ E^G[ \eta^{\mathsf Y}_t ] \bigl|_{t=\sigma^{\mathsf Z}_1} \bigr] \stackrel{(\ref{exp})}{=} E [ \sigma^{\mathsf Z}_1 ] w(G) /|G| = w(G)/|G|.
\end{align*}
The statements (\ref{incl1'}) and (\ref{incvar'}) are shown similarly, using additionally stochastic domination of $\eta^{\mathsf Y}_t$ by a Poisson random variable of parameter $c_1t$ (cf.~(\ref{sd2})).
%For (\ref{incl1'}) and (\ref{incvar'}), we consider any $x = (y,z) \in G \times {\mathbb Z}$ and proceed similarly, using in addition stochastic %domination of $\eta^{\mathsf Y}_t$ by a Poisson random variable $\eta^{c_1 t}$ (cf.~(\ref{sd2})):
%\begin{align*}
%&E_x[ \eta^{\mathsf Y}_{\sigma^{\mathsf Z}_1}] = E^{\mathbb Z}_z \bigl[ E_y^G \bigl[\eta^{\mathsf Y}_t] \bigl|_{t=\sigma^{\mathsf Z}_1} \bigr] \leq %E^{\mathbb Z}_z [ c_1 \sigma^{\mathsf Z}_1 ] = c_1,\\
%& E_x [(\eta^{\mathsf Y}_{\sigma^{\mathsf Z}_1})^2] = E^{\mathbb Z}_z \bigl[ E_y^G \bigl[(\eta^{\mathsf Y}_t)^2] \bigl|_{t=\sigma^{\mathsf Z}_1} \bigr]  %\leq E^{\mathbb Z}_z \bigl[ c_1 \sigma^{\mathsf Z}_1 + (c_1 \sigma^{\mathsf Z}_1)^2 \bigr]  = c_1 + 2c_1^2. 
%\end{align*} 
\end{proof}

We now come to the two main results of this section. As announced, we now analyze the asymptotic behavior of $\eta^{\mathsf X}_{\alpha |G|^2}$, where the whole difficulty comes from the component $\eta^{\mathsf Y}_{\alpha |G|^2}$. The method we use is to split the time interval $[0, \alpha |G|^2]$ into $[|G|^{\epsilon/2}]$ increments of length longer than $\lambda_N^{-1}$. This is possible by \ref{hyp:mix} and ensures that the bound from (\ref{inccov}) on the covariance between different increments of $\eta^{\mathsf Y}$ becomes useful for non-adjacent increments. The following lemma follows from the second moment Chebyshev inequality and the covariance bound applied to pairs of non-adjacent increments.

\begin{lemma}\label{lem:erg}
Assuming \ref{hyp:w} and (\ref{hyp:d}),
\begin{align}
\lim_N E \bigl[ |\eta^{\mathsf X}_{\alpha |G|^2}/(\alpha |G|^2) - (1+\beta)| \wedge 1 \bigr] = 0, \textrm{ for } \alpha>0. \label{erg}
\end{align}
\end{lemma}

\begin{proof}
The law of large numbers implies that $\eta^{\mathsf Z}_{\alpha |G|^2}/(\alpha |G|^2)$ converges to $1$, $P^{\mathbb Z}_0$-a.s. (see, for example \cite{durrett}, Chapter 1, Theorem 7.3). Moreover, $\lim_N w(G)/|G| = \beta$ by (\ref{hyp:d}). Since $\eta^{\mathsf X} = \eta^{\mathsf Y} + \eta^{\mathsf Z}$, it hence suffices to show that
\begin{align}
\lim_N E^G \bigl[ (|\eta^{\mathsf Y}_{\alpha |G|^2}/(\alpha |G|^2) - w(G)/|G||) \wedge 1 \bigr] = 0. \label{eq:erg1}
\end{align}
To this end, put $a = [|G|^{\epsilon/2}]$, $\tau= \alpha |G|^2 / a$, and write 
\begin{align}
&\eta^{\mathsf Y}_{\alpha |G|^2} -  \alpha |G|^2 (w(G)/|G|)  = \sum_{\substack{1 \leq n \leq a, \\ n \textrm{ even}}} I^{\mathsf Y}_{(n-1)\tau,n\tau} + \sum_{\substack{1 \leq n \leq a, \\ n \textrm{ odd}}} I^{\mathsf Y}_{(n-1)\tau,n\tau} \stackrel{\textrm{(def.)}}{=} \Sigma_1 + \Sigma_2, \label{eq:erg2} 
\end{align}
for $I^{\mathsf Y}$ as in (\ref{def:Yinc}). Fix any $\delta > 0$ and $\Sigma \in \{\Sigma_1, \Sigma_2\}$. By Chebyshev's inequality,
\begin{align}
P^G [ |\Sigma| \geq \delta \alpha |G|^2 ] &\leq \frac{1}{\delta^2 \alpha^2 |G|^4} E^G [ \Sigma^2] \label{eq:erg3}\\
&= \frac{1}{\delta^2 \alpha^2 |G|^4} \Bigl( \sum_i E^G[(I^{\mathsf Y}_{(i-1)\tau,i\tau})^2] + \sum_{i \neq j} E^G [I^{\mathsf Y}_{(i-1)\tau,i\tau} I^{\mathsf Y}_{(j-1)\tau,j\tau}] \Bigr), \nonumber
\end{align}
where the two sums are over unordered indices $i$ and $j$ in $\{1, \ldots, a\}$ that are either all even or all odd, depending on whether $\Sigma$ is equal to $\Sigma_1$ or to $\Sigma_2$. The right-hand side of (\ref{eq:erg3}) can now be bounded with the help of the estimates on the increments of $\eta^{\mathsf Y}$ in Lemma~\ref{lem:inc}. Indeed, with (\ref{incvar}), the first sum is bounded by $ca \tau^2 \leq c(\alpha) |G|^{4-\epsilon/2}$. For the second sum, we observe that $|i-j| \geq 2$ for all indices $i$ and $j$, apply (\ref{inccov}) and \ref{hyp:mix} and bound the sum with $(|G| \tau)^c \exp \{-c(\alpha) \tau \lambda_N \} \leq |G|^c \exp \{-c(\alpha) |G|^{\epsilon/2} \}$. Hence we find that
\begin{align*}
P^G [ |\Sigma| \geq \delta \alpha |G|^2 ] \leq c(\alpha,\delta) (|G|^{-\epsilon/2}+|G|^c \exp \{-c(\alpha) |G|^{\epsilon/2} \}) \to 0, \textrm{ as } N \to \infty,
\end{align*}
from which we deduce with (\ref{eq:erg2}) that for our arbitrarily chosen $\delta > 0$,
\begin{align*}
P^G \bigl[ |\eta^{\mathsf Y}_{\alpha |G|^2}/(\alpha |G|^2) - w/|G| | \geq 2 \delta \bigr] \leq P^G [ |\Sigma_1| \geq \delta \alpha |G|^2 ]  +  P^G [ |\Sigma_2| \geq \delta \alpha |G|^2 ] \to 0,
\end{align*}
as $N$ tends to infinity, showing (\ref{eq:erg1}). This completes the proof of Lemma~\ref{lem:erg}.
\end{proof}

In the final lemma of this section, we apply a similar analysis to the local time of the process $\pi_{\mathbb Z}(X)$ evaluated at time $\eta^{\mathsf X}_{\alpha |G|^2}$. The proof is similar to the preceding argument, although the appearance of $\eta^{\mathsf Y}$ evaluated at the random times $\sigma^{\mathsf Z}_n$ complicates matters. We recall the notation $L$ and $\hat L$ for the local times of $\pi_{\mathbb Z}(X)$ and $Z$ from (\ref{def:locpr}) and (\ref{def:loct}).

\begin{lemma} \label{lem:loct}
Assuming \ref{hyp:w}, \ref{hyp:mix} and (\ref{hyp:d}),
\begin{align}
&\lim_N \sup_{z \in {\mathbb Z}} E \bigl[ (|L^z_{\eta^{\mathsf X}_{\alpha |G|^2}} - (1+\beta) {\hat L}^z_{\eta^{\mathsf Z}_{\alpha |G|^2}}|/|G|) \wedge 1 \bigr] = 0, \textrm{ for } \alpha >0. \label{loct} 
\end{align}
\end{lemma}

\begin{proof}
Set $T = \alpha |G|^2$. By independence of $\eta^{\mathsf Z}$ and $Z$, we have
\begin{align*}
E[{\hat L}_{\eta^{\mathsf Z}_T}] = E \bigl[ \sum_{n \geq 0} \mathbf{1}_{\{ n<\eta^{\mathsf Z}_T \}} P^{\mathbb Z}_0 [Z_n=z] \bigr] \stackrel{(\ref{eq:zk2})}{\leq} c E \Bigl[ \sqrt{\eta^{\mathsf Z}_T} \Bigr] \stackrel{(\textrm{Jensen})}{\leq} c(\alpha) |G|.
\end{align*}
From this estimate and the assumption $w(G)/|G| \to \beta$ made in (\ref{hyp:d}), it follows that it suffices to prove (\ref{loct}) with $w(G)/|G|$ in place of $\beta$.
It follows from the definition of $L^z$ in (\ref{def:locpr}) that
\begin{align}
&\sum_{n=0}^{\eta^{\mathsf Z}_T -1} \mathbf{1}_{\{Z_n = z\}} (1+ \eta^{\mathsf Y}_{\sigma^{\mathsf Z}_{n+1}} - \eta^{\mathsf Y}_{\sigma^{\mathsf Z}_n}) \leq L^z_{\eta^{\mathsf X}_T} \leq  \sum_{n=0}^{\eta^{\mathsf Z}_T} \mathbf{1}_{\{Z_n = z\}} (1+ \eta^{\mathsf Y}_{\sigma^{\mathsf Z}_{n+1}} - \eta^{\mathsf Y}_{\sigma^{\mathsf Z}_n}), \textrm{ hence} \nonumber\\
&\sup_{z \in {\mathbb Z}} E \Bigl[ \Bigl| L^z_{\eta^{\mathsf X}_T} - \sum_{n=0}^{\eta^{\mathsf Z}_T -1 } \mathbf{1}_{\{Z_n = z\}} (1+ \eta^{\mathsf Y}_{\sigma^{\mathsf Z}_{n+1}} - \eta^{\mathsf Y}_{\sigma^{\mathsf Z}_n}) \Bigr| \Bigr] \leq 1+ E[ \eta^{\mathsf Y}_{\sigma^{\mathsf Z}_{\eta^{\mathsf Z}_T+1}} - \eta^{\mathsf Y}_{\sigma^{\mathsf Z}_{\eta^{\mathsf Z}_T}}]. \label{eq:loct1}
\end{align}
By independence of $\eta^{\mathsf Y}$ and $(\sigma^{\mathsf Z},\eta^{\mathsf Z})$ and the simple Markov property (under $P^G$) applied at time $\sigma^{\mathsf Z}_{\eta^{\mathsf Z}_T}$, the expectation on the right-hand side is with (\ref{incl1}) bounded by $cE[\sigma^{\mathsf Z}_{\eta^{\mathsf Z}_T+1} - \sigma^{\mathsf Z}_{\eta^{\mathsf Z}_T}]$. This last expectation is equal to the sum of two independent $\exp(1)$-distributed random variables, so it follows that the right-hand side of (\ref{eq:loct1}) is bounded by a constant. By these observations, the proof will be complete once we show that
\begin{align}
&\lim_N \sup_{z \in {\mathbb Z}} E \Bigl[ \Bigl( \Bigl| \sum_{n=0}^{\eta^{\mathsf Z}_T -1} \mathbf{1}_{\{Z_n=z\}} S_n \Bigr| / |G| \Bigr) \wedge 1 \Bigr] =0, \textrm{ where} \label{eq:loct2}\\
&S_n = \eta^{\mathsf Y}_{\sigma^{\mathsf Z}_{n+1}} - \eta^{\mathsf Y}_{\sigma^{\mathsf Z}_n} - w(G)/|G|, \textrm{ for } n \geq 0. \nonumber
\end{align}
To this end, we will prove that
\begin{align}
 &\lim_N \sup_{z \in {\mathbb Z}} E \Bigl[ \Bigl( \Bigl| \sum_{n=0}^{\eta^{\mathsf Z}_T -1} \mathbf{1}_{\{Z_n=z\}} S_n   - \sum_{n=0}^{[T]} \mathbf{1}_{\{Z_n=z\}} S_n\Bigr| / |G| \Bigr) \wedge 1 \Bigr] =0, \textrm{ and } \label{eq:loct2.1}\\
 & \lim_N \sup_{z \in {\mathbb Z}} E \Bigl[ \Bigl| \sum_{n=0}^{[T]} \mathbf{1}_{\{Z_n=z\}} S_n \Bigr| / |G| \Bigr] =0. \label{eq:loct2.2}
\end{align}
In order to show (\ref{eq:loct2.1}), we note that by the Chebyshev inequality,
\begin{align}
 P[|\eta^{\mathsf Z}_T - T| \geq T^{3/4} ] \leq cT^{-3/2} E[(\eta^{\mathsf Z}_T-T)^2] = T^{-1/2}. \label{eq:loct2.3}
\end{align}
The expectation in (\ref{eq:loct2.1}), taken on the complement of the event $\{ |\eta^{\mathsf Z}_T - T| \geq T^{3/4} \}$, is bounded by
\begin{align}
\frac{1}{|G|} \sum_{T-cT^{3/4} \leq n \leq T+ cT^{3/4}} E[ \mathbf{1}_{\{Z_n=z\}} |S_n|]. \label{eq:loct2.4}
\end{align}
Using independence of $Z$ and $\eta^{\mathsf Y}_{\sigma^{\mathsf Z}_.}$ and the heat-kernel bound (\ref{eq:zk2}), we find that the last expectation is bounded by $cE[|S_n|]/\sqrt{n}$, which by the strong Markov property applied at time $\sigma^{\mathsf Z}_n$, (\ref{exp'}) and \ref{hyp:w} is bounded by $c/\sqrt{n}$. The expression in (\ref{eq:loct2.4}) is thus bounded by $c T^{3/8}/|G| = c \alpha |G|^{-1/4}$ and with (\ref{eq:loct2.3}), we have proved (\ref{eq:loct2.1}).

We now come to (\ref{eq:loct2.2}). By the Cauchy-Schwarz inequality, we have for all $z \in {\mathbb Z}$,
\begin{align}
E \Bigl[ \Bigl| \sum_{n=0}^{[T]} \mathbf{1}_{\{Z_n=z\}} S_n \Bigr| / |G| \Bigr]^2 \leq \frac{1}{|G|^2} E \Bigl[ \Bigl| \sum_{n=0}^{[T]} \mathbf{1}_{\{Z_n=z\}} S_n \Bigr|^2 \Bigr]. \label{eq:loct2.5}
\end{align}
We will now expand the square and respectively sum over identical indices, indices of distance at most $[|G|^{2-\epsilon/2}]$, indices of distance greater than $[|G|^{2-\epsilon/2}]$. Proceeding in this fashion, the right-hand side of (\ref{eq:loct2.5}) equals
\begin{align}
& \frac{1}{|G|^2} \Biggl( \sum_{0 \leq n \leq T} E \bigl[Z_n=z, S_n^2 \bigr] + 2  \sum_{0 \leq n < n' \leq (n+b) \wedge [T]} E \bigl[ Z_n=Z_{n'}=z, S_n S_{n'} \bigr] \label{eq:loct3}\\
&\qquad  + 2 \sum_{0 \leq n, \, n+b < n' \leq [T]}  E \bigl[ Z_n=Z_{n'}=z, S_n S_{n'}  \bigr] \Biggr) , \textrm{ where } b = [|G|^{2-\epsilon/2}]. \nonumber
\end{align}
We now treat each of these three sums separately, starting with the first one. By the strong Markov property, (\ref{incvar'}) and \ref{hyp:w},
\begin{align}
\sum_{0 \leq n \leq [T]} E \bigl[Z_n=z, S_n^2 \bigr] &= \sum_{0 \leq n \leq [T]} E \bigl[Z_n=z, E_{{\mathsf X}_{\sigma^{\mathsf Z}_n}} [S_0^2] \bigr]  \leq c\sum_{0 \leq n \leq [T]} P[Z_n=z]. \label{eq:loct3.1}
\end{align}
By the heat-kernel bound (\ref{eq:zk2}), this last sum is bounded by $\sum_n c/{\sqrt{n}} \leq c \sqrt{T}$. We have thus found that
\begin{align}
\sum_{0 \leq n \leq [T]} E \bigl[Z_n=z, S_n^2 \bigr] \leq c(\alpha) |G|. \label{eq:loct4}
\end{align}
For the second sum in (\ref{eq:loct3}), we proceed in a similar fashion. The strong Markov property applied at time $\sigma^{\mathsf Z}_{n'} \geq \sigma^{\mathsf Z}_{n+1}$ and the estimate (\ref{incl1'}) together yield
\begin{align*}
&\sum_{0 \leq n < n' \leq (n+b) \wedge [T]} E \bigl[ Z_n=Z_{n'}=z, S_n S_{n'} \bigr]  = \sum_{n,n'} E \bigl[ Z_n=Z_{n'}=z, S_n E_{{\mathsf X}_{\sigma^{\mathsf Z}_{n'}}}[S_0] \bigr] \Bigr| \\
&\qquad \leq c \sum_{0 \leq n \leq [T]} E \Bigl[Z_n=z, |S_n| \sum_{n'=n+1}^{n+b} \mathbf{1}_{\{Z_{n'}=z\}} \Bigr].
\end{align*}
Applying the strong Markov property at time $\sigma^{\mathsf Z}_{n+1}$, we bound the right-hand side by
\begin{align*}
&c \sum_{0 \leq n \leq [T]} \Bigl( E [Z_n=z, |S_n|] \sum_{n'=0}^{b-1} \sup_{z' \in {\mathbb Z}} P^{\mathbb Z}_{z'} [Z_{n'}=z] \Bigr) \stackrel{(\ref{eq:zk2})}{\leq} c \sqrt{b} \sum_{0 \leq n \leq [T]}  E [Z_n=z, |S_n|].
\end{align*}
The sum on the right-hand side can be bounded by $c(\alpha) |G|$ with the same arguments as in (\ref{eq:loct3.1})-(\ref{eq:loct4}), the only difference being the use of the estimate (\ref{incl1'}) rather than (\ref{incvar'}). Inserting the definition of $b$ from (\ref{eq:loct3}), we then obtain
\begin{align}
\sum_{0 \leq n < n' \leq n+b} E \bigl[ Z_n=Z_{n'}=z, S_n S_{n'} \bigr]  \leq c(\alpha) |G|^{2-\epsilon/4}. \label{eq:loct5}
\end{align}
For the expectation in the third sum in (\ref{eq:loct3}),we first use independence of $Z$ and $S_.$, then \eqref{exp'} and the fact that the process $\sigma^{\mathsf Z}$ has iid $\exp(1)$-distributed increments for the second line and thus obtain
\begin{align*}
&\bigl|E \bigl[ Z_n=Z_{n'}=z, S_n S_{n'}  \bigr] \bigr| = P[Z_n=Z_{n'}=z] \bigl|E \bigl[ S_n S_{n'}  \bigr] \bigr| \leq   \bigl|E \bigl[ S_n S_{n'}  \bigr] \bigr|   \\
&\qquad = \Bigl| E \bigl[ (\eta^{\mathsf Y}_{\sigma^{\mathsf Z}_{n+1}} - \eta^{\mathsf Y}_{\sigma^{\mathsf Z}_n}) (\eta^{\mathsf Y}_{\sigma^{\mathsf Z}_{n'+1}} - \eta^{\mathsf Y}_{\sigma^{\mathsf Z}_{n'}})] - \frac{w(G)^2}{|G|^2} E[(\sigma^{\mathsf Z}_{n+1} - \sigma^{\mathsf Z}_n) (\sigma^{\mathsf Z}_{n'+1} - \sigma^{\mathsf Z}_{n'}) \bigr] \Bigr|.
\end{align*}
Independence of $\eta^{\mathsf Y}$ and $\sigma^{\mathsf Z}$ and an application of Fubini's theorem then allows to bound the the third sum in (\ref{eq:loct3}) by 
\begin{align*}
\sum_{0 \leq n, \, n+b < n' \leq [T]}   \bigl| E^{\mathbb Z}_0 [ h(\sigma^{\mathsf Z}_n,\sigma^{\mathsf Z}_{n+1},\sigma^{\mathsf Z}_{n'},\sigma^{\mathsf Z}_{n'+1})  ] \bigr|, \textrm{ where } h(s,t,s',t') = \textup{cov}_{P^G}(\eta^{\mathsf Y}_t - \eta^{\mathsf Y}_s, \eta^{\mathsf Y}_{t'}- \eta^{\mathsf Y}_{s'}).
\end{align*}
Via the estimate (\ref{inccov}) on the covariance, this expression is bounded by
\begin{align*}
c|G| \sum_{0 \leq n, \, n+b < n' \leq [T]}  E^{\mathbb Z}_0 \bigl[ (\sigma^{\mathsf Z}_{n+1} - \sigma^{\mathsf Z}_n) (\sigma^{\mathsf Z}_{n'+1}-\sigma^{\mathsf Z}_{n'}) \exp\{ - (\sigma^{\mathsf Z}_{n'} - \sigma^{\mathsf Z}_{n+1}) \lambda_N\} \bigr].
\end{align*}
Since the process $\sigma^{\mathsf Z}$ has iid $\exp(1)$-distributed increments, this sum can be simplified to
\begin{align*}
&\sum_{0 \leq n, \, n+b < n' \leq [T]} E \bigl[ \exp\{ - \sigma^{\mathsf Z}_1 \lambda_N\} \bigr]^{n'-n-1} \leq \sum_{ 0 \leq n \leq [T]} \sum_{ n' > n+b}   \Bigl( \frac{1}{1+\lambda_N} \Bigr)^{n'-n-1} \\
& \quad = [T] \frac{1+\lambda_N}{\lambda_N} \Bigl(\frac{1}{1+\lambda_N} \Bigr)^b \leq c(\alpha) |G|^c e^{-cb \lambda_N} \leq c(\alpha) |G|^c \exp\{- c|G|^{\epsilon/2}\}, \textrm{ by } \ref{hyp:mix}.
\end{align*}
Combining this bound on the third sum in (\ref{eq:loct3}) with the bounds (\ref{eq:loct4}) and (\ref{eq:loct5}) on the first and second sums, we have shown (\ref{eq:loct2.2}), hence (\ref{eq:loct2}). This completes the proof of Lemma~\ref{lem:loct}.
\end{proof}

\section{Proof of the result in discrete time} \label{sec:proof}

In this section, we prove Theorem~\ref{thm:d}. We assume that \ref{hyp:w}-\ref{hyp:k} and (\ref{hyp:d}) hold. The proof uses the estimates of the previous section to deduce Theorem~\ref{thm:d} from the continuous-time version stated in Theorem~\ref{thm:c}.

\begin{proof}[Proof of Theorem~\ref{thm:d}.]
The transience of the graphs ${\mathbb G}_m \times {\mathbb Z}$ follows from Theorem~\ref{thm:c}. Consider again finite subsets ${\mathbb V}_m$ of ${\mathbb G}_m \times {\mathbb Z}$, $1 \leq m \leq M$ and set $V_m = \Phi^{-1}_m ({\mathbb V}_m)$.
We show that for $\theta_m \in {\mathbb R}_+$, $\alpha > 0$,
\begin{align}
&\lim_N E \Bigl[ \prod_{1 \leq m \leq M} \mathbf{1}_{\{H_{V_m} > T\}} \exp \Bigl\{-  \frac{\theta_m }{|G|}  L^{z_m}_{T} \Bigr\} \Bigr] = B(\alpha), \textrm{ where } T = \alpha |G|^2 \textrm{ and} \label{eq:d1}\\
&B(\alpha)= E^W \Bigl[ \exp \Bigl\{ - \sum_{1 \leq m \leq M} L(v_m, \alpha/(1+\beta)) (\textup{cap}^m({\mathbb      V}_m) + (1+\beta) \theta_m) \Bigr\} \Bigr]. \nonumber
\end{align}
This implies Theorem~\ref{thm:d}, by the standard arguments described below (\ref{eq:A}). Recall that two sequences are said to be limit equivalent if their difference tends to $0$ as $N$ tends to infinity. 
If we apply Theorem~\ref{thm:c} with $\alpha/(1+\beta)$ in place of $\alpha$, we obtain
\begin{align*}
\lim_N E \Bigl[ \prod_{1 \leq m \leq M} \mathbf{1}_{\{H_{V_m} > \eta^{\mathsf X}_{T/(1+\beta)}\}} \exp \Bigl\{-  \frac{\theta_m (1+\beta)}{|G|}  {\mathsf L}^{z_m}_{T/(1+\beta)} \Bigr\} \Bigr] = B(\alpha).
\end{align*}
By (\ref{loc1}), the expression on the left-hand side is limit equivalent to the same expression with $\mathsf L$ replaced by ${\hat L}$. Hence, we have
\begin{align*}
\lim_N E \Bigl[ \prod_{1 \leq m \leq M} \mathbf{1}_{\{H_{V_m} > \eta^{\mathsf X}_{T/(1+\beta)}\}} \exp \Bigl\{-  \frac{\theta_m (1+\beta)}{|G|}  {\hat L}^{z_m}_{T/(1+\beta)} \Bigr\} \Bigr] = B(\alpha).
\end{align*}
By the law of large numbers,
$\lim_N \eta^{\mathsf Z}_{T/(1+\beta)}(T/(1+\beta))^{-1} = 1, \textrm{ } P$-a.s. 
Making use of the monotonicity of the left-hand side in the local time and continuity of $B(.)$, we deduce that
\begin{align*}
\lim_N E \Bigl[ \prod_{1 \leq m \leq M} \mathbf{1}_{\{H_{V_m} > \eta^{\mathsf X}_{T/(1+\beta)}\}} \exp \Bigl\{-  \frac{\theta_m (1+\beta)}{|G|}  {\hat L}^{z_m}_{\eta^{\mathsf Z}_{T/(1+\beta)}} \Bigr\} \Bigr] = B(\alpha).
\end{align*}
The estimate (\ref{loct}) then shows that the expression on the left-hand side is limit equivalent to the same expression with  $(1+\beta){\hat L}^{z_m}_{\eta^{\mathsf Z}_{T/(1+\beta)}}$ replaced by $L^{z_m}_{\eta^{\mathsf X}_{T/(1+\beta)}}$, i.e.
\begin{align*}
\lim_N E \Bigl[ \prod_{1 \leq m \leq M} \mathbf{1}_{\{H_{V_m} > \eta^{\mathsf X}_{T/(1+\beta)}\}} \exp \Bigl\{-  \frac{\theta_m }{|G|}  L^{z_m}_{\eta^{\mathsf X}_{T/(1+\beta)}} \Bigr\} \Bigr] = B(\alpha).
\end{align*}
Applying the estimate (\ref{erg}), with the same monotonicity and continuity arguments as in the beginning of the proof, we can replace $\eta^{\mathsf X}_{T/(1+\beta)}$ by $T$, hence infer that (\ref{eq:d1}) holds.
\end{proof}

\section{Examples} \label{sec:ex}

In this section, we apply Theorem~\ref{thm:d} to three examples of graphs $G$: The $d$-dimensional box of side-length $N$, the Sierpinski graph of depth $N$, and the $d$-ary tree of depth $N$ ($d \geq 2$). In each case, we check assumptions \ref{hyp:w}-\ref{hyp:k}, stated after (\ref{def:spec}). In all examples it is implicitly understood that all edges of the graphs have weight $1/2$. 
We begin with a lemma from \cite{SC} asserting that the continuous-time spectral gap has the same order of magnitude as its discrete-time analog $\lambda^d_N$. This result will be useful for checking \ref{hyp:mix}.

\begin{lemma} \label{lem:eval}
Assume \ref{hyp:w} and let $\lambda^d_N$ bet the smallest non-zero eigenvalue of the matrix $I-P(G)$, where $P(G) = (p^G(y,y'))$ is the transition matrix of $Y$ under $P^G$. Then there are constants $c(c_0,c_1)$, $c'(c_0,c_1) >0$ (cf.~\ref{hyp:w}), such that for all $N$, 
\begin{align}
c(c_0,c_1) \lambda^d_N \leq \lambda_N \leq c'(c_0,c_1) \lambda^d_N. \label{eval}
\end{align}
\end{lemma}

\begin{proof}
We follow arguments contained in \cite{SC}. With the Dirichlet form ${\mathcal D}_\pi(.,.)$ defined as
${\mathcal D}_\pi(f,f) = {\mathcal D}_N(f,f) \frac{|G|}{w(G)},$ for $f: G \to {\mathbb R}$ (cf.~\eqref{def:dir}), one has (cf.~\cite{SC}, Definition 2.1.3, p.~327)
\begin{align}
\lambda^d_N = \min \biggl\{ \frac{{\mathcal D}_\pi (f,f)}{\textup{var}_\pi (f)}: f \textrm{ is not constant} \biggr\}. \label{eq:eval0}
\end{align}
From \ref{hyp:w}, it follows that
\begin{align}
\label{eq:eval1}
\begin{array}{cccccl}
 c_1^{-1} {\mathcal D}_N(f,f) &\leq& {\mathcal D}_\pi(f,f) &\leq& c_0^{-1} {\mathcal D}_N(f,f),& \textrm{for any } f: G \to {\mathbb R}, \textrm{ and }\\
 c_0 c_1^{-1} \mu(y) &\leq& \pi(y) &\leq& c_1 c_0^{-1} \mu(y),& \textrm{for any } y \in G.
\end{array}
\end{align}
Using $\textup{var}_\pi (f) = \inf_{\theta \in {\mathbb R}} \sum_{y \in G} (f(y) - \theta)^2 \pi(y)$ and the analogous statement for $\textup{var}_\mu$, the estimate in the second line implies that
\begin{align}
 \label{eq:eval2} c_0c_1^{-1} \textup{var}_\mu (f) \leq \textup{var}_\pi (f) \leq c_1 c_0^{-1} \textup{var}_\mu (f), \textrm{ for any } f: G \to {\mathbb R}.
\end{align}
Lemma~\ref{lem:eval} then follows by using \eqref{eq:eval1} and \eqref{eq:eval2} to compare the definition \eqref{def:spec} of $\lambda_N$ with the characterization \eqref{eq:eval0} of $\lambda_N^d$.
\end{proof}

The following lemma provides a sufficient criterion for assumption \ref{hyp:k}.
\begin{lemma} \label{lem:kch}
Assuming \ref{hyp:w}-\ref{hyp:trans} and that
\begin{align}
\lim_N \sum_{n=1}^{[\lambda_N^{-1} |G|^\epsilon]} \sup_{\substack{y_0 \in \partial (C_m^c) \\ y \in B(y_m,\rho_0)}} p^G_n(y_0,y) \frac{1}{\sqrt{n}} = 0, \textrm{ for any } \rho_0 > 0, \label{kch}
\end{align}
\ref{hyp:k} holds as well.
\end{lemma}

\begin{proof}
For ${\mathtt x}=({\mathtt y},z)$, the probability in \ref{hyp:k} is bounded from above by
\begin{align}
\sum_{n=1}^{[\lambda_N^{-1} |G|^\epsilon]} P_{(y_0,z_0)} [ Y_n = \phi_m^{-1}({\mathtt y}) , z_m +z \in {\mathsf Z}_{[\sigma^{\mathsf Y}_n, \sigma^{\mathsf Y}_{n+1}]}], \label{eq:k1}
\end{align}
using that $y_0 \neq \phi_m^{-1}({\mathtt y})$ for large $N$ (cf.~\ref{hyp:inc}) in order to drop the term $n=0$.
With the same estimates as in the proof of Lemma~\ref{lem:mntrans}, see (\ref{eq:mntrans2})-(\ref{eq:mntrans3}), the expression in (\ref{eq:k1}) can be bounded by a constant times the sum on the left-hand side of (\ref{kch}). 
\end{proof}

%\begin{theorem}
%Consider $M \geq 1$ and for each $N \geq 1$, $x_m = (y_m,z_m)$, $1 \leq m \leq M$, points in $({\mathbb Z}/N{\mathbb Z})^d \times {\mathbb Z} = {\mathbb %T} \times {\mathbb Z}$ such that
%\begin{align}
%&\lim_N \min_{1 \leq m < m' \leq M} |x_m-x_{m'}|_\infty = \infty, \label{torus1} \\
%&\lim_N \frac{z_m}{N^{2d}} = v_m \in {\mathbb R}, \textrm{ for } 1 \leq m \leq M.
%\end{align}
%Set $r_0 = \frac{1}{2} \min_{1 \leq m < m' \leq M} |x_m-x_{m'}|_\infty$ and define the isomorphisms $\phi_m$ as the bijections from $B(y_m,r_0)$ to %$B(0,r_0) \subset {\mathbb Z}^d$ such that $\pi_{\mathbb T}(\phi_m(y))=y-y_m$ for $y \in B(y_m,r_0) \subset {\mathbb T}$, where $\pi_{\mathbb T}$ is the %canonical projection from ${\mathbb Z}^d$ onto $\mathbb T$.
%As $N$ tends to infinity the $(\{0,1\}^{{\mathbb Z}^{d+1}} \times {\mathbb R}_+)^M$-valued random variables
%\begin{align}
%(\omega^m_{\alpha N^{2d}}, L^{z_m}_{\alpha N^{2d}}/N^{2d})_{m=1}^M, \textrm{ } \alpha > 0,
%\end{align}
%as defined in (\ref{def:pic}), converge in joint distribution under $P$ to the law of the random vector
%\begin{align}
%(\omega_m, (d+1)U_m)_{m=1}^M,
%\end{align}
%where $(U_m)_{m=1}^M$ is distributed as $(L(v_m, \alpha/(d+1)))_{m=1}^M$ under $W$, and conditionally on $(U_m)_{m=1}^M$, the random variables %$\omega_m$ are independent with distribution ${\mathbb Q}^{{\mathbb Z}^d}_{U_m}$.
%\end{theorem}

\subsection{The $d$-dimensional box}

The $d$-dimensional box is defined as the graph with vertices
\begin{align*}
G_N = {\mathbb Z}^d \cap [0,N-1]^d, \textrm{ for } d \geq 2,
\end{align*}
and edges between any two vertices at Euclidean distance $1$. In contrast to the similar integer torus considered in \cite{S08}, the box admits different limit models for the local pictures, depending on how many coordinates $y^i_m$ of the points $y_m$ are near the boundary.

\begin{theorem} \label{thm:box}
Consider $x_{m,N}$, $1 \leq m \leq M$, in $G_N \times {\mathbb Z}$ satisfying \textup{\ref{hyp:dist}} and \textup{\ref{hyp:zconv}}, and assume that
for any $1 \leq m \leq M$, there is a number $0 \leq d(m) \leq d$, such that
\begin{align}
&y_{m,N}^i \wedge (N-y_{m,N}^i) \textrm{ is constant for } 1 \leq i \leq d(m) \textrm{ and all large $N$}, \label{box3} \\
&\lim_N y_{m,N}^i \wedge (N-y_{m,N}^i) = \infty \textrm{ for } d(m) < i \leq d. \label{box4}
\end{align}
Then the conclusion of Theorem~\ref{thm:d} holds with ${\mathbb G}_m = {\mathbb Z}_+^{d(m)} \times {\mathbb Z}^{d-d(m)}$ and $\beta = d$. 
%\begin{align*}
%&\Bigl(\omega^m_{\alpha N^{2d}}, \frac{L^{z_m}_{\alpha N^{2d}}}{N^d} \Bigr) \rightharpoonup \Bigl( \otimes_m {\mathbb Q}^{{\mathbb %G}_m}_{L(v_m,\alpha/(d+1))}, (d+1) L \Bigl( v_m, \frac{\alpha}{d+1} \Bigr) \Bigr), \textrm{ for } \alpha>0, \\
%&\textrm{where } {\mathbb G}_m = {\mathbb Z}_+^{d(m)} \times {\mathbb Z}^{d+1-d(m)}.
%\end{align*}
\end{theorem}

\begin{proof}
We check that assumptions \ref{hyp:w}-\ref{hyp:k} and (\ref{hyp:d}) are satisfied and apply Theorem~\ref{thm:d}.
Assumption \ref{hyp:w} is checked immediately. With Lemma~\ref{lem:eval} and standard estimates on $\lambda^d_N$ for simple random walk on $[0,N-1]^d$ (cf.~\cite{SC}, Example~2.1.1. on p.~329 and Lemma~2.2.11, p.~338), we see that $cN^{-2} \leq \lambda_N$, and \ref{hyp:mix} follows. We have assumed \ref{hyp:dist} and \ref{hyp:zconv} in the statement. For \ref{hyp:iso1}, we define the sequence $r_N$, the vertices $o_m \in {\mathbb G}_m$ and the isomorphisms $\phi_m$ by
\begin{align*}
r_N &= \frac{1}{4^M 10} \bigl( \min_{m \neq m'} |x_m-x_{m'}|_\infty \wedge \min_{m} \min_{d(m)<i \leq d} (y^i_m \wedge (N-y^i_m))  \wedge N \bigr), \\
o_m &= (y_m^1 \wedge (N-y_m^1), \ldots, y_m^{d(m)} \wedge (N-y_m^{d(m)}), 0, \ldots, 0), \\
\phi_m(y) &= (y^1 \wedge (N-y^1), \ldots, y^{d(m)} \wedge (N-y^{d(m)}), y^{d(m)+1} - y^{d(m)+1}_m, \ldots, y^d - y^d_m). 
\end{align*}
Then $r_N \to \infty$ by \ref{hyp:dist} and (\ref{box4}), $o_m$ remains fixed by (\ref{box3}), $\phi_m$ is an isomorphism from $B(y_m,r_N)$ to $B(o_m,r_N)$ for large $N$, and \ref{hyp:iso1} follows. 
Recall that a crucial step in the proof of Theorem~\ref{thm:d} was to prove that the random walk, when started at the boundary of one of the balls $B_m$, does not return to the close vicinity of the point $x_m$ before exiting $G \times [-h_N,h_N]$, see Lemma~\ref{lem:cap2}, (\ref{eq:ch5.1}) and below. In the present context, $h_N$ is roughly of order $N$, see (\ref{s1}). However, the radius ${\mathsf r}_N$ of the ball $B_m$ can be required to be much smaller if the distances between different points diverge only slowly, cf.~(\ref{eq:rs}). We therefore needed to assume that larger neighborhoods $C_m \times {\mathbb Z}$ of the points $x_m$ are sufficiently transient by requiring that the sets ${\bar C}_m$ are isomorphic to subsets of suitable infinite graphs ${\hat {\mathbb G}}_m$. In the present context, we choose ${\hat {\mathbb G}}_m = {\mathbb Z}^d_+$ for all $m$, see Remark~\ref{rem:aux} below on why a choice different from ${\mathbb G}_m$ is required. We choose the sets $C_m$ with the help of Lemma~\ref{lem:met}. Applied to the points $y_1, \ldots, y_m$, with $a = \frac{1}{4^M 10} N$ and $b=2$, Lemma~\ref{lem:met} yields points $y^*_1, \ldots, y^*_{M}$ (some of them may be identical) and a $p$ between $\frac{1}{4^M 10} N$ and $\frac{1}{10} N$, such that 
\begin{align}
\textrm{either $C_m = C_{m'}$ or $C_m \cap C_{m'} = \emptyset$ for $C_m = B(y^*_m, 2p), \, 1 \leq m \leq M$,} \label{eq:box1}
\end{align}
and such that the balls with the same centers and radius $p$ still cover $\{y_1, \ldots, y_M\}$. Since $r_N \leq p$, we can associate to any $m$ one of the sets $C_m$ such that \ref{hyp:inc} is satisfied. The diameter of ${\bar C}_m$ is at most $2N/5 +3$, so each of the one-dimensional projections $\pi_k({\bar C}_m)$, $1 \leq k \leq d$, of ${\bar C}_m$ on the $d$ different axes contains at most one of the two integers $0$ and $N-1$ for large $N$. Hence, there is an isomorphism $\psi_m$ from ${\bar C}_m$ into ${\mathbb Z}^d_+$ such that \ref{hyp:iso2} is satisfied.
Assumption \ref{hyp:disj} directly follows from from (\ref{eq:box1}). We now turn to \ref{hyp:trans}. By embedding ${\mathbb Z}^d_+$ into ${\mathbb Z}^d$, one has for any ${\mathtt y}$ and ${\mathtt y}'$ in ${\mathbb Z}^d_+$,
\begin{align*}
p^{{\mathbb Z}^d_+}_n ({\mathtt y},{\mathtt y}') \leq 2^d \sup_{{\mathtt y},{\mathtt y}' \in {\mathbb Z}^d} p^{{\mathbb Z}^d}_n ({\mathtt y},{\mathtt y}') \leq c(d) n^{-d/2},
\end{align*}
using the standard heat kernel estimate for simple random walk on ${\mathbb Z}^d$, see for example \cite{L91}, p.~14, (1.10). Since $d \geq 2$, this is more than enough for \ref{hyp:trans}. In order to check \ref{hyp:k}, it is sufficient to prove the hypothesis (\ref{kch}) of Lemma~\ref{lem:kch}. To this end, we compare the probability $P^G_{y_0}$ with $P^{{\mathbb Z}^d}_{y_0}$ under which the canonical process $(Y_n)_{n \geq 0}$ is a simple random walk on ${\mathbb Z}^d$. We define the map $\pi: {\mathbb Z}^d \to G_N$ by $\pi((y_i)_{1 \leq i \leq d}) = (\min_{k \in {\mathbb Z}} |y_i-2kN|)_{1 \leq i \leq d}$, i.e. in each coordinate, $\pi$ is a sawtooth map. Then $(Y_n)_{n \geq 0}$ under $P^G_{y_0}$ has the same distribution as $(\pi(Y_n))_{n \geq 0}$ under $P^{{\mathbb Z}^d}_{y_0}$. It follows that for $y_0 \in \partial(C_m^c)$, $y \in B(y_m,\rho_0)$,
\begin{align}
&p^G_n(y_0,y) = \sum_{y' \in {\mathcal S}_y} p^{{\mathbb Z}^d}_n (y_0,y'),  \textrm{ where } {\mathcal S}_y =  2N {\mathbb Z}^d + \Bigl\{ \sum_{1 \leq i \leq d} l_i e_i y^i : l \in \{-1,1\}^d \Bigr\}, \label{eq:box3}
\end{align}
The probability in this sum is bounded by
\begin{align*}
 \frac{c}{n^{d/2}} \exp \Bigl\{ \frac{c'|y_0-y'|^2}{n} \Bigr\},
\end{align*}
as follows, for example, from Telcs \cite{T06}, Theorem~8.2 on p.~99, combined with the on-diagonal estimate from the local central limit theorem (cf.~\cite{L91}, p.~14, (1.10)). If we insert this bound into (\ref{eq:box3}) and split the sum into all possible distances between $y_0$ and $y'$ (necessarily this distance is at least $p - \rho_0 \geq cN$, cf.~(\ref{eq:box1})), we obtain
\begin{align*}
p^G_n(y_0,y) &\leq \sum_{k \geq 1} \frac{c}{n^{d/2}} \exp \Big\{ - \frac{c' k^2 N^2}{n} \Bigr\} k^{d-1} \leq \frac{c}{n^{d/2}} \int_0^\infty x^{d-1} \exp \Bigl\{ - \frac{c' x^2 N^2}{n} \Bigr\} dx \leq \frac{c}{N^d}.
\end{align*}
By $cN^{-2} \leq \lambda_N$, checked under \ref{hyp:mix} above, this is more than enough to imply (\ref{kch}), hence \ref{hyp:k}. 
Finally, one immediately checks that (\ref{hyp:d}) holds with $\beta = d$. 
Hence, Theorem~\ref{thm:d} applies and yields the result.
\end{proof}

\begin{remark} \label{rem:aux}
\textup{In the last proof, we have used the possibility of choosing the auxiliary graphs ${\hat {\mathbb G}}_m$ in assumption \ref{hyp:iso2} different from the graphs ${\mathbb G}_m$ in \ref{hyp:iso1}. This is necessary for the following reason: To check assumption \ref{hyp:k}, we need the diameter of each set ${\bar C}_m$ to be of order $N$ in the above argument. Hence, the set ${\bar C}_m$ can look quite different from the ball $B(y_m,r_N)$. Indeed, ${\bar C}_m$ may touch the boundary of the box $G$ in more dimensions than its much smaller subset $B(y_m, r_N)$. As a result, ${\bar C}_m$ may not to be isomorphic to a neighborhood in the same graph ${\mathbb G}_m$ as $B(y_m,r_N)$. However, our chosen ${\bar C}_m$  is always isomorphic to a neighborhood in ${\mathbb Z}_+^d$ for all $m$.}
\end{remark}

\subsection{The Sierpinski graph}

For $y \in {\mathbb R}^2$ and $\theta \in [0,2\pi)$, we denote by $\rho_{y,\theta}$ the anticlockwise rotation around $y$ by the angle $\theta$.
The vertex-set of the Sierpinski graph $G_N$ of depth $N$ is defined by the following increasing sequence (see also the top of Figure~\ref{fig:s}):
\begin{align*}
G_0 &= \{s_0=(0,0), s_1=(1,0), s_2= \rho_{(0,0),\pi/3}(s_1)\} \subset {\mathbb R}^2, \\
G_{N+1} &=  G_{N} \cup (\rho_{2^Ns_1,4\pi/3} G_{N}) \cup (\rho_{2^Ns_2,2\pi/3}G_{N}), \textrm{ for } N \geq 0.
\end{align*}
\begin{figure}[h]
\psfrag{G3}[tl][tl][3][0]{$G_3:$}
\psfrag{G+}[tl][tl][3][0]{${\mathbb G}_\infty^+:$}
\psfrag{G}[tl][tl][3][0]{${\mathbb G}_\infty:$}
\psfrag{0}[tl][tl][3][0]{$s_0=(0,0)$}
\psfrag{1}[tl][tl][3][0]{$2^2 s_1$}
\psfrag{2}[tl][tl][3][0]{$2^3 s_1$}
\psfrag{3}[tl][tl][3][0]{$2^3 s_2$}
\psfrag{4}[tl][tl][3][0]{$2^2 s_2$}
\psfrag{5}[tl][tl][3][0]{$2^2 (s_1+ s_2)$}
\psfrag{o}[tl][tl][3][0]{$o_m$}
\begin{center}
\includegraphics[angle=0,
width=0.7\textwidth]{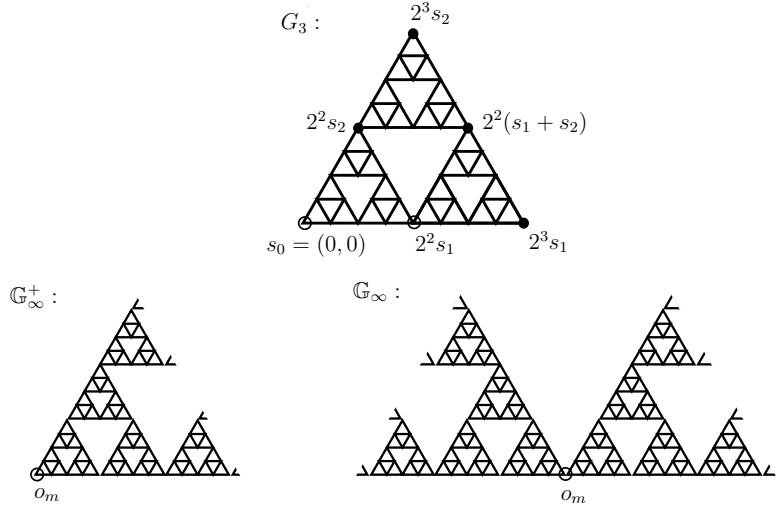}
\end{center}
\caption{An illustration of $G_3$ (top) and the infinite limit models ${\mathbb G}^+_\infty$ (bottom left) and ${\mathbb G}_\infty$ (bottom right).} \label{fig:s}
\end{figure}

\noindent The edge-set of $G_N$ contains an edge between every pair of vertices in $G_N$ at Euclidean distance $1$. Note that the vertices in $2^NG_0 \subset G_N$ have degree $2$ and all other vertices of $G_N$ have degree $4$. 

Denoting the reflection around the $y$-axis by $\sigma$, i.e.~$\sigma((y_1,y_2)) = (-y_1,y_2)$ for $(y_1,y_2) \in {\mathbb R}^2$, the two-sided infinite Sierpinski graph has vertices
\begin{align*}
{\mathbb G}_\infty &= {\mathbb G}^+_\infty \cup \sigma {\mathbb G}^+_\infty, \textrm{ where } {\mathbb G}^+_\infty = \cup_{N \geq 0} G_N,
\end{align*}
and an edge between any pair of vertices in ${\mathbb G}^+_\infty$ or in $\sigma {\mathbb G}^+_\infty$ at Euclidean distance $1$. We refer to the bottom of Figure~\ref{fig:s} for illustrations. For $N \geq 0$, we define the surjection $s_N: G_{N+1} \to G_N$ by
\begin{align*}
 s_N(y) = \left\{ \begin{array}{ll}
                    y & \textrm{for } y \in G_N,\\
\rho_{2^Ns_1,2\pi/3}(y) & \textrm{for } y \in \rho_{2^Ns_1,4\pi/3}(G_N) \setminus G_N \\
\rho_{2^Ns_2,4\pi/3}(y) & \textrm{for } y \in \rho_{2^Ns_2,2\pi/3}(G_N) \setminus G_N.
                   \end{array} \right.
\end{align*}
We then define the mapping $\pi_N$ from ${\mathbb G}^+_\infty$ onto $G_N$ by 
\begin{align*}
 \pi_N(y) = s_N \circ s_{N+1} \circ \cdots \circ s_{m-1} (y) & \textrm{ for } y \in G_m \textrm{ with } m>N.
\end{align*}
Note that $\pi_N$ is well-defined: Indeed, the vertex-sets $G_N$ are increasing in $N$ and if $y \in G_{m_1} \subset G_{m_2}$ for $N < m_1 < m_2$, then $s_k(y) = y$ for $k \geq m_1$, so that  $s_N \circ \cdots \circ s_{m_2-1} (y) = s_N \circ \cdots \circ s_{m_1-1} (y)$. We will use the following lemma: 

\begin{lemma} \label{lem:sdist}
For any ${\mathtt y} \in {\mathbb G}^+_\infty$, the distribution of the random walk $(Y_n)_{n \geq0}$ under $P^{G_N}_{\pi_N({\mathtt y})}$ is equal to the distribution of the random walk $(\pi_N(Y_n))_{n \geq 0}$ under ${\mathbb P}^{{\mathbb G}^+_\infty}_{\mathtt y}$.
\end{lemma}

\begin{proof}
 The result follows from the Markov property once we check that for any $y,y' \in G_N$, ${\mathtt y} \in {\mathbb G}^+_\infty$ with $y = \pi_N({\mathtt y})$,
\begin{align}
 p^{G_N}(y,y') = \sum_{y'_1 \in {\pi_N}^{-1}(y')} p^{{\mathbb G}^+_\infty}({\mathtt y},y'_1).
\end{align}
We choose $m \geq N$ such that ${\mathtt y} \in G_{m}$. Then the right-hand side equals
\begin{align*}
 \sum_{y'_1 \in {\pi_N}^{-1}(y')} p^{^{G_{m+1}}}({\mathtt y},y'_1) &= \sum_{y'_1 \in s_N^{-1}(y')} \sum_{y'_2 \in s_{N+1}^{-1}(y'_1)} \cdots \sum_{y'_m \in s_{m}^{-1}(y'_{m-N})} p^{^{G_{m+1}}}({\mathtt y},y'_m).
\end{align*}
By induction on $m$, it hence suffices to show that for $y,y' \in G_m$ and ${\hat y} \in s_m^{-1}(y)$,  
\begin{align}
  p^{G_m}(y,y') = \sum_{y'_1 \in s_m^{-1}(y') \cap B({\hat y},1) \subset G_{m+1}} p^{^{G_{m+1}}}({\hat y},y'_1). \label{eq:sdist1}
\end{align}
If ${\hat y} \in G_{m+1} \setminus \{2^ms_1, 2^m s_2, 2^m(s_1+s_2)\}$, then (\ref{eq:sdist1}) follows from the observation that $s_m$ maps the distinct neighbors of $\hat y$ in $G_{m+1}$ to the distinct neighbors of $y$ in $G_m$. If ${\hat y} \in   \{2^ms_1, 2^m s_2, 2^m(s_1+s_2)\}$, then $\hat y$ has four neighbors in $G_{m+1}$, two of which are mapped to each of the two neighbors of $y \in \{2^ms_1,2^ms_2,(0,0)\}$ in $G_m$ and this implies again (\ref{eq:sdist1}).
\end{proof}

In the following theorem, we consider points $y_m$ that are either the corner $(0,0)$ or the vertex $(2^{N-1},0)$ and obtain the two different limit models ${\mathbb G}^+_\infty \times {\mathbb Z}$ and ${\mathbb G}_\infty \times {\mathbb Z}$ for the corresponding local pictures. 
\begin{theorem}
Consider $0 \leq M' \leq M$ and vertices $x_{m,N}$, $1 \leq m \leq M$, in $G_N \times {\mathbb Z}$ satisfying \textup{\ref{hyp:dist}} and \textup{\ref{hyp:zconv}} and assume that
%such that
\begin{align}
&y_{m,N} = (0,0), \textrm{ for $1 \leq m \leq M'$, and } y_{m,N} = (2^{N-1},0), \textrm{ for $M' < m \leq M$.} \label{si1}
\end{align}
Then the conclusion of Theorem \ref{thm:d} holds with ${\mathbb G}_m = {\mathbb G}^+_\infty$ for $1 \leq m \leq M'$, ${\mathbb G}_m = {\mathbb G}_\infty$ for $M' < m \leq M$ and $\beta = 2$.
%Then with the notation of (\ref{thm:c}),
%\begin{align*}
%&\Bigl( \omega^m_{\alpha |S_N|^2}, \frac{L^{z_m}_{\alpha |S_N|^2}}{|S_N|^2} \Bigr) \rightharpoonup \Bigl( \otimes_m {\mathbb Q}^{{\mathbb G}_m}_{L(v_m, %\alpha/3)}, 3 L \Bigl( v_m, \frac{\alpha}{3} \Bigr) \Bigr), \textrm{ for } \alpha >0, \textrm{ where }\\
%&{\mathbb G}_m = S^+_\infty \times {\mathbb Z} \textrm{ for } 1 \leq m \leq M', \textrm{ and } {\mathbb G}_m = S_\infty \times {\mathbb Z} \textrm{ for %} M' < m \leq M.
%\end{align*}
\end{theorem}

\begin{proof}
Let us again check that the hypotheses \ref{hyp:w}-\ref{hyp:k} and (\ref{hyp:d}) are satisfied. 
One easily checks that \ref{hyp:w} holds with $c_0 = 1$ and $c_1 = 2$. 
Using Lemma~\ref{lem:eval} and the explicit calculation of $\lambda^d_N$ by Shima \cite{SH91}, we find that $c 5^{-N} \leq \lambda_N \leq c' 5^{-N}.$ Indeed, in the notation of \cite{SH91}, Proposition~3.3 in \cite{SH91} shows that $\lambda^d_N$ is given by $\phi_-^{(N)}(3)$ for the function $\phi_-$ defined above Remark~2.16, using our $N$ in place of $m$ and setting the $N$ of \cite{SH91} equal to $3$. Then $\lambda^d_N=\phi_-^{(N)}(3)$ is decreasing in $N$ and converges to the fixed point $0$ of $\phi_-$. With Taylor's theorem it then follows that $\lambda^d_N 5^N$ converges to $1$. Since
$|G_N| = 3 + \sum_{n=1}^{N} 3^n \leq c 3^N$, \ref{hyp:mix} holds. 
We have assumed \ref{hyp:dist} and \ref{hyp:zconv} in the statement. For \ref{hyp:iso1}, we define the radius
\begin{align*}
r_N &= \frac{1}{4} ( 2^{N-1} \wedge \min_{1 \leq m < m' \leq M} d(x_m,x_{m'})), \textrm{ and set}\\
o_m &= (0,0) \textrm{, for all } m.
\end{align*}
The balls $B(y_m,r_N) \subset G_N$ intersect $2^NG_0$ only at the points $y_m$, because the distance between different points of $2^N G_0$ equals $2^N$. We can therefore define the isomorphisms $\phi_m$ from $B(y_m,r_N)$ to $B((0,0),r_N) \subset {\mathbb G}_m$ as the identity for $m \leq M'$ and as the translation by $(-2^{N-1},0)$ for $m>M'$ and \ref{hyp:iso1} follows. As in the previous example, the radius ${\mathsf r}_N$ defined in (\ref{eq:rs}) can be small compared with the square root of the relaxation time, so it is essential for the proof that larger neighborhoods $C_m \times {\mathbb Z}$ of the points $x_m$ are sufficiently transient. In the present case, we define the auxiliary graphs as ${\hat {\mathbb G}}_m = {\mathbb G}_m$ and $C_m = B(y_m, 2^{N-1}/3)$ for $1 \leq m \leq M$. Then \ref{hyp:inc} holds, because $r_N < 2^{N-1}/3$ for large $N$ and the isomorphisms $\psi_m$ required for \ref{hyp:iso2} can be defined in a similar fashion as the isomorphisms $\phi_m$ above. Assumption \ref{hyp:disj} is immediate.
We now check \ref{hyp:trans}. It is known from \cite{B00} (see also \cite{J96}) that for any ${\mathtt y}$ and ${\mathtt y}'$ in ${\mathbb G}_\infty$,
\begin{align}
p^{{\mathbb G}_\infty}_n({\mathtt y},{\mathtt y}') \leq c n^{-d_s/ 2} \exp \Bigl\{ - c' \Bigl( \frac{d({\mathtt y},{\mathtt y}')^{d_w}}{n} \Bigr)^{1/(d_w-1)} \Bigr\}, \label{eq:si5}
\end{align}
for $d_s = 2 \log 3/ \log 5$, $d_w = \log 5/ \log 2$ and $n \geq 1$.
Since 
\begin{align}
p^{{\mathbb G}^+_\infty}_n ({\mathtt y}_0,{\mathtt y}) = p^{{\mathbb G}_\infty}_n({\mathtt y}_0,{\mathtt y}) + p^{{\mathbb G}_\infty}_n({\mathtt y}_0, \sigma {\mathtt y}) \label{eq:si6}
\end{align}
and $\log 3/ \log 5>1/2$, this is enough for \ref{hyp:trans}. To prove \ref{hyp:k}, we use Lemma~\ref{lem:kch} and only check (\ref{kch}). To this end, note that $B(y_m,\rho_0) \subseteq K \subseteq G_N$, for $K=\cup_{y' \in 2^{N-1}G_1} B(y', \rho_0)$ and that the preimage of the vertices in $2^{N-k}G_k \subset G_N$ under $\pi_N$ is $2^{N-k} {\mathbb G}^+_\infty$ for $0 \leq k \leq N$. It follows from Lemma~\ref{lem:sdist} that for $y_0 \in \partial (C_m^c)$, $y \in B(y_m,\rho_0) \subseteq K$ and $N \geq c(\rho_0)$,
\begin{align}
p^{G_N}_n(y_0,y) &\leq \sum_{y' \in K} p^{G_N}_n(y_0,y') = \sum_{{\mathtt y}' \in {\mathcal K}} p^{{\mathbb G}^+_\infty}_n(y_0,{\mathtt y}'), \textrm{for } {\mathcal K} = \bigcup_{{\mathtt y} \in 2^{N-1} {\mathbb G}^+_\infty} B({\mathtt y},\rho_0). \label{eq:si4}
\end{align}
Observe now that for any given vertex ${\mathtt y}'$ in ${\mathbb G}_\infty$, the number of vertices in $B({\mathtt y}',2^k) \cap {\mathcal K}$ is less than $c(\rho_0)|B({\mathtt y}',2^k) \cap 2^{N-1}{\mathbb G}^+_\infty| \leq c(\rho_0) 3^{k-N}$. Also, it follows from the choice of $C_m$ that $d(y_0,2^{N-1}{\mathbb G}^+_\infty) \geq c2^{N}$, so the distance between $y_0$ and any point in ${\mathcal K}$ is at least $c(\rho_0)2^N$. Summing over all possible distances in (\ref{eq:si4}), we deduce with the help of (\ref{eq:si5}) and (\ref{eq:si6}) that
\begin{align*}
p^{G_N}_n(y_0,y) &\leq c(\rho_0) \sum_{l=1}^\infty 3^l n^{-d_s/2} \exp \Bigl\{ - c'(\rho_0) \Bigl( \frac{2^{(N+l)d_w}}{n} \Bigr)^{1/(d_w-1)} \Bigr\} \\
&\leq c(\rho_0) n^{-d_s/2} \int_0^\infty 3^x \exp \Bigl\{ - c'(\rho_0) \Bigl( \frac{5^{N+x}}{n} \Bigr)^{1/(d_w-1)} \Bigr\} dx.
\end{align*}
After substituting $x=y - N + \log n/ \log 5$, this expression is seen to be bounded by
\begin{align*}
c(\rho_0)3^{-N} \int_{-\infty}^\infty 3^y \exp \bigl\{ -c'(\rho_0) 5^{y/(d_w-1)} \bigr\} dy \leq c(\rho_0) 3^{-N}.
\end{align*}
By $\sqrt{5} < 3$ and $c5^{-N} \leq \lambda_N$, as we have seen under \ref{hyp:mix}, this is more than enough for (\ref{lem:kch}), hence \ref{hyp:k}.
Finally, it is straightforward to check that (\ref{hyp:d}) holds with $\beta = 2$.
Hence, Theorem \ref{thm:d} applies and yields the result.
\end{proof}

\subsection{The $d$-ary tree}

For a fixed integer $d \geq 2$, we let ${\mathbb G}_o$ be the infinite $d+1$-regular graph without cycles, called the infinite $d$-ary tree. We fix an arbitrary vertex $o \in {\mathbb G}_o$ and call it the root of the tree. See Figure~\ref{fig:tree} (left) for a schematic illustration in the case $d=2$.

\begin{figure}[h]
\psfrag{o}[bl][Bl][2][0]{$o$}
\begin{center}
\includegraphics[angle=0,width=0.7\textwidth]{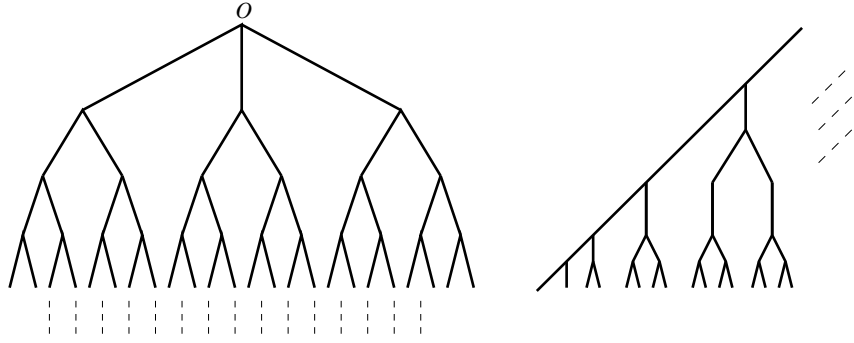}
\end{center}
\caption{A schematic illustration of ${\mathbb G}_o$ (left) and ${\mathbb G}_\lozenge$ (right) for $d=2$.} \label{fig:tree}
\end{figure}

\noindent We choose $G_N$ as the ball of radius $N$ centered at $o \in {\mathbb G}_o$. For any vertex $y$ in $G_N$, we refer to the number $|y|=N-d(y,o)$ as the height of $y$. Vertices in $G_N$ of depth $N$ (or height $0$) are called leaves. 
%For large $N$, any neighbourhood of a vertex near the root in $G_N$ is isomorphic to the neighbourhood of the root in ${\mathbb G}_o$. However, this is not true for vertices near the leaves, so we also introduce the following infinite limit model: 
The boundary-tree ${\mathbb G}_\lozenge$ contains the vertices
\begin{align*}
{\mathbb G}_\lozenge = \{ (k;s): k \geq 0, s \in S_d \},
\end{align*}
where $S_d$ is the set of infinite sequences $s=(s_1,s_2, \ldots)$ in $\{1, \ldots, d\}^{[1, \infty)}$ with at most finitely many terms different from $1$. The graph ${\mathbb G}_\lozenge$ has edges $\{(k;s), (k+1;s')\}$ for vertices $(k;s)$ and $(k+1;s')$ whenever $s_{n+1} = s'_n$ for all $n \geq 1$. In this case, we refer to the number $k = |(k;s)|$ as the height of the vertex $(k;s)$ and to all vertices at height $0$ as leaves. See Figure~\ref{fig:tree} (right) for an illustration of ${\mathbb G}_\lozenge$. The following rough heat-kernel estimates will suffice for our purposes: 
\begin{lemma} \label{lem:trk}
\begin{align}
&p^{{\mathbb G}_o}_n({\mathtt y}_0, {\mathtt y}) \leq e^{-c(d)n}, \label{trk0}\\
&p^{{\mathbb G}_\lozenge}_n({\mathtt y}_0, {\mathtt y}) \leq n^{-3/5} + c(d,|{\mathtt y}|) \exp \{ - c'(d,|{\mathtt y}|) n^{c(d)}\} \textrm{ and }\label{trk1} \\
&p^{G_N}_n(y_0, y) \leq ce^{-c(d)d(y_0,y)} {\mathbf 1}_{n \leq N^3} + c(d) \bigl( d^{-N}+n^{-3/5} \bigr) {\mathbf 1}_{n > N^3}. \label{trk2}
\end{align}
(We refer to the end of the introduction for our convention on constants.)
\end{lemma}

\begin{proof}
The estimate (\ref{trk0}) can be shown by an elementary estimate on the biased random walk $(d(Y_n,{\mathtt y}))_{n \geq 0}$ on $\mathbb N$. More generally, (\ref{trk0}) is a consequence of the non-amenability of ${\mathbb G}_o$, see \cite{W00}, Corollary 12.5,~p.~125.

\vspace{3mm}

We now prove (\ref{trk1}).
Under $P^{{\mathbb G}_\lozenge}_{{\mathtt y}_0}$, the height $|Y|$ of $Y$ is distributed as a random walk on $\mathbb N$ starting from $|y_0|$ with transition probabilities $w_{k,k+1}= \frac{1}{d+1}$, $w_{k,k-1} = \frac{d}{d+1}$ for $k \geq 1$ and reflecting barrier at $0$.  
We set for $n \geq 1$,
\begin{align}
L = \Bigl[ \frac{3}{5 \log d} \log n \Bigr] +1, \label{eq:trk0}
\end{align}
and define the stopping time $S$ as the first time when $Y$ reaches the level $|{\mathtt y}|+L$:
\begin{align*}
S = \inf \{n \geq 0: |Y_n| \geq |{\mathtt y}|+L\}.
\end{align*}
Then we have
\begin{align*}
p^{{\mathbb G}_\lozenge}_n({\mathtt y}_0, {\mathtt y}) \leq 
P^{{\mathbb G}_\lozenge}_{{\mathtt y}_0} [S \leq n, Y_n = {\mathtt y}] + P^{{\mathbb G}_\lozenge}_{|{\mathtt y}_0|} [S>n], \textrm{ for } n \geq 0.
\end{align*}
Observe that the second probability on the right-hand side can only increase if we replace $|{\mathtt y}_0|$ by $0$. We now apply the simple Markov property and this last observation at integer multiples of the time $|{\mathtt y}|+L$ to the second probability and the strong Markov property at time $S$ to the first probability on the right-hand side and obtain
\begin{align}
p^{{\mathbb G}_\lozenge}_n({\mathtt y}_0, {\mathtt y}) \leq E^{{\mathbb G}_\lozenge}_{{\mathtt y}_0} \bigl[S \leq n, P^{{\mathbb G}_\lozenge}_{Y_S} [Y_m = {\mathtt y}] \bigl|_{m=n-S} \bigr] + P^{{\mathbb G}_\lozenge}_0 [S>|{\mathtt y}|+L]^{[ \frac{n}{|{\mathtt y}|+L} ]}. \label{eq:trk1}
\end{align}
The second probability on the right-hand side is equal to $1-(d+1)^{-(|{\mathtt y}|+L)}$.
In order to bound the expectation, note that by definition of $S$, there are $d^L$ descendants ${\mathtt y}'$ of $Y_S$ at the same height as ${\mathtt y}$, and the $P^{{\mathbb G}_\lozenge}_{Y_S}$-probability that $Y_m$ equals ${\mathtt y}'$ is the same for all such ${\mathtt y}'$. Hence, the expectation on the right-hand side of (\ref{eq:trk1}) is bounded by $d^{-L}.$
We have hence shown that
\begin{align*}
p^{{\mathbb G}_\lozenge}_n({\mathtt y}_0, {\mathtt y}) \leq \Bigl( \frac{1}{d} \Bigr)^L + \Bigl( 1 - \Bigl(\frac{1}{d+1} \Bigr)^{|{\mathtt y}|+L} \Bigr)^{[\frac{n}{|{\mathtt y}|+L}]}.
\end{align*}
Substituting the definition of $L$ from (\ref{eq:trk0}) and using that $\frac{\log(d+1)}{\log d} \leq \frac{\log 3}{\log 2} < \frac{5}{3}$ for the second term, one finds (\ref{trk1}). 

\vspace{3mm}

We now come to (\ref{trk2}) and first treat the case $n \leq N^3$. By uniform boundedness and reversibility of the measure $y \mapsto w_y$, we have $p^{G_N}_n(y_0,y) \leq c p^{G_N}_n(y,y_0)$, so we can freely exchange $y_0$ and $y$ in our estimates. In particular, we can assume that $d(y_0,o) \leq d(y,o)$. Now we denote by $y_1$ the first vertex at which the shortest path from $y_0$ to $o$ meets the shortest path from $y$ to $o$. Then any path from $y_0$ to $y$ must pass through $y_1$. From the strong Markov property applied at time $H_{y_1}$, it follows that
\begin{align}
 p^{G_N}_n(y_0,y) = E^{G_N}_{y_0} \bigl[ \{H_{y_1} \leq n\}, P^{G_N}_{H_{y_1}}[ Y_k = y] \bigl|_{k=n-H_{y_1}} \bigr]. \label{eq:trk1.1}
\end{align}
The $P^{G_N}_{H_{y_1}}$-probability on the right-hand side remains unchanged if $y$ is replaced by any of the $d^{d(y_1,y)}$ descendants $y'$ of $y_1$ at the same height as $y$. Moreover, the assumption $d(y_0,o) \leq d(y,o)$ implies that $d(y_1,y) \geq d(y_1,y_0)$, hence $2d(y_1,y) \geq d(y_0,y)$. In particular, there are at least $d^{d(y_0,y)/2}$ different vertices $y'$ for which $P^{G_N}_{H_{y_1}}[ Y_k = y] = P^{G_N}_{H_{y_1}}[ Y_k = y']$.
By (\ref{eq:trk1.1}), this proves the estimate (\ref{trk2}) for $n \leq N^3$. We now treat the case $n>N^3$. The argument used to prove (\ref{trk1}) with $(|y|+L) \wedge N$ playing the role of $|{\mathtt y}| + L$ yields
\begin{align}
 p^{G_N}_n(y_0,y) \leq c(d,|y|) \bigl(d^{-N} \vee n^{-3/5} + e^{-c(d,|y|)n^{c(d)}} \bigr). \label{eq:trk2}
\end{align}
The assumption $n > N^3$ will now allow us to remove the dependence on $|y|$ of the right-hand side. By applying the strong Markov property at the entrance time $H_{\partial B(o,N-1)}$ of the random walk into the set $\partial B(o,N-1)$ of leaves of $G_N$, we have
\begin{align}
 p^{G_N}_n (y_0,y) \leq P^{G_N}_{y_0} [H_{\partial B(o,N-1)} > N^3 /2] + \sup_{y': |y'|=0} \sup_{n-N^3/2 \leq k \leq n} p^{G_N}_k (y',y), \textrm{ for } n > N^3. \nonumber
\end{align}
Applying reversibility to exchange $y'$ and $y$, then (\ref{eq:trk2}) to the second term, we infer that
\begin{align}
 p^{G_N}_n (y_0,y) \leq P^{G_N}_{y_0} [H_{\partial B(o,N-1)} > N^3 /2] + c(d) \bigl( d^{-N} + n^{-3/5} \bigr), \textrm{ for $n > N^3$,} \label{eq:trk2.1}
\end{align}
where we have used that $e^{-c(d)n^{c(d)}} \leq c(d)n^{-2/3}$. In order to bound the first term on the right-hand side, we apply the Markov property at integer multiples of $10N$ and obtain
\begin{align}
 P^{G_N}_{y_0} [H_{\partial B(o,N-1)} > N^3 /2] &\leq \sup_{y \in G_N} P^{G_N}_y [H_{\partial B(o,N-1)} > 10N]^{cN^2}. \label{eq:trk3}
\end{align}
Note that the random walk on ${\mathbb G}_o \supset G_N$, started at any vertex $y$ in $G_N = B(o,N)$, must hit $\partial B(o,N-1)$ before exiting $B(y,2N)$. Applying this observation to the probability on the right-hand side of (\ref{eq:trk3}), we deduce with (\ref{eq:trk2.1}) that
\begin{align}
 p^{G_N}_n (y_0,y) \leq P^{{\mathbb G}_o}_o [T_{B(o,2N)} > 10N]^{cN^2} + c(d) \bigl( d^{-N} + n^{-3/5} \bigr), \textrm{ for $n > N^3$.}\nonumber
\end{align}
The probability on the right-hand side is bounded by the probability that a random walk on $\mathbb Z$ with transition probabilities $p_{z,z+1}=d/(d+1)$ and $p_{z,z-1} = 1/(d+1)$ starting at $0$ is at a site in $(-\infty,2N]$ after $10N$ steps. From the law of large numbers applied to the iid increments with expectation $(d-1)/(d+1) \geq 1/3$ of such a random walk, it follows that this probability is bounded from above by $1-c < 1$ for $N \geq c'$, hence bounded by $1-c''<1$ for all $N$ (by taking $1-c'' = (1-c) \vee \max \{P^{{\mathbb G}_o}_o [T_{B(o,2N)} > 10N]: N < c'\}$). It follows that 
\begin{align*}
 p^{G_N}_n (y_0,y) \leq e^{-c(d)N^2} + c(d) \bigl( d^{-N} + n^{-3/5} \bigr) \leq c(d) \bigl( d^{-N} + n^{-3/5} \bigr), \textrm{ for } n > N^3.
\end{align*}
This completes the proof of (\ref{trk2}) and of Lemma~\ref{lem:trk}.
\end{proof}
We now consider vertices $y_m$ in $G_N$ that remain at a height that is either of order $N$ or constant. This gives rise to the two different transient limit models ${\mathbb G}_o \times {\mathbb Z}$ and ${\mathbb G}_\lozenge \times {\mathbb Z}$.

\begin{theorem} \label{thm:tree}
\textup{($d \geq 2$)} 
Consider vertices $x_{m,N}$, $1 \leq m \leq M$, in $G_N \times {\mathbb Z}$ satisfying \textup{\ref{hyp:dist}} and \textup{\ref{hyp:zconv}}
%\begin{align}
%&\lim_N \min_{1 \leq m < m' \leq M} |x_m - x_{m'}|_\infty = \infty, \textrm{ and } \label{tr1}\\
%&\lim_N \frac{z_m}{|T|^2} = v_m \in {\mathbb R} \textrm{ for } 1 \leq m \leq M. \label{tr2}
%\end{align}
and assume that for some number $0 \leq M' \leq M$ and some $\delta \in (0,1)$,
\begin{align}
&\liminf_N |y_{m,N}|/N > \delta,  \textrm{ for $1 \leq m \leq M'$, and}\label{tr:bd}\\
&|y_{m,N}|  \textrm{ is constant for $M' < m \leq M$ and large $N$.}\label{tr:d}
\end{align}
Then the conclusion of Theorem~\ref{thm:d} holds with ${\mathbb G}_m = {\mathbb G}_o$ for $1 \leq m \leq M'$, ${\mathbb G}_m = {\mathbb G}_\lozenge$ for $M' < m \leq M$ and $\beta=1$. 
%Then in the notation of (\ref{thm:d}),
%\begin{align*}
%&\Bigl( \omega^m_{\alpha |T|^2}, \frac{L^{z_m}_{\alpha |T|^2}}{|T|^2} \Bigr) \rightharpoonup \Bigl( \otimes_m {\mathbb Q}^{{\mathbb %G}_m}_{L(v_m,\alpha/2)}, 2 L \Bigl(v_m, \frac{\alpha}{2} \Bigr) \Bigr), \textrm{ for } \alpha > 0, \textrm{ where }\\
%&{\mathbb G}_m = T^\lozenge \textrm{ for } 1 \leq m \leq M' \textrm{ and } {\mathbb G}_m = T^\circ \textrm{ for } M' < m \leq M.
%\end{align*}
\end{theorem}

\begin{proof}
Once more, we check \ref{hyp:w}-\ref{hyp:k} and (\ref{hyp:d}) and apply Theorem~\ref{thm:d}. 
It is immediate to check \ref{hyp:w}. For the estimate \ref{hyp:mix}, the degree of the root of the tree does not play a role, as can readily be seen from the definition (\ref{def:spec}) of $\lambda_N$. We can hence change the degree of the root from $d+1$ to $d$ and apply the estimate from Aldous and Fill in \cite{AF}, Chapter 5, p.~26, equation (59). Combined with Lemma~\ref{lem:eval} relating the discrete- and continuous time spectral gaps, this shows that $c(d) |G_N|^{-1} \leq \lambda_N$. In particular, \ref{hyp:mix} holds.
We are assuming \ref{hyp:dist} and \ref{hyp:zconv} in the statement.
For \ref{hyp:iso1}, we define 
\begin{align*}
r_N &= \frac{1}{4^M10} \bigl( \min_{1 \leq m < m' \leq M} d(x_m,x_{m'}) \wedge \delta N \bigr), \textrm{ as well as}\\
o_m &= o \textrm{ for $1 \leq m \leq M'$ and } o_m = (|y_m|;{\mathbf 1}), \textrm{ for $M'<m \leq M$,}
\end{align*}
where $\mathbf 1$ denotes the infinite sequence of ones.
Then for $1 \leq m \leq M'$, the ball $B(y_m,r_N)$ does not contain any leaves of $G_N$ for large $N$, so there is an isomorphism $\phi_m$ mapping $B(y_m,r_N)$ to $B(o,r_N) \subset {\mathbb G}_o$. For $M' < m \leq M$, note that assumption (\ref{tr:d}) and the choice of $r_N$ imply that for large $N$, all vertices in the ball $B(y_m,r_N)$ have a common ancestor $y_* \in G_N \setminus (B(y_m,r_N) \cup \{o\})$ (we can define $y_*$ as the first vertex not belonging to $B(y_m,r_N)$ on the shortest path from $y_m$ to $o$). We now associate a label $l(y)$ in $\{1, \ldots, d\}$ to all descendants $y$ of $y_*$ in the following manner: We label the $d$ children of $y_*$ by $1, \ldots, d$ such that the vertex belonging to the shortest path from $y_*$ to $y_m$ is labelled $1$. We then do the following for any descendant $y$ of $y_*$: If one of the children of $y$ belongs to the shortest path from $y_*$ to $y_m$, we associate the label $1$ to this child and associate the labels $2, \ldots, d$ to the remaining $d-1$ children in an arbitrary fashion. If none of the children of $y$ belong to the shortest path from $y_*$ to $y_m$, we label the $d$ children of $y$ by $1, \ldots, d$ in an arbitrary fashion. Having labelled all descendants of $y$ in this way, we define for any descendant $y$ of $y_*$ the finite sequence $s(y)$ by $l(y),l(y_1), \ldots, l(y_{d(y,y_*)-1})$, where $(y,y_1, \ldots, y_{d(y,y_*)-1}, y_*)$ is the shortest path from $y$ to $y_*$.
Then the function $\phi_m$ from $B(y_m,r_N)$ to ${\mathbb G}_\lozenge$, defined by
\begin{align}
\phi_m(y)= (|y|; s(y),1,1,\ldots), \label{eq:tr0.0}
\end{align}
is an isomorphism from $B(y_m,r_N)$ into ${\mathbb G}_\lozenge$ mapping $y_m$ to $(|y_m|;{\mathbf 1})$, as required. Hence, \ref{hyp:iso1} holds. As in the previous examples, we now choose the sets $C_m$ ensuring that the probability of escaping to the complement of a large box from the boundaries of $B_m$ (cf.~(\ref{def:balls})) is large. We define the auxiliary graphs as ${\hat {\mathbb G}}_m = {\mathbb G}_m$. 
As in the example of the box, we then apply Lemma~\ref{lem:met} to find the required sets $C_m$. Applied to the points $y_1, \ldots, y_m$, with $a =  \frac{\delta}{4^M10} N$ and $b=2$, Lemma~\ref{lem:met} yields points $y^*_1, \ldots, y^*_M$, some of which may be identical, and a $p$ between $ \frac{\delta}{4^M10} N$ and $\frac{\delta}{10} N$ such that
\begin{align}
 \textrm{either } C_m = C_{m'} \textrm{ or } C_m \cap C_{m'} = \emptyset \textrm{ for } C_m =  B(y^*_m, 2p), \, 1 \leq m \leq M, \label{eq:tr0}
\end{align}
and such that the balls with the same centers and radius $p$ still cover $\{y_1, \ldots, y_M\}$. Since $r_N \leq p$, we can associate a set $C_m$ to any $B(y_m,r_N)$ such that \ref{hyp:inc} holds. Concerning \ref{hyp:iso2}, note that the definition of $r_N$ immediately implies that ${\bar C}_m$ contains leaves of $G_N$ if and only if $m > M'$ and in this case all vertices in ${\bar C}_m$ have a common ancestor in $G_N \setminus ({\bar C}_m \cup \{o\})$ (one can take the first vertex not belonging to ${\bar C}_m$ on the shortest path from $y_m$ to $o$). We can hence define the isomorphisms $\psi_m$ from ${\bar C}_m$ into ${\hat {\mathbb G}}_m$ in the same way as we defined the isomorphisms $\phi_m$ above, so \ref{hyp:iso2} holds. Assumption \ref{hyp:disj} directly follows from (\ref{eq:tr0}). We now turn to \ref{hyp:trans}. For $1 \leq m \leq M'$, this assumption is immediate from (\ref{trk0}). For $M'<m \leq M$, note that the isomorphism $\psi_m$, defined in the same way as $\phi_m$ in (\ref{eq:tr0.0}), preserves the height of any vertex.  In particular, $|\psi_m(y_m)|$ remains constant for large $N$ by (\ref{tr:d}) and the estimate required for \ref{hyp:trans} follows from (\ref{trk1}).
In order to check \ref{hyp:k}, we again use Lemma~\ref{lem:kch} and only verify (\ref{kch}). Note that for any $1 \leq m \leq M$, the distance between vertices $y_0 \in \partial (C_m^c)$ and $y \in B(y_m, \rho_0)$ is at least $c(\delta, M, \rho_0) N$. With the estimate (\ref{trk2}) and the bound on $\lambda_N^{-1}$ shown under \ref{hyp:mix}, we find that the sum in (\ref{kch}) is bounded by
\begin{align*}
 N^3 cd^{-c(\delta, M, \rho_0)N} + c(d) \Bigl( |G_N|^{-(1-\epsilon)/2} + \sum_{n=N^3}^{\infty}  n^{-3/5-1/2}  \Bigr),
\end{align*}
which tends to $0$ as $N$ tends to infinity for $0 < \epsilon <1$. We have thus shown that \ref{hyp:k} holds.
Finally, we check (\ref{hyp:d}). To this end, note first that all vertices in $G_{N-1} \subset G_N$ have degree $d+1$ in $G_N$, and the remaining vertices of $G_N$ (the leaves) have degree $1$. Hence,
\begin{align}
 \frac{w(G_N)}{|G_N|} = \frac{|G_{N-1}|}{|G_N|} \frac{d+1}{2} + \Bigl( 1- \frac{|G_{N-1}|}{|G_N|} \Bigr) \frac{1}{2}. \label{eq:tr1}
\end{align}
Now $G_N$ contains one vertex of depth $0$ (the root) and $(d+1)d^{k-1}$ vertices of depth $k$ for $k=1, \ldots, N$. It follows that $|G_N| = 1 + (d+1) (1+ d+ \ldots + d^{N-1}) = 1 + \frac{d+1}{d-1} (d^N-1)$ and that $\lim_N |G_{N-1}|/|G_N| = 1/d$. With (\ref{eq:tr1}), this yields
\begin{align*}
 \lim_N \frac{w(G_N)}{|G_N|} = \frac{d+1}{2d} + \frac{d-1}{2d} = 1.
\end{align*}
Therefore, (\ref{hyp:d}) holds with $\beta=1$. The result follows by application of Theorem~\ref{thm:d}.
\end{proof}

\begin{remark}
\textup{
The last theorem shows in particular that the parameters of the Brownian local times and hence the parameters of the random interlacements appearing in the large $N$ limit do not depend on the degree $d+1$ of the tree. Indeed, we have $\beta=1$ for any $d \geq 1$. The above calculation shows that this is an effect of the large number of leaves of $G_N$. This behavior is in contrast to the example of the Euclidean box treated in Theorem~\ref{thm:box}, where the effect of the boundary on the levels of the appearing random interlacements is negligible.
}
\end{remark}

\end{document}